\documentclass[11pt]{amsart}
\usepackage[foot]{amsaddr}
\usepackage{import}
\usepackage[english]{babel}
\usepackage[%pagebackref,
pdftex,hyperfootnotes]
{hyperref}

\usepackage[total={19cm,24cm},top=2.5cm, left=1.5cm]{geometry}
\usepackage{rotating,multirow,pdflscape}
\usepackage{amssymb,latexsym,amsmath,amsthm}
\usepackage{mathrsfs}
\usepackage{mathtools,arydshln,mathdots}
\usepackage{bigints}
\makeatletter
\renewcommand*\env@matrix[1][*\c@MaxMatrixCols c]{%
  \hskip -\arraycolsep
  \let\@ifnextchar\new@ifnextchar
  \array{#1}}
\makeatother

\mathtoolsset{centercolon}
\usepackage[table]{xcolor}
\usepackage[toc,page]{appendix}

\usepackage{tocvsec2}
\usepackage[T1]{fontenc}
\usepackage[scaled]{helvet} % ss
\usepackage{courier} % tt
\usepackage{eulervm} % a better implementation of the euler package (not in gwTeX)
\normalfont

\newtheorem{coro}{{Corollary}}

\newtheorem{defi}{{ Definition}}
\newtheorem{teo}{Theorem}
\newtheorem{pro}{ Proposition }

\newtheorem{rem}{Remark}

\renewcommand{\d}{\operatorname{d}}
\newcommand{\Exp}[1]{\operatorname{e}^{#1}}

\newcommand{\diag}{\operatorname{diag}}

\newcommand{\C}{\mathbb{X}}
\newcommand{\T}{\mathbb{T}}

\newcommand{\R}{\mathbb{R}}

\newcommand{\prodint}[1]{\left\langle{#1}\right\rangle}

\begin{document}
	
	\title[CMV biorthogonal  Laurent polynomials: Christoffel and Geronimus transformations]{%Christoffel and Geronimus transformations for\\
		CMV biorthogonal  Laurent polynomials:\\ Christoffel formulas for Christoffel and Geronimus perturbations}
	
	 \author[G Ariznabarreta]{Gerardo Ariznabarreta$^{1,2}$}
	 \address{Departamento de Física Teórica II (Métodos Matemáticos de la Física), Universidad Complutense de Madrid, Ciudad Universitaria, Plaza de Ciencias 1,  28040 Madrid, Spain}
	 \email{gariznab@ucm.es}
	
	 \author[M Mañas]{Manuel Mañas$^{1}$}
	 % \address{Departamento de Física Teórica II (Métodos Matemáticos de la Física), Universidad Complutense de Madrid, Ciudad Universitaria, Plaza de Ciencias 1, 28040 Madrid, Spain}
	 \email{manuel.manas@ucm.es}
	
	 \author[A Toledano] {Alfredo Toledano}
	 \email{alfrtole@ucm.es}
	
	 \thanks{$^1$Thanks financial support from the Spanish ``Ministerio de Economía y Competitividad" research project [MTM2015-65888-C4-3-P],\emph{ Ortogonalidad, teoría de la aproximación y aplicaciones en física matemática}}
	 \thanks{$^2$Thanks financial support from the Universidad Complutense de Madrid  Program ``Ayudas para Becas y Contratos Complutenses Predoctorales en España 2011"}
	
\begin{abstract}
Quasidefinite sesquilinear forms for Laurent polynomials  in the complex plane  and corresponding CMV biorthogonal Laurent polynomial families are studied.  Bivariate linear functionals encompass large families of orthogonalities like Sobolev and discrete Sobolev  types.  Two possible  Christoffel transformations of these linear functionals are discussed. Either the  linear functionals are multiplied by a Laurent polynomial, or are multiplied by the  complex conjugate of a Laurent polynomial. For the Geronimus transformation,   the linear functional is perturbed in two possible manners as well, by a division by  a Laurent polynomial or by a complex conjugate of a Laurent polynomial, in both cases  the addition of  appropriate masses (linear functionals supported on the zeros of the perturbing Laurent polynomial) is considered.
The connection formulas for the CMV biorthogonal Laurent polynomials, its norms, and Christoffel--Darboux kernels, in  all the four cases, are given.  For the Geronimus transformation, the connection formulas for the second kind functions and mixed Christoffel--Darboux kernels are also given in the two possible cases.
For prepared Laurent polynomials, i.e. of the form $L(z)=L_nz^n+\cdots+L_{-n}z^{-n}$, $L_nL_{-n}\neq 0$, these connection formulas lead to quasideterminantal (quotient of determinants) Christoffel  formulas for all the four transformations, expressing an arbitrary degree  perturbed biorthogonal Laurent polynomial in terms of
$2n$ unperturbed biorthogonal Laurent polynomials, their second kind functions or Christoffel--Darboux kernels and its mixed versions.
Different curves are presented as examples, like the real line, the circle, the Cassini oval and the cardioid.
The unit circle case, given its exceptional properties, is discussed in more detail. In this case, a particularly relevant role is played by the reciprocal polynomial, and
the Christoffel formulas  provide now with two possible ways of expressing  the same perturbed quantities in terms of the original ones, one using only the nonperturbed biorthogonal family of  Laurent polynomials, and the other using  the Christoffel--Darboux kernels and its mixed versions, as well.	
\end{abstract}



 \maketitle	

%\tableofcontents

\section{Introduction}

The study of perturbations of a linear functional $u$  in the  space of polynomials is  an active area of research in the theory of orthogonal polynomials.
When you deal with  positive definite measures, this study provides  information about the Gaussian quadrature rules \cite{Gaut1,Gaut2}.  Christoffel perturbations, $\hat{u}= p(x) u$, where $p(x)$ is a polynomial,
were studied in 1858 by Christoffel   \cite{christoffel} giving explicit formulas relating the corresponding sequences of orthogonal polynomials with respect to two measures.
These are  called Christoffel formulas, and can be considered  a classical result in the theory of orthogonal polynomials which can be found in a number of  textbooks, see for example \cite{Chi,Sze,Gaut2}.
Explicit relations between the corresponding sequences of orthogonal polynomials  have been extensively studied, see \cite{Gaut1}.
Connection formulas between two families of orthogonal polynomials allow to express any polynomial of a given degree $n$ as a linear combination of all polynomials of degree less than or equal to $n$ in the second family. A remarkable  fact regarding  the Christoffel finding is that in  some cases the number of terms does not grow with the degree $n$ but remain constant,  equal to the degree of the perturbing polynomial. See \cite{Gaut1,Gaut2}.

 When the perturbed  functional is $v$ given  by $p(x) v=u,$ for  $p(x)$ a polynomial, we say that we have a Geronimus transformations. Was Geronimus \cite{Geronimus}, studying the results of \cite{Hann} concerning the characterization of classical orthogonal polynomials (Hermite, Laguerre, Jacobi, and Bessel), the
  first who discussed these transformations.   Christoffel type formulas in terms  of the second kind functions  where  found in \cite{Maro}. Let us notice that in \cite{Geronimus} no Christoffel type formula was derived. Despite this fact, we will refer to the Christoffel formulas for Geronimus transformations as  Christoffel--Geronimus formulas.
The problem of  given two functionals $u$ and $v$ such that  $p(x)u= q(x)v,$ where $p(x), q(x)$ are polynomials was  analyzed in \cite{Uva} when $u,v$ are positive definite measures supported on the real line and in \cite{dini} for  linear functionals, see also \cite{Zhe}.   Uvarov  \cite{Uva}  found Christoffel type formulas, and the addition of a finite number of Dirac masses to a linear functional appears in the framework of the spectral analysis of fourth order linear differential operators with polynomial coefficients and with orthogonal polynomials as eigenfunctions.  Geronimus perturbations of degree two of scalar bilinear forms have been very recently treated in \cite{Derevyagin} and in the general case in \cite{DereM}.

For a positive definite functional the orthogonal polynomials in the unit circle $\mathbb T$ or Szeg\H{o} polynomials are  monic polynomials  $P_n$ of degree  $n$ which fulfill  $\int_{\T}P_n(z) z^{-k} \d \mu(z)=0$, for $ k=0,1,\dots,n-1$, \cite{Sze} and also  \cite{Simon1}.
Orthogonal polynomials on the real line with support on  $[-1,1]$ are connected with the  Szeg\H{o} polynomials    \cite{Freud,Berriochoa}.   The extension to the unit circle context of the three-term relations and tridiagonal Jacobi matrices of the real scenario,  require of  Hessenberg matrices and give the Szeg\H{o} recursion relation, which  is expressed in terms of  the reciprocal, or reverse, Szeg\H{o} polynomials $P^*_l(z):=z^l \overline{P_l(\bar z^{-1})}$ and
reflection or Verblunsky  coefficients $\alpha_l:=P_l(0)$,  $\left(\begin{smallmatrix} P_l \\ P_{l}^* \end{smallmatrix}\right)=
\left(\begin{smallmatrix} z & \alpha_l \\ z \bar \alpha_l & 1 \end{smallmatrix}\right)
\left(\begin{smallmatrix} P_{l-1} \\ P_{l-1}^* \end{smallmatrix}\right)$.
Szeg\H{o}'s theorem implies for a nontrivial probability measure $\d\mu$  on $\T$ with
Verblunsky coefficients $\{\alpha_n\}_{n=0}^\infty$ that the corresponding Szeg\H{o}'s polynomials are dense in $L^2(\T,\mu)$ if and
only if $\prod_{n=0}^\infty (1-|\alpha_n)|^2)=0$. For an absolutely continuous probability measure  Kolmogorov's density theorem ensures that density in $L^2(\T,\mu)$ of the  Szeg\H{o} polynomials holds iff   the so called Szeg\H{o}'s condition $\int_{\T}\log (w(\theta)\d\theta=-\infty$ is fulfilled, \cite{Simon-S}.
The studies of  \cite{Jones-3,Jones-4}  about the the strong Stieltjes moment problem constitute one of the seeds for the activity  about orthogonal Laurent polynomials on the real line.
If there is a  solution of the moment problem we can find   Laurent polynomials $\{Q_n\}_{n=0}^\infty$ satisfying   $\int_{\R}x^{-n+j}Q_n(x)\d \mu(x)=0$ for $j=0,\dots,n-1$.
 Laurent polynomials on the unit circle $\T$ where discussed in \cite{Thron} , see also \cite{Barroso-Vera,CMV,Barroso-Daruis,Barroso-Snake} where recursion relations, Favard's theorem, quadrature problems, and Christoffel--Darboux formul{\ae} were considered.
Despite the set of orthogonal Laurent polynomials being dense in $L^2(\T,\mu)$ in general this is not true for the Szeg\H{o} polynomials,   \cite{Bul} and \cite{Barroso-Vera}.  The Szeg\H{o} recursion relation is replaced by a  five-term relation similar to the real line situation. Generic orders in the basis used to span the space of orthogonal Laurent polynomials in the unit circle were discussed in \cite{Barroso-Snake}. The CMV (Cantero--Moral--Velázquez) matrices \cite{CMV}, which constitute the representation of the multiplication operator in terms of the basis of orthonormal Laurent polynomials, where discussed in \cite{Cantero} where the connection with Darboux transformations and their applications to integrable systems has been  analyzed. Regarding the CMV ordering, orthogonal Laurent polynomials and Toda systems, see \S 4.4 of \cite{CM} where discrete Toda flows and their connection with Darboux transformations were discussed. For   a matrix version of  this discussion see \cite{ari}.
As was pointed out in \cite{CMV-Simon} the discovery the CMV ordering goes back to   \cite{watkins}.
In \cite{Garza1} spectral transformations of measures supported on the unit circle with real moments are considered and the connection with spectral transformations of measures supported on the interval $[-1, 1]$ using the Szeg\H{o} transformation is presented.
Then, in \cite{Garza2}, a Geronimus perturbation of a nontrivial positive measure supported on the unit circle is studied  as well as the connection between the associated Hessenberg matrices. Finally, in \cite{Garza3} linear spectral transformations of Hermitian linear functionals using the multiplication by some class of Laurent polynomials are considered and the behavior of the Verblunsky parameters of the perturbed linear functional is found.

Strong links do exist between Hermitian linear functionals in the  space of Laurent polynomials with complex coefficients and  the theory of orthogonal polynomials on the unit circle.
For every  positive definite Hermitian linear functional  there exists a probability measure supported on a  subset of the unit circle  giving an  integral representation for the functional.
 In this context, it was at the early 1990's when two groups found two different kind of Christoffel type formulas. Godoy and Marcellán published two papers,
    in \cite{Godoy1}  extensions of the Christoffel determinantal type formulas were given for the analogue of the Christoffel transformation, with an arbitrary degree polynomial having multiple roots, using the original Szeg\H{o} polynomials and its  Christoffel--Darboux kernels. This paper has a distinctive approach from the classical by Christoffel, which allows to express the Christoffel formulas in terms of the Christoffel--Darboux kernel and one polynomial solely. This, is in principle (as you have the Christoffel--Darboux formula at your disposal) simpler than the classical Christoffel formula, in where one needs to evaluate not only one, but  a number --related  to the degree of the perturbing polynomial-- of orthogonal polynomials.  Let us notice that this technique relies on the orthogonal decomposition of linear spaces, which is lost in the context of biorthogonality with respect to a  sesquilinear form, as we consider in our  present paper.
Then, in \cite{Godoy2}, the Geronimus transformation for OPUC was  discussed in terms of second kind functions. In particular,  in Proposition 3 of that paper a Christoffel--Geronimus  formula for a perturbation of degree 2 is given, where no masses are considered. On the other hand,  in \cite{ismail}, alternative formulas \emph{á la Christoffel}, not based on the Christoffel--Darboux kernel  \cite{Godoy1},  were presented in terms of determinantal expressions of the Szeg\H{o} polynomials and their reverse polynomials, also as Uvarov did in \cite{Uva}, they considered multiplication by rational functions, but no masses at all where discussed in this paper.
 Finally, in \cite{R. W. Ruedemann}, some examples on concrete cases where considered within  the biorthogonal scenario.

The transformations  considered in this work are also known as Darboux transformations \cite{matveev}. Indeed,
in the context of  the  Sturm--Liouville theory, Darboux discussed  in \cite{darboux2} a dimensional simplification of a geometrical transformation in two dimensions founded previously  \cite{moutard} which can be considered, as we called it today, a \emph{Darboux} transformation.
In \cite{Yoon},  in the context of orthogonal polynomials,  Darboux transformations are considered in terms of the factorization of the Jacobi matrices, while  in \cite{gru1,gru2} the  bispectrality and Darboux transformations are discussed.
Let us mention that in the theory  of Differential Geometry \cite{eisenhart} the Christoffel, Geronimus, Uvarov and linear spectral transformations are related to geometrical transformations like  the Laplace, Lévy, adjoint Lévy and the fundamental Jonas transformations. See \cite{dsm} for a discrete version of these geometrical transformations, and its connection with integrable systems. The interested reader may consult the excellent monographs \cite{matveev-book} and \cite{rogers-schief}.

 The framework for  biorthogonality in this paper is that of continuous sesquilinear forms in the space of Laurent  polynomials. We base our discussion on the Schwartz's  \emph{noyau-distribution} \cite{Schwartz1}. These general non Toeplitz scenarios have been considered in the scalar case,  for an orthogonal polynomial approach see \cite{Bueno} while from an integrable systems point of view see \cite{adler,adler-van moerbeke}.  For a linear functional setting for orthogonal polynomials see \cite{Maroni1985espaces,Maroni1988calcul}. The space of Laurent polynomials $\mathbb C[z,z^{-1}]$, with an appropriate topology,  is considered as the space of fundamental functions, in the sense of \cite{gelfand-distribu1,gelfand-distribu2}, and  the corresponding space of generalized functions provides with a linear functional setting for orthogonal polynomials theory. Discrete orthogonality appears when we consider linear functionals with discrete  and infinite support  \cite{Nikiforov1991Discrete}.
 We will consider an arbitrary nondegenerate continuous sesquilinear form given by a generalized kernel $u_{z_1,\bar z_2}$ with a quasidefinite Gram matrix. This scheme not only contains the more usual choices of Gram matrices like those  of Toeplitz  type on the unit circle, or those leading to  discrete orthogonality but also Sobolev orthogonality,  as these first examples correspond to matrices of generalized kernels supported by the diagonal,  \cite{Hormander}, and our scheme is applicable to the general case, with support off the diagonal as happens in the last Sobolev mentioned case. The quasidefinite condition on the Gram matrix ensures for a block Gauss--Borel factorization and, consequently,  for the existence of two biorthogonal families of Laurent polynomials.

Even though we are in a commutative scenario, not like those necessary for the matrix or multivariate orthogonality, we will express our Christoffel formulas in terms of quasideterminants.
This is justified for the simplicity that the formulas adopt an also because they will be ready for a future update to more general situations. We refer the reader to \cite{gelfand}
and \cite{Olver} for two complementary expositions on this subject. 	For our needs, given a $2\times 2$ block matrix $\begin{bsmallmatrix}
	A & B \\
	C & D
	\end{bsmallmatrix}$, where $A\in\mathbb C^{p\times p}$, $B,C^\top\in\mathbb C^p$ and $D\in\mathbb C$, the last quasideterminant  is the Schur complement \cite{zhang}, that in our case is just the quotient of two determinants
$\Theta_{*}\begin{bsmallmatrix}
A & B \\
C & D
\end{bsmallmatrix}= \frac{\begin{vsmallmatrix}
	A & B \\
	C & D
	\end{vsmallmatrix}}{|A|}$.

The scheme  based on the Gauss--Borel factorization problem, used here for  transformations for matrix orthogonal polynomials or non-Abelian 2D Toda lattices,  has been applied  by our group also in the following situations
\begin{enumerate}
	\item In cite \cite{am} and in \cite{ari3} we consider some extensions of the Christoffel--Darboux formula to generalized orthogonal polynomials \cite{adler-van moerbeke} and to multiple orthogonal polynomials, respectively.
	\item Matrix orthogonal polynomials, its Christoffel transformations and the relation with non-Abelian Toda hierarchies were studied in \cite{alvarez2015Christoffel}, and in \cite{alvarez2016Transformation} we extended those results to include the Geronimus, Geronimus--Uvarov and Uvarov transformations.
	\item  Multiple orthogonal polynomials and multicomponent Toda \cite{amu}.
	\item  For matrix orthogonal Laurent polynomials on the unit circle, CMV orderings,  and non-Abelian  lattices on the circle \cite{ari}.
	\item Multivariate orthogonal polynomials in several real variables and corresponding multispectral integrable Toda hierarchy  \cite{MVOPR,ari0}.
Multivariate orthogonal polynomials on the multidimensional unit torus, the  multivariate extension of the CMV ordering and integrable Toda hierarchies \cite{ari1}.
\end{enumerate}

\subsection{Objectives, results,   layout of the paper and perspectives}

In the paper we consider a general  sesquilinear form in the complex plane  determined by a bivariate linear functional, its  biorthogonal Laurent families and its behavior under Christoffel and Geronimus perturbations. The Gauss--Borel factorization of  the Gram matrix, which we assume to be quasidefinite, leads to connection formulas for the biorthogonal Laurent polynomial families, the corresponding second kind functions and the standard and mixed Christoffel--Darboux kernels.  This result allows us for the finding of quasideterminantal Christoffel formulas for the Christoffel transformation as well as for the Geronimus transformation. Let us observe that regarding previous works \cite{Godoy1,Godoy2,ismail}  on the subject  we may say that
\begin{enumerate}
	\item The results of \cite{Godoy1,Godoy2,ismail} concern sesquilinear forms supported on the diagonal (in fact associated with a positive Borel measure). Our scheme allows for more general biorthogonality and therefore includes Sobolev orthogonality and discrete Sobolev orthogonality.
	\item  The interesting version of the Christoffel formula in terms of Christoffel--Darboux kernels given in \cite{Godoy1} can not be extended to bivariate linear functionals supported off diagonal.
	\item Regarding Geronimus or Geronimus--Uvarov transformations the papers \cite{Godoy2, ismail} do not incorporate masses at all. In our paper we include a very general class of masses. %that we reduce to those bivariate linear functional supported on the diagonal, given discrete Sobolev products
\end{enumerate}

We have considered two possible Christoffel and two possible Geronimus transformations, and found the corresponding  Christoffel formulas. These two transformations, when the bivariate linear functional reduces to  an univariate linear functional  supported on the unit circle, can be made to coincide. When this happens we are  lead to two possible Christoffel formulas, giving alternative expressions of the perturbed objects. This is characteristic of the unit circle, and does not happens in other cases, like  the Cassini oval or the cardioid.

The layout of the paper is as follows. We now proceed with a reminder on known results regarding CMV Laurent polynomials and biorthogonality.  Then, in \S 2 we perturb a general quasidefinite sesquilinear form by multiplying the corresponding bivariate linear functional with a Laurent polynomial or the complex conjugate of a Laurent polynomial. Assuming quasidefiniteness of the perturbed sesquilinear forms, we derive then, using the Gauss--Borel factorization, connections formulas for the biorthogonal Laurent polynomials and the Christoffel--Darboux kernels, see Propositions \ref{ConnectionChristoffel} and \ref{ConnectionChristoffelCD}. Then, when the perturbing Laurent polynomials are prepared,
i.e. their larger positive and smaller negative powers are equal, we derive the Christoffel formulas, see  Theorem \ref{Christoffel Formulas}. We use spectral jets because we work in the general setting of generic multiplicities of the zeros of the perturbing polynomials and express all relations in terms of quasideterminants, as they are more compact and ready for future use with matrix or multivariate orthogonalities. We give the expressions for the  two biorthogonal families of Laurent polynomials as well as their norms, and for both type of transformations. Then, we discuss the situation that appears when the bivariate linear functional collapses to a univariate linear functional supported on the unit circle. We see that both types of perturbations can be made equal, by choosing one of the polynomials the reciprocal of the other. This gives us two alternative   quasideterminantal Christoffel formulas for the perturbed biorthogonal polynomials and norms, see Proposition \ref{Christoffel FormulasCircle}. Some comments are made for other supports different from the unit circle.
Section 3 is devoted to the  analysis of the Geronimus transformations, which can be considered as the inversion of the previous Christoffel transformations and, as we are working in a linear functional setting, some masses, which are supported in the zeros of the perturbing Laurent polynomials, can be included.  In the first place we derive connection  formulas for the biorthogonal Laurent polynomials, Cauchy second kind functions, Christoffel--Darboux kernels,  and mixed Christoffel--Darboux kernels, see Propositions \ref{Geronimus Connection Laurent}, \ref{Geronimus Connection Cauchy}, \ref{Geronimus Connection CD}, and \ref{Geronimus Connection mixed}. With this at hand we present Theorem \ref{Christoffel-Geronimus formulas}, where we  the Christoffel--Geronimus formulas for both type of perturbations, i.e., dividing by a prepared Laurent polynomial or dividing by the complex conjugate of a prepared Laurent polynomial and the addition of general masses are given. Moreover, we write, for the first time, this type of expressions including the masses.   As we commented above,  to our best knowledge, this has not been discussed so far  for Christoffel--Geronimus formulas in the unit circle or complex plane scenarios. Then, in Theorem \ref{Christoffel-Geronimus formulas circle general mases} we consider the situation in where the bivariate linear functional is just an univariate linear functional with support over the unit circle. Giving, as for the Christoffel transformations, two alternative forms of expressing the perturbed biorthogonal Laurent polynomials and their norms. The masses are in this case the most general ones and therefore go beyond the standard masses supported on the diagonal, discussed in Proposition \ref{Christoffel-Geronimus formulas circle diagonal mases}, and can be considered as a discrete Sobolev perturbation.  There is an Appendix containing some of the  proofs.

For perspectives and  future work, we want to extend these results to the linear spectral case, where there is a multiplication by a quotient of two prepared Laurent polynomials. Moreover, we have some preliminary results regarding perturbations with not prepared polynomials, that we what to understand better. Finally, a matrix version on these results is needed as well as a discussion for multivariate orthogonal polynomials in the complex plane and in the multidimensional torus.

\subsection{CMV biorthogonal Laurent polynomials}

 \begin{defi} [Sequilinear forms]\label{def:sesquilinear}
	A sesquilinear form     $\prodint{\cdot,\cdot}$  in the ring of Laurent polinomials $\mathbb{C}[z,z^{-1}]$  is a continous map $\prodint{\cdot,\cdot}: \mathbb{C}[z,z^{-1}]\times\mathbb{C}[z,z^{-1}]\longrightarrow \mathbb{C},$
%\begin{align*}
%\begin{array}{cccc}
%\prodint{\cdot,\cdot}: &\mathbb{C}[z,z^{-1}]\times\mathbb{C}[z,z^{-1}]&\longrightarrow &\mathbb{C},\\
%&(L(z), M(z))&\mapsto& \prodint{L(z_1),M(z_2)},
%\end{array}
%\end{align*}
such that for any triple $L(z),M(z),N(z)\in  \mathbb{C}[z,z^{-1}]$ the following conditions are satisfied
\begin{enumerate}
	\item  $\prodint{AL(z_1)+BM(z_1),N(z_2)}=A\prodint{L(z_1),N(z_2)}+B\prodint{M(z_1),N(z_2)}$, $\forall A,B\in\mathbb{C}$,
	\item $\prodint{L(z_1),AM(z_2)+BN(z_2)}=\prodint{L(z_1),M(z_2)}\bar A+\prodint{L(z_1),N(z_2)}\bar B$, $\forall A,B\in\mathbb C$.
\end{enumerate}
\end{defi}

Given the ordered biinfinite  basis $\{ z^{l}\}_{l=-\infty}^\infty$ or the semiinfinite CMV basis   $\{\chi^{(l)}(z)\}_{l=0}^\infty$,  $\chi^{(l)}:=\begin{cases}
z^{l/2}, & \text{$l$ even}\\
z^{-(l+1)/2}, & \text{$l$ odd.}
\end{cases}$ of  $\mathbb C[z,z^{-1}]$ the sesquilinear form is characterized by the corresponding Gram matrix. For example, in the first case we have the bi-infinite Gram matrix
$g=\begin{bsmallmatrix}
g_{0,0 } &g_{0,1}& \dots\\
g_{1,0} & g_{1,1} & \dots\\
\vdots & \vdots
\end{bsmallmatrix}$ with  $
g_{k,l}=\prodint{(z_1)^k ,(z_2)^l }$ and $ k,l\in\mathbb Z$.

\begin{defi}[Laurent polynomial spectrum]
The zero set of a function $L(z)$ in  $\mathbb C^*:=\mathbb C\setminus\{0\}$ will be denoted by  $\sigma(L)$ and said to be the spectrum of $L(z)$.
\end{defi}

In this paper we will consider sesquilinear forms constructed in terms of  bivariate linear functionals with well-defined  support. The space of distributions is the space of generalized functions when the space fundamental functions is the set of complex smooth functions \cite{Schwartz} with compact support  $ \mathcal D^*:=C_0^\infty(\mathbb C^*)$.
The zero set of a distribution $u\in(\mathcal D^*)'$ is the open region  $\Omega\subset \mathbb C^*$ whenever for every  $f(z)$ supported on $\Omega$ we have $\langle u, f\rangle =0$. Its complementary set, which is closed, is the support, $\operatorname{supp} (u)$, of the distribution $u$.   The distributions of compact support, $u\in(\mathcal E^*)'$,
are the generalized functions with fundamental functions
$\mathcal E^*=C^\infty(\mathbb C^*)$.
As  $\mathbb C[z,z^{-1}]\subsetneq \mathcal E^*$ we deduce that $(\mathcal E^*)'\subsetneq (\mathbb C[z,z^{-1}])'\cap \mathcal (\mathcal D^*)'$.  %Argumentos similares conducen a $(\mathcal E)'\subsetneq (\mathbb C[z])'\cap \mathcal (\mathcal D)'$.
For any bivariate linear functional   $u_{z_1,\bar z_2}$ with support $\operatorname{supp} (u_{z_1,\bar z_2})$ its projections in the axis $z_i$ are denoted by
$\operatorname{supp}_{i} (u_{z_1,\bar z_2})$, $i=1,2$.
Sesquilinear forms are constructed in terms of bivariate linear functionals. The space
$\mathbb C[z,z^{-1}]$ , with a suitable topology  is considered  as the space of fundamental or test functions and the generalized functions are the continuous linear functionals on this space.
\begin{defi}\label{def:sesquilineal}
We consider the following sesquilinear forms
$ \prodint{L(z_1), M(z_2)}_{u}=\prodint{u_{ z_1,\bar z_2}, L( z_1)\otimes \overline{ M(z_2)}}$ with $ L(z),M(z)\in\mathbb C  [z,z^{-1}]$.
\end{defi}
 Hence, the following sesquilinear forms $ \prodint{L(z_1), M(z_2)}_{u}=\sum\limits_{0\leq n,m\ll\infty}\int  \frac{\partial^nL}{\partial z^n} ( z_1)\overline{\frac{\partial^{m} M}{\partial z^m} (z_2)}\d\mu^{(m,n)}(z_1,z_2)$,
 for Borel  measures $\mu^{(m,n)}(z_1,z_2)$ in $\mathbb C^{2}$, with at least one of them with infinite support, are included in our considerations.

Notice that in the bi-infinite basis
$\{z^n\}_{n\in\mathbb Z}$ we have the Gram matrix
$g=[g_{n,m}]$ with $ g_{n,m}= \prodint {(z_1)^n, ( z_2)^{m}}_{u}=\prodint{u_{z_1,\bar z_2},z_1\otimes (\bar z_2)^m}$.
A bivariate linear functional  $u_{z_1,\bar z_2}$ is supported on the diagonal $z_1=z_2$ if
\begin{align}\label{diagonal}
\prodint{ L(z_1), M(z_2)}_u=\sum_{0\leq n,m\ll \infty}\prodint{u_z^{(n,m)}, \frac{\partial^nL}{\partial z^n} ( z)\overline{\frac{\partial^{m} M}{\partial z^m} (z)}},
\end{align}
where $u_z^{(n,m)}$ are univariate linear functionals, i.e., we are dealing with a Sobolev sesquilinear form.

A particularly relevant example is
          \begin{align}%\label{unicase}
             \prodint {L(z),M(z)}_u=\prodint{u_z,L(z)\overline{M(z)}}
             \end{align}
that when $\operatorname{supp}u\subset \mathbb T$ gives
$g_{n,m}=\prodint{u_z,z^{n-m}}$,
           which happens to be a  Toeplitz matrix, $g_{n,m}=g_{n+1,m+1}$. We will refer to this case as the Toeplitz case.

Following \cite{CMV,watkins}, we will use the  CMV   basis $\big\{\chi^{(0)},\chi^{(1)},\chi^{(2)},\dots\big\}$ with $\chi^{(l)}(z)=\begin{cases}
z^k, &l=2k,\\
z^{-k-1}, &l=2k+1.
\end{cases}$

\begin{defi}
Let us consider
\begin{align*}
\chi_{1}(z)&:= [1,0,z,0,z^{2},0,\ldots]^\top, &
\chi_{2}(z)&:= [0,1,0,z,0,z^{2},0,\dots]^{\top},\\
\chi_{1}^{*}(z)&:=z^{-1}\chi_{1}(z^{-1})=[z^{-1},0,z^{-2},0,z^{-3},0,\ldots]^\top,&
\chi_{2}^{*}(z)&:=z^{-1}\chi_{2}(z^{-1})= [0,z^{-1},0,z^{-2},0,z^{-3},0,\dots]^{\top}.
\end{align*}
In terms of which we define the CMV sequences
\begin{align*}
\chi(z)&:=\chi_{1}(z)+\chi_{2}^{*}(z)=[1,z^{-1},z,z^{-2},\ldots]^{\top},&
\chi^{*}(z)&:=\chi_{1}^{*}(z)+\chi_{2}(z)=[z^{-1},1,z^{-2},z,z^{-3},z^{2},\dots]^{\top}.
\end{align*}
We also consider the semi-infinite matrix
\begin{align*}\Upsilon&:=\left[
\begin{array}{c|cc|cc|cc|cc|cc}
0 & 0 & 1 & 0 & 0 & 0 & 0 & 0 & 0 & 0 &\cdots  \\\hline
1 & 0 & 0 & 0 & 0 & 0 & 0 & 0 & 0 & 0 &\cdots  \\
0 & 0 & 0 & 0 & 1 & 0 & 0 & 0 & 0 & 0 &\cdots  \\\hline
0 & 1 & 0 & 0 & 0 & 0 & 0 & 0 & 0 & 0 &\cdots  \\
0 & 0 & 0 & 0 & 0 & 0 & 1 & 0 & 0 & 0 &\cdots  \\\hline
0 & 0 & 0 & 1 & 0 & 0 & 0 & 0 & 0 & 0 &\cdots  \\
0 & 0 & 0 & 0 & 0 & 0 & 0 & 0 & 1 & 0 &\cdots  \\\hline
0 & 0 & 0 & 0 & 0 & 1 & 0 & 0 & 0 & 0 &\cdots  \\
0 & 0 & 0 & 0 & 0 & 0 & 0 & 0 & 0 & 0 &\cdots  \\\hline
0 & 0 & 0 & 0 & 0 & 0 & 0 & 1 & 0 & 0 &\cdots  \\
\vdots & \vdots & \vdots & \vdots & \vdots & \vdots &
\vdots & \vdots & \vdots & \vdots &\ddots
\end{array}\right].
\end{align*}
Given a bivariate linear functional $u_{z_1,\bar z_2}$ and the associated sesquilinear form  we consider the corresponding Gram matrix
\begin{align*}
G=\prodint{\chi(z_1),(\chi(z_2))^\top}_u=\prodint{u_{z_1,\bar z_2},\chi(z_1)\otimes \big(\chi(z_2)\big)^\dagger}.
\end{align*}
\end{defi}

For the Toeplitz scenario    this Gram matrix has the following moment matrix form
$G=\prodint{u_z,\chi(z)(\chi(z))^\dagger}$, $\operatorname{supp}(u)\subset\mathbb T$. If  $u_z$ is a real functional

\begin{pro}
The semi-infinite matrix $\Upsilon$, is unitary
$	\Upsilon^{\top}=\Upsilon^{-1}$,
	and has the important sepctral porpertities
$	\Upsilon \chi(z)=z\chi(z)$ and $ \Upsilon^{-1}\chi(z)=z^{-1}\chi(z)$.
\end{pro}

For the truncation of  semi-infinity matrices we will use the following notation
$A=\begin{bsmallmatrix}
A_{0,0} & A_{0,1} &\dots\\
A_{1,0} & A_{1,1} &\dots\\
\vdots & \vdots & \\
\end{bsmallmatrix}$ and $A^{\left[l\right]}:=
\begin{bsmallmatrix}
A_{0,0} & A_{0,1} & \cdots & A_{0,l-1} \\
A_{1,0} & A_{1,1} & \cdots & A_{1,l-1} \\
\vdots &  \vdots&  &  \vdots\\
A_{l-1,0} & A_{l-1,1} & \cdots & A_{l-1,l-1} \\
\end{bsmallmatrix}$
and also, for the corresponding block structure, we will write
$
A=
\left[
\begin{array}{c|c}
	A^{\left[l\right]} &
	A^{\left[l,\geq l\right]}\\
	\hline
	A^{\left[\geq l,l\right]} & A^{\left[\geq l\right]}
	\end{array}
	\right]$.

In this paper we assume that the Gram matrix $G$  is  quasidefinite, i.e.,  all its principal minors are not zero, so that  the following  Gauss--Borel o $LU$ factorization of  $G$ holds
\begin{align}\label{LU}
G=S_{1}^{-1}H(S_{2}^{-1})^{\dag},
\end{align}
where $S_{1}$ and $S_{2}$ are lower unitriangular matrices and $H$ is a diagonal matrix with no zeros at the diagonal.
\begin{defi}Let us introduce the following vectors of Laurent poynomials
\begin{align}\label{P}
\phi_{1}(z)&:=S_{1}\chi(z), &
\phi_{2}(z)&:=S_{2}\chi(z).
\end{align}
\end{defi}
Its components $\phi_1(z)=[\phi_{1,0}(z), \phi_{1,1}(z),\dots]^\top$ and $\phi_2(z)=[\phi_{2,0}(z), \phi_{2,1}(z),\dots]^\top$ are such that
\begin{pro}[Biorthogonal polynomials]
The following biothogonality conditions hold
\begin{align*}
\prodint{\phi_{1,n}(z_1),\phi_{2,m}(z_2)}_u&=\delta_{n,m}H_n, &n,m&\in\{0,1,2,\dots\}
\end{align*}
\end{pro}

\begin{coro}\label{ortogonalidad}
The orthogonality relations are satisfied
 \begin{align*}
 \prodint{ \phi_{1,2k}(z_1),(z_2)^l}_u&=0, &  -k\leq & l \leq k-1, \\
 \prodint{ \phi_{1,2k+1}(z_1),(z_2)^l} _u&=0, &  -k\leq & l \leq k, \\
\prodint{  (z_1)^l,\phi_{2,2k}(z_2) }_u&=0, &  -k\leq & l \leq k-1, \\
\prodint{  (z_1)^l,\phi_{2,2k+1}(z_2) }_u&=0, &  -k\leq & l \leq k.
 \end{align*}
\end{coro}

\begin{defi}
	The Christoffel--Darboux  kernel is
	\begin{align}\label{kernelChristoffeldefinicion}
	K^{[l]}(\bar z_1,z_2):=
	\sum_{k=0}^{l-1}\overline{\phi{_{2,k}}(z_1)}H_{k}^{-1}\phi_{1,k}(z_2)
	=[\phi_{2}(z_1)^{\dagger}]^{[l]}(H^{-1})^{[l]}[\phi_{1}(z_2)]^{[l]}.%
	\end{align}
%	We also introduce the derivatives
%	\begin{align*}
%	K^{[l;n,m]}:=\frac{\partial ^{n+m}K^{[l]}}{\partial \bar z_1^n\partial z_2^m}.
%	\end{align*}
\end{defi}
\begin{pro}
	The Christoffel--Darboux kernel satisfies the ABC theorem:
	$K^{[l]}(\bar z_1,z_2)=(\chi^{[l]}(z_1))^\dagger (G^{[l]})^{-1}\chi^{[l]}(z_2)$.
\end{pro}

\begin{pro}
The  Christoffel--Darboux kernel satisfies the projection   properties
\begin{align*}
\prodint{
	\sum_{j=0}^M f_j \phi_{1,j}(z_1),
	\overline{K^{[l}(\bar z_2,z)}
}_u&=	\sum_{j=0}^{l-1} f_j \phi_{1,j}(z),&
\prodint{K^{[l]}(\bar z,z_1),
	\sum_{j=0}^M f_j \phi_{2,j}(z_2)}_u&=\overline{\sum_{j=0}^{l-1} f_j \phi_{2,j}(z)}.
\end{align*}
\end{pro}

\begin{proof}
If we have an expansion of the form
$\sum_{j=0}^M f_j \phi_{1,j}(z)$
then
\begin{align*}
\prodint{
	\sum_{j=0}^M f_j \phi_{1,j}(z_1),
	\overline{K^{[l]}(\bar z_2,z)}
		}_u&=
	\prodint{
		\sum_{j=0}^M f_j\phi_{1,j}(z_1),
		\sum_{k=0}^{l-1}\phi{_{2,k}}(z_2)\bar H_{k}^{-1}\overline{\phi_{1,k}(z)}
		}_u
%	\\
%	&=	\sum_{j=0}^M	\sum_{k=0}^{l-1} f_j 	\prodint{
%	\phi_{1,j}(z_1),
%	\phi{_{2,k}}(z_2)}_u H_{k}^{-1}\phi_{1,k}(z)
\\&=	\sum_{j=0}^{l-1} f_j \phi_{1,j}(z).
\end{align*}
Analogously,
\begin{align*}
\prodint{K^{[l]}(\bar z,z_1),
	\sum_{j=0}^M f_j \phi_{2,j}(z_2)}_u&=\prodint{\sum_{k=0}^{l-1}\overline{\phi{_{2,k}}(z)}H_{k}^{-1}\phi_{1,k}(z_1),	\sum_{j=0}^M f_j \phi_{2,j}(z_2)}_u
%\\
%&=\sum_{k=0}^{l-1}\sum_{j=0}^MH_{k}^{-1}\bar f_j
%\overline{\phi_{2,k}(z)}
%\prodint{
%	{\phi_{1,k}(z_1)},	 \phi_{2,j}(z_2)
%		}_u
	\\
&=\overline{\sum_{j=0}^{l-1} f_j \phi_{2,j}(z)}.
\end{align*}
\end{proof}
This simply says  that, when acting on the right  $\overline{K^{[l+1]}(\bar z_2,z)}$ projects  over $\Lambda_{l}:=\mathbb C\{\chi^{(k)}(z)\}_{k=0}^{l}$ while when acting on the left  $K^{[l+1]}(\bar z,z_1)$ projects on $\overline{\Lambda_l}$. Notice that
$\Lambda_{2k}=\{1,z^{-1},z,\dots,z^{-k},z^k\}$ and $\Lambda_{2k+1}=\{1,z^{-1},z,\dots,z^k,z^{-k-1}\}$.
\begin{coro}\label{CDproyeccion}
	If $L(z)\in\Lambda_l:=\mathbb C\{\chi^{(k)}(z)\}_{k=0}^{l-1}$  then
$
	\prodint{
		L(z_1),
		\overline{K^{[l]}(\bar z_2,z)}
	}_u=	L(z)$ and $	\prodint{K^{[l]}(\bar z,z_1),
		L(z_2)}_u=\overline{L(z)}$.
\end{coro}

%\begin{coro}
%		The  Christoffel--Darboux satisfy
%		\begin{align*}
%		\prodint{
%			\sum_{j=0}^M f_j \phi_{1,j}(z_1),
%			\overline{K^{[l;0,n]}(\bar z_2,z)}
%		}_u&=	\sum_{j=0}^{l-1} f_j \phi^{(n)}_{1,j}(z),&
%		\prodint{K^{[l;n,0]}(\bar z,z_1),
%			\sum_{j=0}^M f_j \phi_{2,j}(z_2)}_u&=\overline{\sum_{j=0}^{l-1} f_j \phi^{(n)}_{2,j}(z)}.
%		\end{align*}
%	If $L(z)\in\Lambda_l:=\mathbb C\{\chi^{(k)}(z)\}_{k=0}^{l-1}$  then
%	\begin{align}\label{poject der}
%	\prodint{
%		L(z_1),
%		\overline{K^{[l;0,n]}(\bar z_2,z)}
%	}_u&=	L^{(n)}(z),&
%	\prodint{K^{[l;n,0]}(\bar z,z_1),
%	L(z_2)}_u&=\overline{L^{(n)}(z)}.
%	\end{align}
%\end{coro}

\begin{defi} The Gram partial second kind functions are defined by
	\begin{align}\label{alfre15}
	C_{1,1}(z)&:=H(S_{2}^{-1})^{\dagger}\chi_{1}^{*}(z),&
	C_{1,2}(z)&:=H(S_{2}^{-1})^{\dagger}\chi_{2}(z),\\
	(C_{2,1}(z))^\dagger&:=(\chi_1^*(z))^\dagger (S_1)^{-1}H,&
	(C_{2,2}(z))^\dagger&:=(\chi_{2}(z))^\dagger (S_1)^{-1}H.
	\end{align}
\end{defi}

\begin{pro}
The following relations hold true
	\begin{align*}
	C_{1,1}(z)&=\prodint{\phi_1(z_1), \frac{1}{\bar z-z_2}}_u, &|z|&>\sup_{z_2\in\operatorname{supp}_2u}|z_2|,\\ C_{1,2}(z)&=-\prodint{\phi_1(z_1), \frac{1}{\bar z-z_2}}_u,&|z|&<\inf_{z_2\in\operatorname{supp}_2u}|z_2|,\\
	(	C_{2,1}(z))^\dagger&=\prodint{\frac{1}{\bar z-z_1}, (\phi_2(z_2))^\top}_u,& |z|&>\sup_{z_1\in\operatorname{supp}_1u}|z_1|,\\ (C_{2,2}(z))^\dagger&=-\prodint{\frac{1}{\bar z-z_1}, (\phi_2(z_2))^\top}_u,&|z|&<\inf_{z_1\in\operatorname{supp}_1u}|z_1|,
	\end{align*}
in the indicated disks, which are  the domains of definition of these functions.
\end{pro}

\begin{proof}
We can write
	\begin{align*}
C_{1,1}(z)&=S_1G{\chi}{_1^{*}}(z)=\prodint{\phi_1(z_1), (\chi(z_2))^\top}_u\chi_1^*(z),\\ C_{1,2}(z)&=S_1G{\chi}{_2}(z)=\prodint{\phi_1(z_1), (\chi(z_2))^\top}_u\chi_2(z),\\
(C_{2,1}(z))^\dagger&=(\chi_1^*(z))^\dagger G (S_2)^\dagger=(\chi_1^*(z))^\dagger\prodint{\chi(z_1), (\phi_2(z_2))^\top}_u,\\ (C_{2,2}(z))^\dagger&=(\chi_2(z))^\dagger G (S_2)^\dagger=(\chi_2(z))^\dagger\prodint{\chi(z_1), (\phi_2(z_2))^\top}_u,
	\end{align*}
but form the following uniform convergence
	\begin{align*}
	(\chi(z_2))^\top\overline{\chi_1^*(z)}&=\frac{1}{\bar z-z_2}, & |z|&>|z_2|,&
(\chi(z_2))^\top\overline{\chi_2(z)}&=-\frac{1}{\bar z-z_2}, & |z|&<|z_2|,\\
(\chi_1^*(z))^\dagger\chi(z_1)&=\frac{1}{\bar z-z_1}, & |z|&>|z_1|,&
(\chi_2(z))^\dagger\chi(z_1)&=-\frac{1}{\bar z-z_1}, & |z|&<|z_1|,
	\end{align*}
and the continuity of the sesquilinear forms we conclude the result.
\end{proof}
Using this result we extend the domain of definition of the Gram  functions of the second  kind to:
\begin{defi}
The Cauchy second kind functions are given by
	\begin{align*}
	C_{1}(z)&=\prodint{\phi_1(z_1), \frac{1}{\bar z-z_2}}_u
	=\prodint{u_{z_1,\bar z_2},\phi_1(z_1)\otimes \frac{1}{z-\bar z_2}}
	, &z&\not\in\overline{\operatorname{supp}_2(u)},\\
	(	C_{2}(z))^\dagger&=\prodint{\frac{1}{\bar z-z_1}, (\phi_2(z_2))^\top}_u=\prodint{u_{z_1,\bar z_2},\frac{1}{\bar z-z_1}\otimes(\phi_{2}(z_2))^\dagger },& z&\not\in\overline{ \operatorname{supp}_1(u)}.
 	\end{align*}
\end{defi}

\begin{defi}
The mixed Christoffel--Darboux  kernels are
	\begin{align}\label{kernelChristoffeldefinicionphiC}
	K_{C,\phi}^{[l]}(\bar z_1,z_2)&:=\sum_{k=0}^{l-1}\overline{C{_{2,k}}(z_1)}H_{k}^{-1}\phi_{1,k}(z_2)=[C_{2}(z_1)^{\dagger}]^{[l]}(H^{-1})^{[l]}[\phi_{1}(z_2)]^{[l]},& z_1&\not\in \overline{\operatorname{supp}_1(u)},\\
		K_{\phi,C}^{[l]}(\bar z_1,z_2)&:=\sum_{k=0}^{l-1}\overline{\phi{_{2,k}}(z_1)}H_{k}^{-1}C_{1,k}(z_2)=[\phi_{2}(z_1)^{\dagger}]^{[l]}(H^{-1})^{[l]}[C_{1}(z_2)]^{[l]},&z_2&\not\in \overline{\operatorname{supp}_2(u).}\label{kernelChristoffeldefinicionCphi}
	\end{align}
\end{defi}

\begin{pro}
The mixed kernels have the following expressions
	\begin{align*}%\label{kernelChristoffeldefinicionphiC}
	K_{C,\phi}^{[l]}(\bar x_1,x_2)&=
	\prodint{\frac{1}{\bar x_1-z_1}, \overline{K^{[l]}(\bar z_2,x_2)}}_u,&
	K_{\phi,C}^{[l]}(\bar x_1,x_2)&:=\prodint{K^{[l]}(\bar x_1,z_1), \frac{1}{\bar x_2-z_2}}_u.
	\end{align*}
\end{pro}
\begin{proof}
	We recall that $\overline{C_{2,k}(x_1)}=\prodint{\frac{1}{\bar x_1-z_1}, \phi_{2,k}(z_2)}$ for $\bar x_1\not\in \operatorname{supp}_1(u)$,
	 and $C_{1,k}(x_2)=\prodint{\phi_{1,k}(z_1), \frac{1}{\bar x_2-z_2}}_u$  for $x_2\not\in\overline{\operatorname{supp}_2(u)}$, so that
	\begin{align*}%\label{kernelChristoffeldefinicionphiC}
	K_{C,\phi}^{[l]}(\bar x_1,x_2)&=\sum_{k=0}^{l-1}\overline{C{_{2,k}}(x_1)}H_{k}^{-1}\phi_{1,k}(x_2)
	=\prodint{\frac{1}{\bar x_1-z_1}, \overline{\sum_{k=0}^{l-1}\overline{\phi_{2,k}(z_2)}H_{k}^{-1}\phi_{1,k}(x_2)}}_u,\\
	K_{\phi,C}^{[l]}(\bar x_1,x_2)&=\sum_{k=0}^{l-1}\overline{\phi{_{2,k}}(x_1)}H_{k}^{-1}C_{1,k}(x_2)=\sum_{k=0}^{l-1}\overline{\phi{_{2,k}}(x_1)}H_{k}^{-1}\prodint{\phi_{1,k}(z_1), \frac{1}{\bar x_2-z_2}}_u.
	\end{align*}
\end{proof}
Hence, the mixed kernels can be thought as the projections of the Cauchy kernels or, equivalently, the Cauchy transforms of the Christoffel--Darboux kernels.

\begin{rem}
	In what follows, we will consider several types of supports that belong to curves  in the plane. As examples, we consider curves $\gamma:=\{z\in\mathbb C: \bar z=\tau(z)\}$
	for some function $\tau(z)$, here we give some examples
	\begin{enumerate}
		\item  	If $\tau(z)=z$ we get that $\gamma=\mathbb R$ and we recover the theory of biorthogonal Laurent polynomials on the real line.
		\item   When $\tau(z)=z^{-1}$ we have that $\gamma=\mathbb T$.
		This  leads  to biorthogonal polynomials on the unit circle.
		\item
		The Cassini oval is the locus  of points in the complex plane such that the product of the distance to the two foci remains constant.
		Let the foci be $(0,a)$ and $(0,-a)$, where $a$ is a positive number, then the  Cassini oval equation is
		$       |z-a||z+a|=b^2$       where $b$ is a positive number,
		so that
$\bar z^2=a^2+\frac{b^4}{z^2-a^2}$.
		This implies that
	$	\bar z=\tau(z)$ with $\tau(z):=\sqrt{\frac{a^2 z^2+b^4-a^4}{z^2-a^2}}$.
		Here we must take a determination of the square root, so that we choose the branch to be at the semiaxis $(-\infty,0)$. The Cassini oval  has three different shapes depending on $a\lessgtr b$.   If  $b>a$ we have a closed loop, enclosing the two foci, that cuts the $x$-axis at $\pm\sqrt{a^2+b^2}$. In this case the function $\tau(z)$ is analytic but for a branch cut at the imaginary axis of the form  $\left(-\sqrt{\frac{b^4-a^4}{a^2}}\operatorname{i}, \sqrt{\frac{b^4-a^4}{a^2}}\operatorname{i}\right)$. When $b<a$ we have two closed loops each of which contains one focus. Now the function  $\tau(z)$ is analytic but for  a branch cut at  $ \left(-\sqrt{\frac{a^4-b^4}{a^2}}, \sqrt{\frac{a^4-b^4}{a^2}}\right)$. Finally, for $a=b$ we have the Bernoulli  lemniscate, having the shape of an eight with  a double point at the origin, in this case $\tau(z)$ is analytic at $\mathbb C^*$.  %Observe that in all cases  $\tau(z)$ is a local  conformal mapping as in its domain of definition we have $\tau'(z)$ is not zero.
		\item The cardioid is the locus of points in the complex plane that satisfy
		$(z{\bar  {z}}-a^{2})^{2}-4a^{2}(z-a)({\bar  {z}}-a)=0$, for $a>0$, that is
		%     \begin{align*}
		%     z^2\bar z^2-2a^2(3z-2a)\bar z+a^3(4z-3a)=0.
		%     \end{align*}
		%     Thus, we have
		$		\Big(    \bar z-a^2z^{-2}(3z-2a)\Big)^2=4a^3z^{-4}(a-z)^3$.
		%\frac{4a^4(3z-2a)^2-4z^2a^3(4z-3a)}{4z^4}
		Therefore, $	\bar z=\tau(z)$, $\tau(z):=\Big(2 a^{3/2}(a-z)^{3/2}+a^2(3z-2a)\Big)z^{-2}$.
		This is an analytic function in $\mathbb C^*$ but for three branch cuts starting at the cusp $z=a$  of the cardioid given by $\Exp{\operatorname{i}\pi/3}(a,+\infty)$,  $(a,+\infty)$, $\Exp{-\operatorname{i}\pi/3}(a,+\infty)$, and a double pole at the origin $z=0$.
		%The cardioid do  intersect  the branch cuts $z=a+a\Exp{\pm\operatorname{i}\pi/3}$.
	\end{enumerate}
\end{rem}

\section{Christoffel transformations}

We consider two possible types of perturbations of the sesquilinear form, the first is the multiplication by a Laurent polynomial $L^{(1)}(z_1)$ acting on the first variable.
Secondly, we also consider a perturbation but now we multiply by  $\overline{L^{(2)}(z_2)}$,  when the non perturbed bivariate linear functional supports lay on the curve $\bar z_2=\tau(z_2)$, this will be a Laurent polynomial on the variable $\tau(z_2)$.

\begin{defi}
Given two  Laurent polynomials $L^{(a)}(z)=L^{(a)}_{n} z^n+\dots+L^{(a)}_{-m} z^{-m}$, with $L^{(a)}_{n}L^{(a)}_{-m}\neq 0$ and $n,m\in\{0,1,2,\dots\}$,  $a\in\{1,2\}$, let us consider
the following two perturbations of the bivariate linear functional
\begin{align}\label{Chris1}
\hat u^{(1)}_{z_1,\bar z_2}&=L^{(1)}(z_1)u_{z_1,\bar z_2},  \\\label{Chris2}
\hat u^{(2)}_{z_1,\bar z_2}&=u_{z_1,\bar z_2} \overline{L^{(2)}(z_2)}.
\end{align}
\end{defi}

\begin{pro}
Christoffel transformations associated with  the Laurent polynomials $ L^{(1)}(z) $  and $ L ^{(2)}(z) $  imply for the corresponding Gram matrices
	\begin{align}
\label{gChristoffelNice}
\hat G^{(1)}&=L^{(1)}(\Upsilon)G, & \hat G^{(2)}&=G \big(L^{(2)}(\Upsilon)\big)^\dagger.
\end{align}
\end{pro}
\begin{proof}
Just follow the following chains of equalities
\begin{align*}
\hat G^{(1)}&=\prodint{\hat u^{(1)}_{z_1,\bar z_2},\chi(z_1)\otimes\big( \chi(z_2)\big)^\dagger}\\
&=\prodint{ u_{z_1,\bar z_2},L^{(1)}(z_1)\chi(z_1)\otimes \big( \chi(z_2)\big)^\dagger}
\\
&=L^{(1)}(\Upsilon)\prodint{ u_{z_1,\bar z_2},\chi(z_1)\otimes \big( \chi(z_2)\big)^\dagger},
\\
\hat G^{(2)}&=\prodint{\hat u^{(2)}_{z_1,\bar z_2},\chi(z_1)\otimes \big( \chi(z_2)\big)^\dagger}\\
&=\prodint{ u_{z_1,\bar z_2},\chi(z_1)\otimes\big( \chi(z_2)\big)^\dagger\, \overline{L^{(2)}(z_2)}}
\\
&=\prodint{ u_{z_1,\bar z_2},\chi(z_1)\otimes \big( \chi(z_2)\big)^\dagger}  \big(L^{(2)}(\Upsilon)\big)^\dagger.
\end{align*}
\end{proof}

We assume that the two   Gram matrices are quasidefinite,  so that   the  Gauss--Borel factorization is ensured in both cases,  i. e.
\begin{align}\label{eq:quasidefGramCris}
\hat G^{(1)}&=\big(\hat S_1^{(1)}\big)^{-1}\hat H^{(1)}\big(\hat S_2^{(1)}\big)^{-\dagger}, &
\hat G^{(2)}&=\big(\hat S_1^{(2)}\big)^{-1}\hat H^{(2)}\big(\hat S_2^{(2)}\big)^{-\dagger}.
\end{align}

\subsection{Connection formulas}

\begin{defi}[Christoffel connectors]
Let us  consider a bivariate linear functional and the two Christoffel transformations \eqref{Chris1} and \eqref{Chris2} associated with the Laurent polynomials $ L^{(a)} (z) $, $a\in\{1,2\}$. The  connectors are the following semi-infinite matrices
	\begin{align*}
	\omega^{(1)}_1&=\hat S^{(1)}_1 L^{(1)}(\Upsilon)
(S_1)^{-1}, & 	\omega^{(1)}_2&=S_2\big(\hat S_2^{(1)}\big)^{-1},\\
\omega^{(2)}_1&=S_1\big(\hat S_1^{(2)}\big)^{-1}, &	\omega^{(2)}_2&=\hat S^{(2)}_2 L^{(2)}(\Upsilon)
	(S_2)^{-1}.
\end{align*}
\end{defi}
\begin{pro}The connectors fulfill the  ligatures	$\hat H^{(1)}\big(\omega_2^{(1)}\big)^\dagger =\omega^{(1)}_1H$ and $\omega_1^{(2)}\hat H^{(2)}=H\big(\omega_2^{(2)}\big)^\dagger$.
\end{pro}
\begin{proof}
From \eqref{gChristoffelNice} y \eqref{eq:quasidefGramCris} we get
\begin{align*}
\big(\hat S_1^{(1)}\big)^{-1}\hat H^{(1)}\big(\hat S_2^{(1)}\big)^{-\dagger}&=L^{(1)}(\Upsilon)(S_1)^{-1}H(S_2)^{-\dagger},&
\big(\hat S_1^{(2)}\big)^{-1}\hat H^{(2)}\big(\hat S_2^{(2)}\big)^{-\dagger}&=(S_1)^{-1}H(S_2)^{-\dagger}L^{(2)}(\Upsilon),
\end{align*}
and, consequently,  we conclude
\begin{align*}
\hat H^{(1)}\big(\hat S_2^{(1)}\big)^{-\dagger}(S_2)^{\dagger}&=\hat S_1^{(1)}L^{(1)}(\Upsilon)(S_1)^{-1}H,&
S_1\big(\hat S_1^{(2)}\big)^{-1}\hat H^{(2)}&=H(S_2)^{-\dagger}\big(L^{(2)}(\Upsilon)\big)^\dagger\big(\hat S_2^{(2)}\big)^{\dagger}.
\end{align*}
\end{proof}
%\begin{defi}
%	We will use the notation $\underline{n}:=\max(n,m)$.
%\end{defi}

%Recall now  that the  matrices $ S $ are lower unitriangular to get
\begin{pro}\label{ConnectionChristoffel}	We will use the notation $\underline{n}:=\max(n,m)$.
\begin{enumerate}
	\item The connector $ \omega ^ {(1)} _ 1 $ is an upper triangular matrix with only $ 2\underline{n} + 1 $ nonzero upper diagonals.
%	In addition $ (\omega ^ {(1)} _ 1) _ {k, k} = \hat H ^ {(1)} _ k (H_k) ^ {- 1} $.
\item The connector $ \omega ^ {(1)} _ 2 $ is a lower unitriangular matrix with only $ 2\underline{n}+ 1 $ nonzero lower diagonals.
\item The connector $ \omega ^ {(2)} _ 1 $ is a lower unitriangular matrix with only $ 2\underline{n} + 1 $ nonzero lower diagonals.
\item The connector $ \omega ^ {(2)} _ 2 $ is an upper triangular matrix with only $ 2\underline{n}+ 1 $ nonzero upper diagonals.		
%Moreover, $(\bar\omega^{(2)}_2)_{k,k}=\hat H^{(2)}_k(H_2)^{-1}$. %y $(\omega^{(1)}_1)_{2k+1,2k+1+2m}=\beta$.
\end{enumerate}\end{pro}
The connectors are banded semi-infinite  matrices that give the relationship between  the perturbed and the original biorthogonal families.
\begin {pro}[Connection formulas for the CMV biorthogonal Laurent polynomials] \label {pro:conexión}
We have the following  connection formulas
\begin{align*}
\omega^{(1)}_1 \phi_1(z) &= L^{(1)}(z)\hat\phi_1^{(1)}(z), & \omega^{(1)}_2\hat \phi^{(1)}_2(z)&=\phi_2(z),\\
\omega_1^{(2)}\hat \phi_1^{(2)}(z)&=\phi_1(z), & \omega^{(2)}_2\phi_2(z)&=L^{(2)}(z)\hat \phi^{(2)}_2(z).
\end{align*}
\end{pro}

%\subsection{Connection formulas for the  Christoffel--Darboux kernels}
\begin{pro}[Connection formulas for the Christoffel--Darboux kernels]\label{ConnectionChristoffelCD}
For $ l \geq 2\underline{n} $,  the perturbed and original Christoffel--Darboux kernels satisfy
\begin{align}
\label{kernelChristoffelrelacion}&\begin{multlined}
K^{[l]}(\bar z_1,z_2)=L^{(1)}(z_2)(\hat{K}^{(1)})^{[l]}(\bar z_1,z_2)-
\overline{\left[\hat\phi_{2,l-2\underline{n}}^{(1)}(z_1),\cdots,\hat{\phi}^{(1)}_{2,l-1}(z_1)\right]}\diag(
	(\hat{H}^{(1)}_{l-2\underline{n}})^{-1},\dots, (\hat H^{(1)}_{l-1})^{-1} )\\
\times\begin{bmatrix}
	(\omega^{(1)}_1)_{l-2\underline{n},l} &0&0&\dots  & 0 \\
	(\omega^{(1)}_1)_{l-2\underline{n}+1,l} &	(\omega^{(1)}_1)_{l-2\underline{n}+1,l+1} &0&\dots  & 0 \\
(\omega^{(1)}_1)_{l-2\underline{n}+2,l} &(\omega^{(1)}_1)_{l-2\underline{n}+2,l+1} &(\omega^{(1)}_1)_{l-2\underline{n}+2,l+2}&\ddots  & 0 \\
\vdots & && \ddots & \\
 (\omega^{(1)}_1)_{l-1,l} &(\omega^{(1)}_1)_{l-1,l+1} &(\omega^{(1)}_1)_{l-1,l+2} &\dots &(\omega^{(1)}_1)_{l-1,l-1+2\underline{n}}
	\end{bmatrix}
\begin{bmatrix}
	\phi_{1,l}(z_2) \\
	 \vdots \\
   \phi_{1,l-1+2\underline{n}}(z_2)\\
	\end{bmatrix},
\end{multlined}	
\\&
\label{kernelChristoffelrelacion2}\begin{multlined}
\overline{K^{[l]}(\bar z_1,z_2)}=L^{(2)}(z_1)\overline{(\hat{K}^{(2)})^{[l]}(\bar z_1,z_2)}-
\overline{\left[\hat\phi_{1,l-2\underline{n}}^{(2)}( z_2),\cdots,\hat\phi^{(2)}_{1,l-1}( z_2)\right]}\overline{\diag(
(\hat{H}^{(2)}_{l-2\underline{n}})^{-1},\dots, (\hat H^{(2)}_{l-1})^{-1} )}\\
\times\begin{bmatrix}
(\omega^{(2)}_2)_{l-2\underline{n},l} &0&0&\dots  & 0 \\
(\omega^{(2)}_2)_{l-2\underline{n}+1,l} &	(\omega^{(2)}_2)_{l-2\underline{n}+1,l+1} &0&\dots  & 0 \\
(\omega^{(2)}_2)_{l-2\underline{n}+2,l} &(\omega^{(2)}_2)_{l-2\underline{n}+2,l+1} &(\omega^{(2)}_2)_{l-2\underline{n}+2,l+2}&\ddots  & 0 \\
\vdots & && \ddots & \\
(\omega^{(2)}_2)_{l-1,l} &(\omega^{(2)}_2)_{l-1,l+1} &(\omega^{(2)}_2)_{l-1,l+2} &\dots &(\omega^{(2)}_2)_{l-1,l-1+2\underline{n}}
\end{bmatrix}
\begin{bmatrix}
\phi_{2,l}(z_1) \\
\vdots \\
\phi_{2,l-1+2\underline{n}}(z_1)\\
\end{bmatrix}.
\end{multlined}	
\end{align}

\end{pro}
\begin{proof} See Appendix. 
\end{proof}
\subsection{Christoffel formulas}

We now require the perturbing polynomials to be of a particular form.  Féjer \cite{fejer} and Riesz \cite{riesz} found a representation  for nonnegative  trigonometric polynomials.  Any nonnegative trigonometric polynomial
$f(\theta)=a_0+\sum_{k=1}^n(a_k\cos(k\theta)+b_k\sin(k\theta))$
has the form $f(\theta)=|p(z)|^2$ where $p(z)=\sum\limits_{l=0}^n p_lz^l$ and $z=\Exp{\operatorname{i}\theta}$.
Following Grenader and Szeg\H{o} , see 1.12 in \cite{Grenader}, this is equivalent to write
$f(\theta)=z^{-n} P(z)$ with $z=\Exp{\operatorname{i}\theta}$
and $P(z)\in\mathbb C[z]$ is a polynomial with $\deg P(z)=2n$ such that
$ P(z)=P^*(z)$,
here we have used the Szeg\H{o} reciprocal polynomial $ P^*(z):= z^{2n}\bar P(z^{-1})$ of $P(z)$, fulfilling $z^{-n}P(z)=|P(z)|$ for $z\in\mathbb T$.  Observe that  for $z\in\mathbb C^*$ the function $L(z)=z^{-n} P(z)$ is not any more a trigonometric polynomial but a Laurent polynomial.  These Laurent polynomials are precisely those that as perturbations allow us to find Christoffel formulas.
Given a Laurent polynomial its reciprocal is given by $L_*(z):=\bar L(z^{-1})$, thus for $L(z)=z^{-n}P(z)$ we have $L_*(z)= z^n\bar P (z^{-1})=z^{-n}P^*(z)$ and if $P(z)=P^*(z)$ we find $L_*(z)=L(z)$; the positivity condition reads: $f(\theta)=L(z)$ with $L(z)$ a Laurent polynomial with $L(z)=L_*(z)$ and $L(z)=|L(z)|$ for $z\in\mathbb T$.

\begin{defi}[Prepared  Laurent polynomials] For every $2n$-degree polynomial $P(z)=P_{2n}z^{2n}+\dots+P_0\in\mathbb C[z]$ with $P_0\neq 0$,  its Féjer--Riesz corresponding  Laurent polynomial is given by
	\begin{align}\label{polinomio perturbador}
 L(z)&=z^{-n}P(z)=L_n z^{n}+\cdots + L_{-n} z^{-n}, & L_n&=P_{2n},&L_{-n}&=P_0.
\end{align}
We say that a  Laurent polynomial is prepared whenever it is the   Féjer--Riesz corresponding  Laurent polynomials  of  an even degree polynomial non vanishing at the origin.
%A prepared Laurent polynomial is said to be $\tau$-ready if $\overline{\sigma(L)}\subset \tau(D)$.
%In this case $\tau^{-1}(\bar{\zeta}_i)$, $\zeta_i\in\sigma(L)$, means that we take one of the possible  preimages, there is at least one of them.
\end{defi}

For the consideration of arbitrary multiplicities of the zeros of the perturbing polynomials we need of
\begin{defi}[Spectral jets]
	Given a Laurent polynomial $L(z)$ with zeros and multiplicities  $ \{\zeta _ {i}, m_ {i} \} _ {i = 1} ^ {d} $ we introduce the spectral jet of a  function $ f (z)$ along $L(z)$ as follows:
	\begin{align*}
	\mathcal J^L_f
	:=\left[f(\zeta_{1}),f'(\zeta_{1}),\cdots,\frac{f^{(m_{1}-1)}(\zeta_{1})}{(m_{1}-1)!},\cdots,f(\zeta_{d}),f'(\zeta_{d})\cdots,\frac{f^{(m_{d}-1)}(\zeta_{d})}{(m_{d}-1)!}\right]\in\mathbb C^{2m}.
	\end{align*}
\end{defi}
The spectral jet $ \mathcal J ^ { L^{(1)}}_ {K ^ {[l]}} (\bar z_1) $ of the kernel $ K ^ {[l]} (\bar z_1, z_2) $ is taken with respect to the variable $z_2 $ leaving the variable $ \bar z_1 $ as a fixed parameter. However,  the  spectral jet $ \mathcal J ^ { L^{(2)}} _ {\overline {K ^ {[l]}}} (\bar z_2) $ is taken with respect to the variable $ z_1 $  leaving the variable $ \bar z_2$ as fixed parameter. 

\begin{rem}
	For  prepared Laurent polynomials we have
that $ (\omega^{(1)}_1)_{k,k+2n} = L^{(1)}_{(-1)^kn}$ and $(\omega^{(2)}_2)_{k,k+2n} =  L^{(2)}_{(-1)^{k}n}$.		

\end{rem}\begin{teo}[Christoffel formulas]\label{Christoffel Formulas}
Given a prepared Laurent polynomial $L^{(1)}(z)$, let us suppose that
 $\hat \tau^{(1)}_{l}:=\begin{vsmallmatrix}
		\mathcal J^{L^{(1)}}_{\phi_{1,l}}\\
		\vdots\\
		\mathcal J^{L^{(1)}}_{\phi_{1,l+2n-1}}
		\end{vsmallmatrix}\neq 0$ and $l\geq 2m$.  Then,  given the Christoffel transformation \eqref{Chris1}, the perturbed biortogonal polynomials are expressed in terms of the original ones according to the following  formulas
\begin{align}\label{Chris11}
\hat{\phi}^{(1)}_{1,l}(z)&=\frac{L^{(1)}_{(-1)^ln}}{L^{(1)}(z)}
        \Theta_{*}
	\begin{bmatrix}
	\mathcal J^{L^{(1)}}_{\phi_{1,l}}& 	\phi_{1,l}(z)\\
	\vdots& \vdots\\
%	\mathcal J^{L^{(1)}}_{\phi_{1,l+2n-1}}&	\phi_{1,l+2n-1}(z)\\
	\mathcal J^{L^{(1)}}
_{\phi_{1,l+2n}}& \phi_{1,l+2n}(z)	
	\end{bmatrix}
	=\frac{L^{(1)}_{(-1)^ln}}{L(z)}\frac{1}{\hat \tau^{(1)}_{l}}
	\begin{vmatrix}
	\mathcal J^{L^{(1)}}_{\phi_{1,l}}&	\phi_{1,l}(z)\\
	\vdots&	\vdots\\
%	\mathcal J^{L^{(1)}}_{\phi_{1,l+2n-1}}&	\phi_{1,l+2n-1}(z)\\
	\mathcal J^{L^{(1)}}
_{\phi_{1,l+2n}}& \phi_{1,l+2n}(z)
	\end{vmatrix},	\end{align}
	\begin{align}\label{Chris10}
\hat{H}^{(1)}_{l}&=
L^{(1)}_{(-1)^ln}H_l\Theta_{*}
		\begin{bmatrix}
		\mathcal J^{L^{(1)}}_{\phi_{1,l}}&	 1\\\mathcal J^{L^{(1)}}_{\phi_{1,l+1}}&0\\
		\vdots&\vdots\\
	%	\mathcal J^{L^{(1)}}_{\phi_{1,l+2n-1}}&0\\
	\mathcal J^{L^{(1)}}_{\phi_{1,l+2n}}& 0	
		\end{bmatrix}
=L^{(1)}_{(-1)^ln}H_l
\frac{\hat \tau^{(1)}_{l+1}}{\hat \tau^{(1)}_{l}},
	\end{align}
	\begin{align}
	\label{Chris12}
	\overline{\hat \phi^{(1)}_{2,l}(z)}&=\frac{\hat{H}^{(1)}_{l}}{L^{(1)}_{(-1)^ln}}
	\Theta_*\begin{bmatrix}
			\mathcal J^{L^{(1)}}_{\phi_{1,l+1}}&0\\
			\vdots&\vdots\\\mathcal J^{L^{(1)}}_{\phi_{1,l+2n-1}}&0\\
			\mathcal J^{L^{(1)}}_{\phi_{1,l+2n}}&1\\
			\mathcal J^{L^{(1)}}_{K^{[l+1]}}(\bar z) & 0	
			\end{bmatrix}
		=-H_{l}\frac{1}{\hat \tau^{(1)}_{l+1}}
		\begin{vmatrix}
		\mathcal J^{L^{(1)}}_{\phi_{1,l+1}}\\
		\vdots\\\mathcal J^{L^{(1)}}_{\phi_{1,l+2n-1}}\\
		\mathcal J^{L^{(1)}}_{K^{[l+1]}}(\bar z)
		\end{vmatrix}.
\end{align}
Given a prepared Laurent polynomial $L^{(2)}(z)$, for the Christoffel transformation  \eqref{Chris2}, whenever $\hat \tau^{(2)}_{l}:=\begin{vsmallmatrix}
\mathcal J^{L^{(2)}}_{\phi_{2,l}}\\
\vdots\\
\mathcal J^{L^{(2)}}_{\phi_{2,l+2n-1}}
\end{vsmallmatrix}\neq 0$, we find
\begin{align}\label{Chris22}
\hat{\phi}^{(2)}_{2,l}(z)&=\frac{ L^{(2)}_{(-1)^{l}n}}{L^{(2)}(z) }
\Theta_{*}
\begin{bmatrix}
\mathcal J^{L^{(2)}}_{\phi_{2,l}}&\phi_{2,l}(z)\\
\vdots& \vdots\\
%\mathcal J^{L^{(2)}}_{\phi_{2,l+2n-1}}& \phi_{2,l+2n-1}(z)\\
\mathcal J^{L^{(2)}}_{\phi_{2,l+2n}}& \phi_{2,l+2n}(z)
\end{bmatrix}
=
\frac{ L^{(2)}_{(-1)^{l}n}}{L^{(2)}(z) }\frac{1}{\hat \tau^{(2)}_{l}}
\begin{vmatrix}
\mathcal J^{L^{(2)}}_{\phi_{2,l}}&\phi_{2,l}(z)\\
\vdots&\vdots\\
%\mathcal J^{L^{(2)}}_{\phi_{2,l+2n-1}}&\phi_{2,l+2m-1}(z)\\
\mathcal J^{L^{(2)}}_{\phi_{2,l+2n}}& \phi_{2,l+2n}(z)
\end{vmatrix},
\end{align}
\begin{align}\label{Chris20}
\bar{\hat{H}}^{(2)}_{l}&=
 L^{(2)}_{(-1)^{l}n}\bar H_l\Theta_{*}
\begin{bmatrix}
\mathcal J^{L^{(2)}}_{\phi_{2,l}}& 1\\\mathcal J^{L^{(2)}}_{\phi_{2,l+1}}& 0\\
\vdots&\vdots\\
%\mathcal J^{L^{(2)}}_{\phi_{2,l+2n-1}}&0\\
\mathcal J^{L^{(2)}}_{\phi_{2,l+2n}}& 0
\end{bmatrix}
= L^{(2)}_{(-1)^{l}n}\bar H_l\frac{\hat \tau^{(2)}_{l+1}}{\hat \tau^{(2)}_{l}},\end{align}
\begin{align}\label{Chris21}
\overline{\hat \phi^{(2)}_{1,l}(z)}&=\frac{\bar{\hat{H}}^{(2)}_{l}}{ L^{(2)}_{(-1)^{l}n}}
\Theta_{*}
\begin{bmatrix}
\mathcal J^{L^{(2)}}_{\phi_{2,l+1}}&0\\
\vdots&\vdots\\\mathcal J^{L^{(2)}}_{\phi_{2,l+2n-1}}&0\\
\mathcal J^{L^{(2)}}_{\phi_{2,l+2n}}&1\\
\mathcal J^{L^{(2)}}_{\overline{K^{[l+1]}}}(\bar z) & 0
\end{bmatrix}
=-\bar H_l\frac{1}{\hat \tau^{(2)}_{l+1}}
\begin{vmatrix}
\mathcal J^{L^{(2)}}_{\phi_{2,l+1}}\\
\vdots\\\mathcal J^{L^{(2)}}_{\phi_{2,l+2n-1}}\\
\mathcal J^{L^{(2)}}_{\overline{K^{[l+1]}}}(\bar z)
\end{vmatrix}.
\end{align}
\end{teo}
\begin{proof} See Appendix.
\end{proof}

	We have two possible Christoffel transformations and corresponding Christoffel formulas. For the simplest diagonal situation
$\prodint{L(z),M(z)}_u=\prodint{u_z, L(z)\overline{M(z)}}$
we have these two possibilities, $\hat u^{(1)}_z=L^{(1)}(z)u_z$ or $ \hat u^{(2)}_z=\overline{L^{(2)}(z)}u_z$. If $\operatorname{supp}u_z\subset \gamma$, where the curve $\gamma=\{z
\in\mathbb C^*:\bar z=\tau(z)
\}$, we have that the perturbation of type 2 can be described as
$	 \hat u^{(2)}_z=\overline{ L^{(2)}}(\tau(z))u_z$.
For example, if we consider the unit circle $\gamma=\mathbb T$,  with $\tau(z)=z^{-1}$, this Christoffel transformation can be thought, as
$	\hat u^{(2)}_z=L^{(2)}_{*}(z)u_z$.
 Hence, taking $L^{(2)}(z)=L^{(1)}_{*}(z)$ we get that $\hat u^{(1)}_z=\hat u^{(2)}_z=L^{(1)}(z)u_z$.
In such a case, we have $\hat G^{(1)}=\hat G^{(2)}=:\hat G$, and
	$\hat S^{(1)}_1=\hat S_1^{(2)}=:\hat S_1$, $\hat S^{(1)}_2=\hat S_2^{(2)}=:\hat S_2$, and $\hat H^{(1)}=\hat H^{(2)}=:\hat H$,
so that $\hat \phi^{(1)}_1(z)=\hat \phi_1^{(2)}=:\hat \phi_1(z)$ y $\hat \phi^{(1)}_2(z)=\hat \phi_2^{(2)}=:\hat \phi_2(z)$
and we find  two alternative forms for the Christoffel formulas.
Realize also that if $P(z)=P_{2n}\prod\limits_{j=1}^d (z-\zeta_j)^{m_j}$ then
$L(z)=L_n z^{-n}\prod_{i=1}^{d}(z-\zeta_{i})^{m_{i}}$, $ \zeta_i\in\mathbb{C}\setminus \{0\}$,
where  $L_{-n}=L_n\prod_{i=1}^d\zeta_i^{m_i}$ and $2n=\sum_{i=1}^dm_i$.
	A prepared Laurent polynomial $L(z)=L_n z^{n}+\cdots + L_{-n} z^{-n}$ is such that
$	L(z)\Lambda_k\subset \Lambda_{k+2n}$ and $L(z)\Lambda_k\not\subset \Lambda_{k+2n-1 }$.

\subsection{Reductions to univariate linear functionals supported on the unit circle}
\begin{pro}\label{Christoffel FormulasCircle}
For a sesquilinear form given by an univariate linear functional and supported on the unit circle $\mathbb T$,
for 	$l\geq 2n$,   we have the following alternative Christoffel formulas, whenever the involved quasideterminants  make any sense,
\begin{align*}
\hat{\phi}_{1,l}(z)&=\frac{L_{(-1)^ln}}{L(z)}
\Theta_{*}
\begin{bmatrix}
\mathcal J^L_{\phi_{1,l}}&\phi_{1,l}(z)\\
\vdots&\vdots\\
%\mathcal J^L_{\phi_{1,l+2n-1}}&\phi_{1,l+2n-1}(z)\\
\mathcal J^L_{\phi_{1,l+2n}}& \phi_{1,l+2n}(z)
\end{bmatrix}
=\overline{\frac{\bar{\hat{H}}_{l}}{\bar  L_{(-1)^{l+1}n}}
\Theta_{*}
\begin{bmatrix}
\mathcal J^{L_*}_{\phi_{2,l+1}}&0\\
\vdots\\\mathcal J^{L_*}_{\phi_{2,l+2n-1}}&0\\
\mathcal J^{L_*}_{\phi_{2,l+2n}}&1\\
\mathcal J^{L_*}_{\overline{K^{[l+1]}}}(z) & 0
\end{bmatrix}},\\
\hat{H}_{l}&=
L_{(-1)^ln}H_l\Theta_{*}
\begin{bmatrix}
\mathcal J^L_{\phi_{1,l}}&1\\\mathcal J^L_{\phi_{1,l+1}}&0\\
\vdots&\vdots\\
%\mathcal J^L_{\phi_{1,l+2n-1}}&0
\mathcal J^L_{\phi_{1,l+2n}}& 0
\end{bmatrix}
=\overline{ \bar  L_{(-1)^{l+1}n}\Theta_{*}
\begin{bmatrix}
\mathcal J^{L_*}_{\phi_{2,l}}&1\\\mathcal J^{L_*}_{\phi_{2,l+1}}&0\\
\vdots&\vdots\\
%\mathcal J^{L_*}_{\phi_{2,l+2n-1}}&0\\
\mathcal J^{L_*}_{\phi_{2,l+2n}}& 0
\end{bmatrix}},\\
\overline{\hat \phi_{2,l}(z)}&=\frac{\hat{H}_{l}}{L_{(-1)^ln}}
\Theta_{*}
\begin{bmatrix}
\mathcal J^L_{\phi_{1,l+1}}&0\\
\vdots&\vdots\\\mathcal J^L_{\phi_{1,l+2n-1}}&0\\
\mathcal J^L_{\phi_{1,l+2n}}&1\\
\mathcal J^L_{K^{l+1}}(z) & 0
\end{bmatrix}
=\overline{\frac{ \bar  L_{(-1)^{l+1}n}}{L_*(z)}
\Theta_{*}
\begin{bmatrix}
\mathcal J^{L_*}_{\phi_{2,l}}&\phi_{2,l}(z)\\
\vdots&\vdots\\
%\mathcal J^{L_*}_{\phi_{2,l+2n-1}}&\phi_{2,l+2n-1}(z)\\
\mathcal J^{L_*}_{\phi_{2,l+2n}}& \phi_{2,l+2n}(z)
\end{bmatrix}}.
\end{align*}
\end{pro}
However, this identification does not hold in general, in which case we have two different Christoffel transformations ---multiplying by a Laurent polynomial $L^{(1)}(z)$,  or multiplying by a Laurent polynomial composed with the supporting curve function $\tau(z)$, $\overline{L^{(2)}} \circ \tau (z) $---  for example, multiplying by
$\overline{L^{(2)}}\Big(\sqrt{a^2+\frac{b^4}{z^2-a^2}}\Big )$ for the Cassini oval and $\overline{L^{(2)}} \Big(2 a^{3/2}\sqrt{(a-z)^3}+a^2\frac{3z-2a}{z^2}\Big)$ for the cardioid.
If the original functional is nonnegative and  biorthogonality  simplifies to orthogonality with $\phi_{1,l}(z)=\phi_{2,l}(z)=:\phi_l(z)$  we could consider a Christoffel transformation with a prepared polynomial $L(z)$ which is equal to its reciprocal, $L_*(z)=L(z)$,  and such that $L(z)=|L(z)|$ for $|z|\in\mathbb T$, i.e. a nonnegative trigonometrical polynomial when evaluated over the circle, and find
\begin{align*}
\hat{\phi}_{l}(z)&=\frac{L_{(-1)^ln}}{L(z)}
\Theta_{*}
\begin{bmatrix}
\mathcal J^L_{\phi_{l}}&\phi_{l}(z)\\
\vdots&\vdots\\
%\mathcal J^L_{\phi_{l+2n-1}}&\phi_{l+2n-1}(z)\\
\mathcal J^L_{\phi_{l+2n}}& \phi_{l+2n}(z)
\end{bmatrix}
=\overline{\frac{\bar{\hat{H}}_{l}}{\bar  L_{(-1)^{l+1}n}}
	\Theta_{*}
	\begin{bmatrix}
	\mathcal J^{L}_{\phi_{l+1}}&0\\
	\vdots&\vdots\\\mathcal J^{L}_{\phi_{l+2n-1}}&0\\
	\mathcal J^{L}_{\phi_{l+2n}}&1\\
		\mathcal J^{L}_{\overline{K^{[l+1]}}}(z) & 0
	\end{bmatrix}},\end{align*}
\begin{align*}
\hat{H}_{l}&=
L_{(-1)^ln}\Theta_{*}
\begin{bmatrix}
\mathcal J^L_{\phi_{l}}&H_l\\\mathcal J^L_{\phi_{l+1}}&0\\
\vdots&\vdots\\
%\mathcal J^L_{\phi_{l+2n-1}}&0\\
\mathcal J^L_{\phi_{l+2n}}& 0
\end{bmatrix}
=\overline{ \bar  L_{(-1)^{l+1}n}\Theta_{*}
	\begin{bmatrix}
	\mathcal J^{L}_{\phi_{l}}&\bar H_l\\\mathcal J^L_{\phi_{l+1}}&0\\
	\vdots&\vdots\\
%	\mathcal J^{L}_{\phi_{l+2n-1}}&0\\
		\mathcal J^{L}_{\phi_{l+2n}}& 0
	\end{bmatrix}}.
\end{align*}
Notice that, in terms of $\hat \tau_l:=\begin{vsmallmatrix}
\mathcal J^{L}_{\phi_{l}}\\\mathcal J^L_{\phi_{l+1}}\\
\vdots\\
\mathcal J^{L}_{\phi_{l+2n-1}}
\end{vsmallmatrix}$, the last relation gives $\hat H_l=L_{(-1)^ln}H_l\frac{\tau_{l+1}}{\hat \tau_l}=L_{(-1)^{l+1}n}H_l\frac{\bar {\hat\tau}_{l+1}}{\bar{ \hat\tau}_l}$, which implies
$L_{(-1)^ln}\hat \tau_{l+1}\bar {\hat\tau}_l=L_{(-1)^{l+1}n}\bar {\hat\tau}_{l+1}\hat\tau_l$. We also have
\begin{align*}
\hat{\phi}_{l}(z)&=\frac{L_{(-1)^ln}}{L(z)}
\frac{\begin{vmatrix}
\mathcal J^L_{\phi_{l}} &\phi_{l}(z)\\
\vdots&\vdots\\
%\mathcal J^L_{\phi_{l+2n-1}} &\phi_{l+2n-1}(z)\\
\mathcal J^L_{\phi_{l+2n}}& \phi_{l+2n}(z)
\end{vmatrix}}{\hat\tau_l}
=H_l\frac{\overline{
	\begin{vmatrix}
	\mathcal J^{L}_{\phi_{l+1}}\\
	\vdots\\\mathcal J^{L}_{\phi_{l+2n-1}}\\
	\mathcal J^{L}_{\overline{K^{[l+1]}}}(z)
	\end{vmatrix}}}{\bar{\hat \tau}_l}.\end{align*}

\section{Geronimus transformations}

For this transformation we need a bivariate linear  functional  with  a well-defined support.
\begin{defi}
	Given Laurent polynomials $L^{(a)}(z)=L^{(a)}_{n} z^n+\dots+L^{(a)}_{-m} z^{-m}$, with $L^{(a)}_{n}L^{(a)}_{-m}\neq 0$ and $n,m\in\{0,1,2,\dots\}$,  $a\in\{1,2\}$,  and $\sigma(L^{(1)}(z))\cap\operatorname{supp}_1u=\varnothing$,
	$\overline{\sigma(L^{(2)}(z))}\cap\operatorname{supp}_2u=\varnothing$, we consider two possible families of Geronimus transformations $\check u^{(1)}_{z_1,\bar z_2}$ y   $\check u^{(2)}_{z_1,\bar z_2}$, of the bivariate linear functional $u_{z_1,\bar z_2}$ characterized by
	\begin{align}
\label{Ger1}
	L^{(1)}(z_1)\check u^{(1)}_{z_1,\bar z_2}&=u_{z_1,\bar z_2},  \\
\label{Ger2}	\check u^{(2)}_{z_1,\bar z_2}\overline{L^{(2)}(z_2)}&=u_{z_1,\bar z_2} .
	\end{align}
\end{defi}

Therefore, the perturbed bivariate linear  functionals   are
\begin{align*}
\check u^{(1)}_{z_1,\bar z_2}&=\frac{u_{z_1,\bar z_2}}{L^{(1)}(z_1)}+\sum_{i=1}^{d^{(1)}}\sum_{l=0}^{m^{(1)}_{i}-1}\frac{(-1)^{l}}{l!}\delta^{(l)}\big(z_1-\zeta^{(1)}_{i}\big)
\otimes \overline{(\xi^{(1)}_{i,l})_{ z_2}},  \\
\check u^{(2)}_{z_1,\bar z_2}&=\frac{u_{z_1,\bar z_2}}{\overline{L^{(2)}(z_2)}}+
\sum_{i=1}^{d^{(2)}}\sum_{l=0}^{m^{(2)}_{i}-1}\frac{(-1)^{l}}{l!}(\xi^{(2)}_{i,l})_{z_1}\otimes\overline{\delta^{(l)}\big(z_2- \zeta^{(2)}_{i}\big)},
\end{align*}
where  $(\xi^{(1)}_{i,l})_{z}$ y $ (\xi^{(2)}_{i,l})_{ z}$ are univariate  linear functionals.

\begin{pro}
	Geronimus transformations associated with the  perturbing Laurent polynomials $ L^{(1)}(z) $ and ${L^{(2)}(z)}$ imply for the corresponding Gram matrices
	the following relations
\begin{align}\label{GGeronimus}
L^{(1)}(\Upsilon)	\check  G^{(1)}&=G, & \check G^{(2)}\big(L^{(2)}(\Upsilon)\big)^\dagger&=G.
	\end{align}
\end{pro}
\begin{proof}
It follows from
	\begin{align*}
	G&=\prodint{ u_{z_1,\bar z_2},\chi(z_1)\otimes (\chi(z_2))^\dagger}\\
	&=\prodint{ \check u^{(1)}_{z_1,\bar z_2},L^{(1)}(z_1)\chi(z_1)\otimes (\chi(z_2))^\dagger}
	\\
	&=L^{(1)}(\Upsilon)\prodint{ \check u^{(1)}_{z_1,\bar z_2},\chi(z_1)\otimes (\chi(z_2))^\dagger},
	\\
G&=\prodint{\check u^{(2)}_{z_1,\bar z_2},\chi(z_1)\otimes (\chi(z_2))^\dagger}\\
	&=\prodint{ u_{z_1,\bar z_2},\chi(z_1)\otimes (\chi(z_2))^\dagger \overline{L^{(2)}(z_2)}}
	\\
	&=\prodint{ u^{(2)}_{z_1,\bar z_2},\chi(z_1)\otimes (\chi(z_2))^\dagger} \big(L^{(2)}(\Upsilon)\big)^\dagger.
	\end{align*}
\end{proof}

When both Gram matrices  are quasidefinite, the following  Gauss--Borel factorizations do exist
\begin{align}\label{eq:quasidefGramGer}
\check G^{(1)}&=\big(\check S_1^{(1)}\big)^{-1}\check H^{(1)}\big(\check S_2^{(1)}\big)^{-\dagger}, &
\check  G^{(2)}&=\big(\check S_1^{(2)}\big)^{-1}\check H^{(2)}\big(\check S_2^{(2)}\big)^{-\dagger}.
\end{align}

\subsection{Connection formulas}

\begin{defi}[Geronimus connectors] Consider a bivariate linear functional and Geronimus transformations associated with  two Laurent polynomial $ L^{(a)} (z) $, $a=1,2$. We define the following connectors
	\begin{align*}
	\Omega^{(1)}_1&= S_1 L^{(1)}(\Upsilon)
	(\check S_1^{(1)})^{-1}, & 	\Omega^{(1)}_2&=\check S_2^{(1)}( S_2)^{-1},\\
	\Omega^{(2)}_1&=\check S_1^{(2)}( S_1)^{-1}, &	\Omega^{(2)}_2&= S_2 L^{(2)}(\Upsilon)
	(\check S_2^{(2)})^{-1}.
	\end{align*}
\end{defi}
\begin{pro}
Geronimus connectors satisfy
$	\Omega^{(1)}_1 \check H^{(1)}=H \big(\Omega^{(1)}_2\big)^\dagger$ and $	\check H^{(2)} \big(\Omega^{(2)}_2\big)^\dagger = \Omega^{(2)}_1 H$.
\end{pro}
\begin{proof}
According to \eqref{GGeronimus}  and  \eqref{eq:quasidefGramGer} one has
	\begin{align*}
	L^{(1)}(\Upsilon)\big(\check  S_1^{(1)}\big)^{-1}\check H^{(1)}\big(\check S_2^{(1)}\big)^{-\dagger}&=(S_1)^{-1}H(S_2)^{-\dagger},&
	\big(\check S_1^{(2)}\big)^{-1}\check H^{(2)}\big(\check S_2^{(2)}\big)^{-\dagger}L^{(2)}(\Upsilon)&=(S_1)^{-1}H(S_2)^{-\dagger}.
	\end{align*}
so that
	\begin{align*}
S_1	L^{(1)}(\Upsilon)\big(\check  S_1^{(1)}\big)^{-1}\check H^{(1)}&=H(S_2)^{-\dagger}\big(\check S_2^{(1)}\big)^{\dagger},&
	\check H^{(2)}\big(\check S_2^{(2)}\big)^{-\dagger}L^{(2)}(\Upsilon)(S_2)^{-\dagger}&=\check S_1^{(2)}(S_1)^{-1}H.
	\end{align*}
\end{proof}
Therefore, since the factors $ S $ are lower unitriangular matrices we conclude the
\begin{pro}
	\begin{enumerate}
\item The connector $ \Omega ^ {(1)} _ 1 $ is an upper triangular matrix with only $ 2\underline{n} + 1 $ nonzero upper above.
% Addition $ (\ Omega ^ {(1)} _ 1) _ {k, k} = H_k (\ check H_k ^ {(1)}) ^ {- 1}. $% And $ (\ omega ^ {(1 )} _ 1) _ {2k, 2k + 2n} = \ alpha $
\item The connector $ \Omega ^ {(1)} _ 2 $ is a lower unitriangular matrix with only $ 2\underline{n} + 1 $ nonzero lower diagonals.

\item The connector $ \Omega ^ {(2)} _ 1 $ is a lower unitriangular matrix with only $ 2\underline{n} + 1 $ nonzero  diagonals below.
\item The connector $ \Omega ^ {(2)} _ 2 $ is an upper triangular matrix with only $ 2\underline{n}+ 1 $ nonzero upper diagonals.
\item We have the formulas
		\begin{align}\label{Omegalambda}
			( \Omega^{(1)}_2)_{k,k-2\underline{n}}&=\overline{ L^{(1)}_{(-1)^k\underline{n}}}		\dfrac{\bar {\check H}^{(1)}_k}{ \bar H_{k-2\underline{n}}}, &
		(\Omega^{(2)}_1)_{k,k-2\underline{n}}&=\overline{ L^{(2)}_{(-1)^k\underline{n}}}		\dfrac{\check H^{(2)}_k}{H_{k-2\underline{n}}}.
%		& \lambda_{i,k}:=\begin{cases}
%	L_{i,n}, & \text{$k$ is even},\\
%L_{i,-n}, & \text{$k$ is odd}.
%		\end{cases}
		\end{align}
	%	Además $(\bar\Omega^{(2)}_2)_{k,k}= H_k(\check H_2^{(2)})^{-1}$. %y $(\omega^{(1)}_1)_{2k+1,2k+1+2m}=\beta$.
	\end{enumerate}\end{pro}
The connectors are matrices that establish the relationships between the perturbed and original families of  biorthogonal  Laurent polynomials	
\begin{pro}[Connection formulas for the CMV Laurent orthogonal polynomials]\label{Geronimus Connection Laurent}
The following connection formulas hold
		\begin{align*}
		\Omega^{(1)}_1  \check\phi_1^{(1)}(z) &= L^{(1)}(z)\phi_1(z), & \Omega^{(1)}_2 \phi_2(z)&=\check\phi^{(1)}_2(z),\\
		\Omega_1^{(2)} \phi_1(z)&=\check\phi^{(2)}_1(z), & \Omega^{(2)}_2\check\phi^{(2)}_2(z)&=L^{(2)}(z) \phi_2(z).
		\end{align*}
	\end{pro}

\begin{defi}
Given a Laurent polynomial $L(z)$ we consider
$	\delta L (z_1,z_2):=\frac{L(z_1)-L(z_2)}{z_1-z_2}$.
and introduce the completely homogeneous symmetric polynomials
$
	h_j(z_1,z_2):=(z_1)^j+(z_1)^{j-1}z_2+\dots +z_1(z_2)^{j-1}+(z_2)^{j}$ 
and their duals $h^*_{j}(z_1,z_2):=(z_1z_2)^{-1}h_{j}\big((z_1)^{-1},(z_2)^{-1}\big)$
with $ j\in\{0,1,2,\dots\}$.
\end{defi}

\begin{pro}\label{simetricos}
It is true that
$\delta L(z_1,z_2)
	=\sum_{j=1}^{n} L_j  h_{j-1}(z_1, z_2)-\sum_{j=1}^{m}L_{-j} h^*_{j-1}(z_1, z_2)$
and therefore the bivariate Laurent polynomial $ \delta L (z_1, z_2) $ is symmetrical and, fixing one of the variables, is a  Laurent polynomial   in the other variable of positive maximum degree $ n-1 $ and  negative degree  $-m $.
\end{pro}
\begin{proof}
For $n\in\{1,2,\dots\}$ we deduce that
$	(z_1)^j-(z_2)^j= (z_1-z_2)h_{j-1}(z_1,z_2)$ and $(z_1)^{-j}-(z_2)^{-j}= -(z_1-z_2)h^*_{j-1}(z_1,z_2)$.
\end{proof}

\begin{pro}[Connection formulas for the Cauchy second kind functions]\label{Geronimus Connection Cauchy}
The following connection formulas are fulfilled
	\begin{align}
 \label{CC12}
	(C_2(z))^\dagger \big(\Omega^{(1)}_2\big)^\dagger - L^{(1)}(\bar z)(\check C^{(1)}_2(z))^\dagger &=-\prodint{
		\delta L^{(1)}(\bar z,z_1),\big(\check\phi^{(1)}_2(z_2)\big)^\top}_{\check u^{(1)}},\\
	\label{CC21}	\Omega^{(2)}_1 C_{1}(z)-\check C_{1}^{(2)}(z)\overline{ L^{(2)}}(z)&=-\prodint{ \check\phi^{(2)}_1(z_1),
		\delta L^{(2)}(\bar z,z_2)}_{ \check u^{(2)} },
	\\\label{CC11}
	\Omega^{(1)}_1 \check C^{(1)}_1(z)&=C_1(z),\\\label{CC22}
		\big(\check C^{(2)}_2(z)\big)^\dagger\big(\Omega_2^{(2)}\big)^\dagger&=\big(C_2(z)\big)^\dagger.
	\end{align}
	\end{pro}

\begin{proof} See Appendix
\end{proof}

	\begin{coro}
	The following connection formulas
		\begin{align}
		\label{CC12k>}
		(\Omega^{(1)}_2)_{l,l-2\underline{n}} C_{2,l-2\underline{n}}(z)+\dots+	(\Omega^{(1)}_2)_{l,l-1} C_{2,l-1}(z)+C_{2,l}(z)&=\overline{  L^{(1)}}(z)\check C^{(1)}_{2,l}(z),\\
		\label{CC21k>}
			(\Omega^{(2)}_1)_{l,l-2\underline{n}} C_{1,l-2\underline{n}}(z)+\dots+	(\Omega^{(2)}_1)_{l,l-1} C_{1,l-1}(z)+C_{1,l}(z)&=\overline{ L^{(2)}}(z)\check C_{1,l}^{(2)}(z),
				\end{align}
				hold for $l\geq 2\underline{n}$.
	\end{coro}
\begin{proof}
It is a consequence of the orthogonality relations  in Corollary \ref{ortogonalidad}, because they involve
\begin{align*}
	\prodint{\delta L^{(1)}(\bar z,z_1),\big(\check\phi^{(1)}_{2,l}(z_2)\big)^\top}_{\check u^{(1)}} =	\prodint{ \check\phi^{(2)}_{1,l}(z_1), \delta L^{(2)}(\bar z,z_2)}_{ \check u^{(2)} }
&=0,
		&l&\geq 2\underline{n}.
	\end{align*}
\end{proof}

%\subsection{Connection formulas for the Christoffel--Darboux kernels and their mixed versions}
\begin{defi}
Let's  define
	\begin{align*}
		\Omega^{(1)}_2[\underline{n},l]&:=\begin{bmatrix}
			(\Omega^{(1)}_2)_{l,l-2\underline{n}} & (\Omega^{(1)}_2)_{l,l-2\underline{n}+1}&(\Omega^{(1)}_2)_{l,l-2\underline{n}+2}&\dots& (\bar\Omega^{(1)}_2)_{l,l-1}\\
			0 & (\Omega^{(1)}_2)_{l+1,l-2\underline{n}+1}&(\Omega^{(1)}_2)_{l+1,l-2\underline{n}+2}&\dots& (\Omega^{(1)}_2)_{l+1,l-1}\\
			0 &0 &(\Omega^{(1)}_2)_{l+2,l-2\underline{n}+1}&\dots& (\Omega^{(1)}_2)_{l+2,l-1}\\
			\vdots& &\ddots & &\vdots\\
			0&0&\dots& &(\Omega^{(1)}_2)_{l+2\underline{n}-1,l-1}
		\end{bmatrix},\\
		\Omega^{(2)}_1[\underline{n},l]&:=\begin{bmatrix}
			(\Omega^{(2)}_1)_{l,l-2\underline{n}} & (\Omega^{(2)}_1)_{l,l-2\underline{n}+1}&(\Omega^{(2)}_1)_{l,l-2\underline{n}+2}&\dots& (\Omega^{(2)}_1)_{l,l-1}\\
			0 & (\Omega^{(2)}_1)_{l+1,l-2\underline{n}+1}&(\Omega^{(2)}_1)_{l+1,l-2\underline{n}+2}&\dots& (\Omega^{(2)}_1)_{l+1,l-1}\\
			0 &0 &(\Omega^{(2)}_1)_{l+2,l-2\underline{n}+1}&\dots& (\Omega^{(2)}_1)_{l+2,l-1}\\
			\vdots& &\ddots & &\vdots\\
			0&0&\dots& &(\Omega^{(2)}_1)_{l+2\underline{n}-1,l-1}
		\end{bmatrix},
	\end{align*}
and consider
$\check H^{(1)}[\underline{n},l]:=\diag({\check H}^{(1)}_l,\dots,{\check H}^{(1)}_{l+2\underline{n}-1})$ and $\check H^{(2)}[\underline{n},l]:=\diag({\check H}^{(2)}_l,\dots,{\check H}^{(2)}_{l+2\underline{n}-1})$.
\end{defi}

\begin{pro}[Connection formulas for the Christoffel--Darboux kernels]\label{Geronimus Connection CD}
For $l\geq 2\underline{n}$, the Christoffel--Darboux kernels and their Geronimus transformations
satisfy	\begin{align}\label{GerKerNor1}
&\begin{multlined}[t][0.9\textwidth]
{\check K^{(1),[l]}(\bar z_1,z_2)}-{L^{(1)}(z_2)} {K^{[l]}(\bar z_1,z_2)}=-{\Big[\check \phi^{(1)}_{1,l}(z_2),\dots,
\check \phi^{(1)}_{1,l+2\underline{n}-1}(z_2)\Big]}({\check H}^{(1)}[\underline{n},l])^{-1}\bar \Omega_2^{(1)}[\underline{n},l]
\begin{bmatrix}
\overline{ \phi_{2,l-2\underline{n}}(z_1)}\\ \vdots\\  \overline{
\phi_{2,l-1}(z_1)},
\end{bmatrix},
\end{multlined}
\\\label{GerKerNor2}
&\begin{multlined}[t][0.9\textwidth]
\check K^{(2),[l]}(\bar z_1,z_2)- {\overline{ L^{(2)}}(\bar z_1)} K^{[l]}(\bar z_1,z_2)=-\Big[\overline{\check\phi^{(2)}_{2,l}(z_1)},\dots, \overline{\check\phi^{(2)}_{2,l+2\underline{n}-1}(z_1)}\Big](\check H^{(2)}[\underline{n},l])^{-1}\Omega^{(2)}_1[\underline{n},l]
\begin{bmatrix}
\phi_{1,l-2\underline{n}}(z_2)\\\vdots\\\phi_{1,l-1}(z_2)
\end{bmatrix}.
\end{multlined}
\end{align}
\end{pro}
\begin{proof} See Appendix.
\end{proof}

	\begin{pro}[Connection formulas for the mixed Christoffel--Darboux kernels]\label{Geronimus Connection mixed}
	When $l\geq 2\underline{n}$,	the mixed Christoffel--Darboux kernels fulfill the following connection formulas
		\begin{align}
		\label{mix CF12}	&\begin{multlined}[t][0.9\textwidth]
		L^{(1)}(\bar x_1)\check K_{C,\phi}^{(1),[l]}(\bar x_1,x_2)-	L^{(1)}(x_2)K_{C,\phi}^{[l]}(\bar x_1,x_2)-
		\delta L^{(1)}(\bar x_1,x_2)
		\\		=
		-{\Big[\check \phi^{(1)}_{1,l}(x_2),\dots,
			\check \phi^{(1)}_{1,l+2\underline{n}-1}(x_2)\Big]}({\check H}^{(1)}[\underline{n},l])^{-1}
		\bar\Omega^{(1)}_2[\underline{n},l]
		\begin{bmatrix}
		\overline{ C_{2,l-2\underline{n}}(x_1)}\\ \vdots\\  \overline{
			C_{2,l-1}(x_1)},
		\end{bmatrix},
		\end{multlined}	
%	\label{mix CF1}	&\begin{multlined}[t][0.9\textwidth]
%			L^{(1)}(\bar x_1)\check K_{C,\phi}^{(1),[l]}(\bar x_1,x_2)-	L^{(1)}(x_2)K_{C,\phi}^{[l]}(\bar x_1,x_2)-	
%			\prodint{\delta L^{(1)}(\bar x_1,z_1), \overline{\check K^{(1),[l]}(\bar z_2,x_2)}}_{\check u^{(1)}}
%			\\=
%	-{\Big[\check \phi^{(1)}_{1,l}(x_2),\dots,
%			\check \phi^{(1)}_{1,l+2\underline{n}-1}(x_2)\Big]}
%({\check H}^{(1)}_l[\underline{n},l])^{-1}
%		 \bar\Omega^{(1)}_2[\underline{n},l]
%		\begin{bmatrix}
%		\overline{ C_{2,l-2\underline{n}}(x_1)}\\ \vdots\\  \overline{
%			C_{2,l-1}(x_1)}
%		\end{bmatrix},
%		\end{multlined}	
		\end{align}
		\begin{align}
		\label{mix CF2}	
		&\begin{multlined}[t][0.9\textwidth]
		\check K_{C,\phi}^{(2),[l]}(\bar x_1,x_2)-K_{C,\phi}^{[l]}(\bar x_1,x_2)=-
		\Big[\overline{\check C^{(2)}_{2,l}(x_1)},\dots, \overline{\check C^{(2)}_{2,l+2\underline{n}-1}(x_1)}\Big]
		(\check H^{(2)}[\underline{n},l])^{-1}
 \Omega^{(2)}_1[\underline{n},l]
\begin{bmatrix}
		\phi_{1,l-2\underline{n}}(x_2)\\\vdots\\\phi_{1,l-1}(x_2)
		\end{bmatrix},
		\end{multlined}
		\end{align}		
		
		\begin{align}	\label{mix FC1}	
		&	\begin{multlined}[t][0.9\textwidth]
		\check K_{\phi,C}^{(1),[l]}(\bar x_1,x_2)-K_{\phi,C}^{[l]}(\bar x_1,x_2)=
		-{\Big[\check C^{(1)}_{1,l}(x_2),\dots,
			\check C^{(1)}_{1,l+2\underline{n}-1}(x_2)\Big]}
({\check H}^{(1)}[\underline{n},l])^{-1}
				 \bar\Omega^{(1)}_2[\underline{n},l]
		\begin{bmatrix}
		\overline{ \phi_{2,l-2\underline{n}}(x_1)}\\ \vdots\\  \overline{
			\phi_{2,l-1}(x_1)}
		\end{bmatrix},
		\end{multlined}
	\end{align}

	\begin{align}
%	\label{mix FC2}	
%&\begin{multlined}[t][0.9\textwidth]
%\overline{ L^{(2)}}(x_2)\check K_{\phi,C}^{(2),[l]}(\bar x_1,x_2)-
%\overline{ L^{(2)}}(\bar x_1) K_{\phi,C}^{[l]}(\bar x_1,x_2)
%-
%\prodint{ \check K^{(2),[l]}(\bar x_1,z_1),
%\delta L^{(2)}(\bar x_2,z_2))}_{\check u^{(2)}}\\=
%-\Big[\overline{\check \phi^{(2)}_{2,l}(x_1)},\dots, \overline{\check \phi^{(2)}_{2,l+2\underline{n}-1}(x_1)}\Big]
%(\check H^{(2)}[\underline{n},l])^{-1}\Omega^{(2)}_1[\underline{n},l]
%\begin{bmatrix}
%C_{1,l-2\underline{n}}(x_2)\\\vdots\\C_{1,l-1}(x_2)
%\end{bmatrix}.
%\end{multlined}
\label{mix FC22}	
&\begin{multlined}[t][0.9\textwidth]
\overline{ L^{(2)}}(x_2)\check K_{\phi,C}^{(2),[l]}(\bar x_1,x_2)-	\overline{ L^{(2)}}(\bar x_1) K_{\phi,C}^{[l]}(\bar x_1,x_2)
-\delta \overline{ L^{(2)}}(\bar x_1,x_2)
%\prodint{ \check u^{(2)}_{z_1,\bar z_2},\check K^{[l]}(\bar z_1,x_2)\otimes
%	\frac{\bar L(  x_2)	-\bar L((z_2)^{-1})}{ x_2-(z_2)^{-1}}}
\\=
-\Big[\overline{\check \phi^{(2)}_{2,l}(x_1)},\dots, \overline{\check \phi^{(2)}_{2,l+2\underline{n}-1}(x_1)}\Big]
(\check H^{(2)}_{}[\underline{n},l])^{-1}
\Omega^{(2)}_1[\underline{n},l]
\begin{bmatrix}
C_{1,l-2\underline{n}}(x_2)\\\vdots\\C_{1,l-1}(x_2)
\end{bmatrix}.
\end{multlined}
		\end{align}
	\end{pro}	

\begin{proof}
See Appendix.
\end{proof}

\subsection{Christoffel--Geronimus formulas}
The perturbed Cauchy second kind functions read as follows
\begin{align*}
\overline{	\check C_{2,k}^{(1)}(z)}%&=\prodint{\frac{1}{\bar z-z_1}, \check \phi^{(1)}_{2,k}(z_2)}_{ (L^{(1)}(z_1))^{-1}u}+\sum_{i=1}^{d^{(1)}}\sum_{l=0}^{m^{(1)}_{i}-1}\frac{(-1)^l}{l!}\prodint{\delta^{(l)}\big(z_1-\zeta^{(1)}_{i}\big)
%	\otimes \overline{(\xi^{(1)}_{i,l})_{ z_2}},\frac{1}{\bar z-z_1}\otimes  	\overline{\check\phi^{(1)}_{2,k}}}\\
&=\prodint{\frac{1}{\bar z-z_1}, \check \phi^{(1)}_{2,k}(z_2)}_{ (L^{(1)}(z_1))^{-1}u}+\sum_{i=1}^{d^{(1)}}\sum_{l=0}^{m^{(1)}_{i}-1}
\frac{1}{l!}\frac{\d^l}{\d \zeta^l}
\Big(
\frac{1}{ \bar z-\zeta}\Big)
\bigg|_{\zeta= \zeta^{(1)}_{i}}\overline{\prodint{\xi^{(1)}_{i,l},	\check\phi^{(1)}_{2,k}}},& z\in\mathbb C\setminus \overline{\operatorname{supp}_1(u)\cup\sigma(L_1)}\end{align*}
\begin{align*}
\check C_{1,k}^{(2)}(z)%&=\prodint{\check\phi^{(2)}_{1,k}(z_1),
%	\frac{1}{\bar z-z_2}}_{u\overline{(L^{(2)}(z_2))^{-1}}}+\sum_{i=1}^{d^{(2)}}\sum_{l=0}^{m^{(2)}_{i}-1}\frac{(-1)^l}{l!}\prodint{(\xi^{(2)}_{i,l})_{z_1}\otimes \overline{\delta^{(l)}\big(z_2-\zeta^{(2)}_{i}\big)},
%	\check\phi^{(2)}_{1,k}(z_1)\otimes \frac{1}{ z-\bar z_2}}\\
&=\prodint{\check\phi^{(2)}_{1,k}(z_1),
	\frac{1}{\bar z-z_2}}_{u\overline{(L^{(2)}(z_2))^{-1}}}+\sum_{i=1}^{d^{(2)}}\sum_{l=0}^{m^{(2)}_{i}-1}\prodint{\xi^{(2)}_{i,l},	\check\phi^{(2)}_{1,k}}
\frac{1}{l!}\overline{\frac{\d^l}{\d \zeta^l}
\Big(
\frac{1}{ \bar z-\zeta}\Big)
\bigg|_{\zeta= \zeta^{(2)}_{		
		i}}},&z\in\mathbb C\setminus \overline{\operatorname{supp}_2(u)\cup \sigma(L_2)}.
\end{align*}
%with definition domains given by  $\operatorname{dom}\big({\check C_{2,k}^{(1)}}\big)=\mathbb C\setminus \overline{\operatorname{supp}_1(u)\cup\sigma(L_1)}$ and   $\operatorname{dom}(\check C_{1,k}^{(2)})=\mathbb C\setminus \overline{\operatorname{supp}_2(u)\cup \sigma(L_2)}$, respectively.

\begin{defi}
	\begin{enumerate}
		\item  For  $a\in\{1,2\}$ we define
		\begin{align*}
		\prodint{\xi^{(a)}_i,
			L}&:=\Big[\prodint{\xi^{(a)}_{i,1},
			L},\dots,\prodint{\xi^{(a)}_{i,m_i^{(a)}-1},
			L}\Big], &\prodint{\xi^{(a)},
		L}&:=\Big[\prodint{\xi^{(a)}_{1},	L
		},\dots,\prodint{\xi^{(a)}_{d^{(a)}}, L}
	\Big].
	\end{align*}
	\item The expression   $L(z)=L_n z^{-m}\prod\limits_{j=1}^{d}(z-\zeta_j)^{m_j}$ with  $m_1+\dots+m_d=n+m$, allows us to introduce
$	L_{[i]}(z):=
	L_n z^{-m}\prod\limits_{\substack{j=1\\j\neq i}}^{d}(z-\zeta_j)^{m_j}$.
\end{enumerate}
\end{defi}
Given relations \eqref{CC21k>} and  \eqref{CC12k>} let's look at
\begin{align*}
\overline{	\bar L^{(1)}(z)\check C_{2,k}^{(1)}(z)}
&=\prodint{\frac{L^{(1)}(\bar z)}{\bar z-z_1}, \check \phi^{(1)}_{2,k}(z_2)}_{ (L^{(1)}(z_1))^{-1}u}+\sum_{i=1}^{d^{(1)}}
\overline{\prodint{\xi^{(1)}_{i},	\check\phi^{(1)}_{2,k}}}
L^{(1)}_{[i]}(\bar z)\begin{bmatrix}
\big(\bar z- \zeta^{(1)}_{i}\big)^{m^{(1)}_{i}-1}\\ \big(\bar z- \zeta^{(1)}_{i}\big)^{m^{(1)}_{1}-2}\\\vdots\\1
\end{bmatrix},\\
\overline{ L^{(2)}}(z)\check C_{1,k}^{(2)}(z)&=
\prodint{\check\phi^{(2)}_{1,k}(z_1),\frac{L^{(2)}(\bar z)}{\bar z-z_2}}_{u(\overline{L^{(2)}(z_2)})^{-1}}
+\sum_{i=1}^{d^{(2)}}\prodint{\xi^{(2)}_{i},	\check\phi^{(2)}_{1,k}}
\overline{ L^{(2)}_{[i]}}(z)
\begin{bmatrix}
\big(\bar z- \zeta^{(2)}_{i}\big)^{m^{(2)}_{i}-1}\\ \big(\bar z- \zeta^{(2)}_{i}\big)^{m^{(2)}_{1}-2}\\\vdots\\1
\end{bmatrix}.
\end{align*}
We need to evaluate the spectral jets
$\mathcal J^{ \overline{L^{(1)}}}_{ \overline{L^{(1)}}\check C_{2,k}^{(1)}}$ y $\mathcal J^{\overline{ L^{(2)}}}_{\overline{ L^{(2)}}\check C_{1,k}^{(2)}}$.
We have to do it as limits since the perturbing polynomial zeros lay  on the border of the perturbed functional support.

\begin{defi}		For $i\in\{1,2\}$, $j\in\{1,\dots,d^{(i)}\}$, $k\in\{0,\dots,m^{(i)}_{j}-1\}$, let's use the notation
	\begin{align*}
	%	\kappa^{[j]}_k:&=\frac{1}{k!}\dfrac{\d^{k}L^{[j]}(\bar z)}{\d \bar z^{k}}\bigg|_{z=\zeta_j}, & 	
	\ell^{(i)}_{{[j]},k}&:=	\frac{1}{k!}
	\frac{\d^{k} \overline{L^{(i)}_{[j]}}(z)}{\d z^{k}}\bigg|_{z=\bar \zeta^{(i)}_{j}},&
	\mathcal L^{(i)}_{[j]}&:=	\begin{bmatrix}0 &0 & 0& & \ell^{[j]}_{i,0}\\
	\vdots&\vdots &\vdots & \iddots&\vdots\\ 0&0 &\ell^{(i)}_{[j],0}&\dots &\ell^{(i)}_{[j],m^{(i)}_{j}-3}\\	
	0&	\ell^{(i)}_{[j],0} &\ell^{(i)}_{[j],1}& \dots&\ell^{(i)}_{[j],m^{(i)}_{j}-2}\\
	\ell^{(i)}_{[j],0} 	&\ell^{(i)}_{[j],1} &\ell^{(i)}_{[j],2}& \dots&\ell^{(i)}_{[j],m^{(i)}_{j}-1}
	\end{bmatrix}, &\mathcal L^{(i)}&:=\diag(\mathcal L^{(i)}_{[1]},\dots,\mathcal L^{(i)}_{[d^{(i)}]}).
	\end{align*}
\end{defi}

\begin{pro}\label{restos}
The spectral jets satisfy
	\begin{align}
	\label{caleGer}
	\mathcal J^{\overline{ L^{(2)}}}_{\overline{ L^{(2)}}\check C^{(2)}_{{1,k}}}&=\prodint{\xi^{(2)},\check\phi^{(2)}_{1,k}}\mathcal L^{(2)},\\\label{chorro (1)}
	\mathcal J^{ \overline{L^{(1)}}}_{\overline{L^{(1)}}\check C_{2,k}^{(1)}}&=	{\prodint{\xi^{(1)},	\check\phi^{(1)}_{2,k}}\mathcal L^{(1)}}.
	\end{align}
\end{pro}
\begin{proof} See Appendix.
\end{proof}

\begin{teo}[Christoffel-Geronimus formulas]\label{Christoffel-Geronimus formulas}
Let's assume that  $n=m$ y $l\geq 2n$ and that  Laurent polynomials $L^{(1)}(z)$ and $L^{(2)}(z)$ are prepared. Then, when
$\check\tau^{(2)}_{l}:=
	\begin{vsmallmatrix}
	\mathcal J^{\overline{ L^{(2)}}}_{ C_{1,l-2n}}-\prodint{\xi^{(2)},\phi_{1,l-2n}}\mathcal L^{(2)}\\\vdots \\
	\mathcal J^{\overline{ L^{(2)}}}_{ C_{1,l-1}}-\prodint{\xi^{(2)},\phi_{1,l-1}}\mathcal L^{(2)}
	\end{vsmallmatrix}\neq 0$,
we have the Christoffel-Geronimus formulas
	\begin{align}\label{Ger12}
	\check\phi^{(2)}_{1,l}(z)&=\Theta_*\begin{bmatrix}
	\mathcal J^{\overline{ L^{(2)}}}_{ C_{1,l-2n}}-\prodint{\xi^{(2)},\phi_{1,l-2n}}\mathcal L^{(2)} & \phi_{1,l-2n}(z)\\\vdots&\vdots \\
	\mathcal J^{\overline{ L^{(2)}}}_{ C_{1,l}}-\prodint{\xi^{(2)},\phi_{1,l}}\mathcal L^{(2)}&\phi_{1,l}(z)
	\end{bmatrix}=\frac{1}{\check\tau^{(2)}_{l}}\begin{vmatrix}
	\mathcal J^{\overline{ L^{(2)}}}_{ C_{1,l-2n}}-\prodint{\xi^{(2)},\phi_{1,l-2n}}\mathcal L^{(2)} & \phi_{1,l-2n}(z)\\\vdots&\vdots \\
	\mathcal J^{\overline{ L^{(2)}}}_{ C_{1,l}}-\prodint{\xi^{(2)},\phi_{1,l}}\mathcal L^{(2)}&\phi_{1,l}(z)
	\end{vmatrix},
	\end{align}
	\begin{align}
	\label{GerH2}
	\check H^{(2)}_{l}&=\dfrac{H_{l-2n}}{\overline{ L^{(2)}_{(-1)^ln}}}\Theta_*\begin{bmatrix}
	\mathcal J^{\overline{ L^{(2)}}}_{ C_{1,l-2n}}-\prodint{\xi^{(2)},\phi_{1,l-2n}}\mathcal L^{(2)} & 1\\
		\mathcal J^{\overline{ L^{(2)}}}_{ C_{1,l-2n+1}}-\prodint{\xi^{(2)},\phi_{1,l-2n+1}}\mathcal L^{(2)} & 0\\
	\vdots&\vdots \\
	\mathcal J^{ \overline{L^{(2)}}}_{ C_{1,l}}-\prodint{\xi^{(2)},\phi_{1,l}}\mathcal L^{(2)}&0
	\end{bmatrix}
	=\dfrac{H_{l-2n}}{\overline{ L^{(2)}_{(-1)^ln}}}\frac{\check \tau^{(2)}_{l+1}}{\check \tau^{(2)}_{l}},\end{align}
	\begin{align}
	\label{Ger22}
\begin{multlined}[t][0.9\textwidth]
	\overline{\check \phi^{(2)}_{2,l}(z)}=-
\dfrac{H_{l-2n}}{\overline{ L^{(2)}_{(-1)^ln}}}	\Theta_*\begin{bmatrix}
	\mathcal J^{\overline{L^{(2)}}} _{C_{1,l-2n}}-\prodint{\xi^{(2)},\phi_{1,l-2n}}\mathcal L^{(2)} & 1\\
	\mathcal J^{\overline{L^{(2)}}} _{C_{1,l-2n+1}}-\prodint{\xi^{(2)},\phi_{1,l-2n+1}}\mathcal L^{(2)}&0\\
	\vdots & \vdots\\
	\mathcal J^{\overline{L^{(2)}}} _{C_{1,l-1}}-\prodint{\xi^{(2)},\phi_{1,l-1}}\mathcal L^{(2)} &0\\
	\overline{L^{(2)}( z) }\Big(\mathcal  J^{\overline{L^{(2)}}}_{K_{\phi,C}^{[l]}}(\bar z)-	\prodint{(\xi^{(2)})_x,K^{[l]}(\bar z,x)}\mathcal L^{(2)}
	\Big)+
	\overline{\mathcal  J^{ L^{(2)}}_{\delta  L^{(2)}}( z) }& 0
	\end{bmatrix}\\ =-\dfrac{H_{l-2n}}{\overline{ L^{(2)}_{(-1)^ln}}}\frac{1}{\check \tau^{(2)}_l}	\begin{vmatrix}
	\mathcal J^{{ \overline{L^{(2)}}} } _{C_{1,l-2n+1}}-\prodint{\xi^{(2)},\phi_{1,l-2n+1}}\mathcal L^{(2)}\\
	\vdots \\
	\mathcal J^{ \overline{L^{(2)}}} _{C_{1,l-1}}-\prodint{\xi^{(2)},\phi_{1,l-1}}\mathcal L^{(2)}\\
	\overline{L^{(2)}( z) }\Big(\mathcal  J^{{ \overline{L^{(2)}}} }_{K_{\phi,C}^{[l]}}(\bar z)-	\prodint{(\xi^{(2)})_x,K^{[l]}(\bar z,x)}\mathcal L^{(2)}
	\Big)+
		\overline{\mathcal  J^{ L^{(2)}}_{\delta  L^{(2)}}( z) }
	\end{vmatrix},
\end{multlined}
	\end{align}
	where the spectral jets of the mixed Christoffel--Darboux kernel and of  $\delta \bar L_2$ are taken with respect to the second variable.
Whenever $\check \tau^{(1)}_l:=\begin{vsmallmatrix}
	\mathcal J^{ \overline{L^{(1)}}}_{C_{2,l-2n}}-{\prodint{\xi^{(1)},	\phi_{2,l-2n}}\mathcal L^{(1)}}
\\\vdots\\
\mathcal J^{ \overline{L^{(1)}}}_{C_{2,l-1}}
	-{\prodint{\xi^{(1)},	\phi_{2,l-1}}{\mathcal L^{(1)}}}\end{vsmallmatrix}\neq 0$ y $l\geq 2n$ the following  Christoffel
formulas are satisfied
	\begin{align}\label{Ger21}
	\check \phi_{2,l}^{(1)}(z)&=\Theta_*\begin{bmatrix}
	\mathcal J^{ \overline{L^{(1)}}}_{C_{2,l-2n}}-{\prodint{\xi^{(1)},	\phi_{2,l-2n}}{\mathcal L^{(1)}} }&\phi_{2,l-2n}(z)\\
	\vdots & \vdots \\
	\mathcal J^{ \overline{L^{(1)}}}_{C_{2,l}}-{\prodint{\xi^{(1)},	\phi_{2,l}}{\mathcal L^{(1)}}} &\phi_{2,l}(z)\
	\end{bmatrix}
=\frac{1}{\check \tau^{(1)}_l}\begin{vmatrix}
	\mathcal J^{ \overline{L^{(1)}}}_{C_{2,l-2n}}-{\prodint{\xi^{(1)},	\phi_{2,l-2n}}{\mathcal L^{(1)}} }&\phi_{2,l-2n}(z)\\
	\vdots & \vdots \\
	\mathcal J^{ \overline{L^{(1)}}}_{C_{2,l}}-{\prodint{\xi^{(1)},	\phi_{2,l}}{\mathcal L^{(1)}}} &\phi_{2,l}(z)\
	\end{vmatrix},\end{align}
	\begin{align}
	\label{GerH1}
	\bar{\check H}_{l}^{(1)}&=\dfrac{\bar H_{l-2n}}{\overline{ L^{(1)}_{(-1)^ln}}}\Theta_*\begin{bmatrix}
	\mathcal J^{ \overline{L^{(1)}}}_{C_{2,l-2n}}-{\prodint{\xi^{(1)},	\phi_{2,l-2n}}{\mathcal L^{(1)}}} &1\\
		\mathcal J^{ \overline{L^{(1)}}}_{C_{2,l-2n+1}}-{\prodint{\xi^{(1)},	\phi_{2,l-2n+1}}{\mathcal L^{(1)}}}&0\\
	\vdots & \vdots \\
	\mathcal J^{ \overline{L^{(1)}}}{C_{2,l}}-{\prodint{\xi^{(1)},	\phi_{2,l}}{\mathcal L^{(1)}} }&0
	\end{bmatrix}=\dfrac{\bar H_{l-2n}}{\overline{ L^{(1)}_{(-1)^ln}}}\frac{\check \tau^{(1)}_{l+1}}{\check \tau^{(1)}_l}
,\end{align}
	\begin{multline}
	\label{Ger11}
\overline{	\check \phi^{(1)}_{1,l}(z)}
		=-\dfrac{\bar H_{l-2n}}{\overline{ L^{(1)}_{(-1)^ln}}}
		\Theta_*\begin{bmatrix}
		{ \mathcal J^{ \overline{L^{(1)}}}_{C_{2,l-2n}}-	{\prodint{\xi^{(1)},	{ \phi_{2,l-2n}}}		\mathcal L^{(1)}}} &1\\
		{ \mathcal J^{ \overline{L^{(1)}}}_{C_{2,l-2n+1}}-	{\prodint{\xi^{(1)},	\phi_{2,l-2n+1}}		\mathcal L^{(1)}} }&0\\\vdots & \vdots\\	
{ \mathcal J^{ \overline{L^{(1)}}}_{C_{2,l-1}}-	{\prodint{\xi^{(1)},	\phi_{2,l-1}}		\mathcal L^{(1)}}}&0\\
		\overline{ L^{(1)}(z)}\Big(\mathcal J^{ L_1}_{\bar K_{C,\phi}^{[l]}}(z)-\prodint{{(\xi^{(1)})_x},	\overline{K^{[l]}(\bar x,z)}}
		{\mathcal L^{(1)}}\Big)+	\overline{\mathcal J^{ \overline{L^{(1)}}}_{\delta  L_1}(z)}&0
		\end{bmatrix}\\
=-
\dfrac{\bar H_{l-2n}}{\overline{ L^{(1)}_{(-1)^ln}}}\frac{1}{\check\tau^{(1)}_l}\begin{vmatrix}
		{ \mathcal J^{ \overline{L^{(1)}}}_{C_{2,l-2n+1}}-	{\prodint{\xi^{(1)},	\phi_{2,l-2n+1}}		\mathcal L^{(1)}} }\\\vdots \\	
{ \mathcal J^{ \overline{L^{(1)}}}_{C_{2,l-1}}-	{\prodint{\xi^{(1)},	\phi_{2,l-1}}		\mathcal L^{(1)}}}\\
		\overline{ L^{(1)}(z)}\Big(\mathcal J^{ L_1}_{\bar K_{C,\phi}^{[l]}}(z)-\prodint{{(\xi^{(1)})_x},	\overline{K^{[l]}(\bar x,z)}}
		{\mathcal L^{(1)}}\Big)+	\overline{\mathcal J^{ \overline{L^{(1)}}}_{\delta  L_1}(z)}
		\end{vmatrix}
,
		\end{multline}
		where the spectral jet of the Christoffel--Darboux kernels and of $\delta L_1$ is taken with respect to the first variable.
\end{teo}
\begin{proof} See Appendix.
\end{proof}

\subsection{Reductions to univariate linear functionals supported on the unit circle}
As we have discussed, for  the two possible Geronimus  transformations we have  the corresponding Christoffel--Geronimus formulas. As for the Christoffel transformation situation we look at the case
$\prodint{L(z),M(z)}_u=\prodint{u_z, L(z)\overline{M(z)}}$
%we have these two possibilities, $u^{(1)}_z=L^{(1)}(z)u_z$ or $ u^{(2)}_z=\overline{L^{(2)}(z)}u_z$. If
with  $\operatorname{supp}u_z\subset \gamma$, where the curve $\gamma=\{z
\in\mathbb C^*:\bar z=\tau(z)
\}$, we have that the perturbation of type 2 can be described as
$
\check u^{(2)}_z=\overline{ L^{(2)}}(\tau(z))u_z$.
For the unit circle $\gamma=\mathbb T$,  with $\tau(z)=z^{-1}$, this Geronimus  transformation can be thought, as
$L^{(2)}_{*}(z)\check u^{(2)}_z=u_z$ and $ L^{(2)}_{*}(z)=\overline{ L^{(2)}}(z^{-1})$,
so that taking $L^{(2)}(z)=L^{(1)}_{*}(z)$ we get that $L^{(1)}(z)\check  u^{(1)}_z=L^{(1)}(z)\check  u^{(2)}_z=u_z$.
Recall that we have the following two mass terms, each coming from one of the two Geronimus transformations we have considered, one generated by $L(z)$ and the other by $L_*(z)=\bar L(z^{-1})$,  having zeros $\sigma(L)=\{\zeta_i\}_{i=1}^d$ and $\sigma(L_*)=\{(\bar\zeta_i)^{-1}\}_{i=1}^d$, respectively,
\begin{align*}
\mathcal M^{(1)}=&\sum_{i=1}^{d}\sum_{k=0}^{m_{i}-1}\frac{(-1)^{k}}{k!}\delta^{(k)}(z_1-\zeta_{i})
\otimes \overline{(\xi^{(1)}_{i,k})_{ z_2}},  &
\mathcal M^{(2)}=&\sum_{j=1}^{d}\sum_{l=0}^{m_{j}-1}\frac{(-1)^{l}}{l!}(\xi^{(2)}_{j,l})_{z_1}\otimes\overline{\delta^{(l)}\big(z_2- (\bar \zeta_{j})^{-1}\big)}.
\end{align*}
If we take
$(\xi^{(1)}_{i,k})_{z_2}=\sum\limits_{j=1}^d\sum\limits_{l=0}^{m_j-1}\frac{(-1)^l}{l!}\bar \Xi_{i,k|j,l}\delta^{(l)}(z_2-(\bar\zeta_{j})^{-1})$ and $
(\xi^{(2)}_{j,l})_{z_1}=\sum\limits_{i=1}^d\sum\limits_{k=0}^{m_i-1}\frac{(-1)^k}{k!}\Xi_{i,k|j,l}\delta^{(k)}(z_1-\zeta_{i})$
we find $\mathcal M^{(1)}=\mathcal M^{(2)}=:\mathcal M$ with
\begin{align}\label{mass_term}
\mathcal M=&\sum_{i,j=1}^{d}\sum_{k=0}^{m_{i}-1}\sum_{l=0}^{m_j-1}\frac{(-1)^{k+l}}{k!l!}\Xi_{i,k|j,l}\delta^{(k)}(z_1-\zeta_{i})\otimes\overline{\delta^{(l)}\big(z_2- (\bar \zeta_{j})^{-1}\big)}.
\end{align}
This mass term is possibly not supported on the diagonal, but it is the most general mass term such that both  Geronimus transformations, of an univariate sesquilinear form,
are equal. The corresponding mass contribution to the  perturbed sesquilinear form is
$\prodint{\mathcal M, M_1(z_1)\otimes \overline{M_2(z_2)}}=\sum\limits_{i,j=1}^{d}\sum\limits_{k=0}^{m_{i}-1}\sum\limits_{l=0}^{m_j-1}\frac{1}{k!l!}\Xi_{i,k|j,l}M_1^{(k)}(\zeta_i)\overline{M_2^{(l)}\big((\bar \zeta_j)^{-1}\big)}$,
so that the perturbed sesquilinear form is
\begin{align*}
\prodint{M_1(z_1),M_2(z_2)}_{\check u}=\prodint{u_z, \frac{M_1(z)\overline{M_2(z)}}{L(z)}}+\sum_{i,j=1}^{d}\sum_{k=0}^{m_{i}-1}\sum_{l=0}^{m_j-1}\Xi_{i,k|j,l}M_1^{(k)}(\zeta_i)\overline{M_2^{(l)}\big((\bar \zeta_j)^{-1}\big)}.
\end{align*}

As we have
\begin{align*}
\prodint{(\xi^{(1)}_{i,k})_{z_2},\phi(z_2)}&=\sum_{j=1}^d\sum_{l=0}^{m_j-1}\frac{1}{l!}\phi^{(l)}((\bar\zeta_{j})^{-1})\bar \Xi_{i,k|j,l},&
\prodint{(\xi^{(2)}_{j,l})_{z_2},\phi(z_2)}&=\sum_{i=1}^d\sum_{k=0}^{m_i-1}\frac{1}{k!}\phi^{(k)}(\zeta_{i})\Xi_{i,k|j,l},
\end{align*}
we  introduce
\begin{align*}
\Xi&:=
\begin{bmatrix}
\Xi_{1,0|1,0}& \cdots & \Xi_{1,0|1,m_1-1}& \dots &\Xi_{1,0|d,0}& \cdots &\Xi_{1,0|d,m_d-1}\\
\vdots & & \vdots & &\vdots & &\vdots\\
 \Xi_{1,m_1-1|1,0}& \cdots& \Xi_{1,m_1-1|1,m_1-1} &\dots &\Xi_{1,m_1-1|d,0}& \cdots &\Xi_{1,m_1-1|d,m_d-1}\\
\vdots & & \vdots & &\vdots & &\vdots\\
\Xi_{d,0|1,0} &\dots &\Xi_{d,0|1,m_1-1}  &\dots &\Xi_{d,0|d,0}& \cdots &\Xi_{d,0|d,m_d-1}\\
\vdots & & \vdots\\
 \Xi_{d,m_d-1|1,0}&\cdots &  \Xi_{d,m_d-1|1,m_1-1}& \dots  & \Xi_{d,m_d-1|d,0}& \cdots & \Xi_{d,m_d-1|d,m_d-1}
\end{bmatrix}\in\mathbb C^{2n\times 2n}
\end{align*}
so that we can we deduce the following expressions in terms of spectral jets
$\prodint{\xi^{(1)},\phi}=\mathcal J ^{L_*}_{\phi}\Xi^\dagger$ and $\prodint{\xi^{(2)},\phi}=\mathcal J ^{L}_{\phi}\Xi$.
We also need of the matrix $\mathcal L_*$ associated to the reciprocal Laurent polynomial $L_*(z)$.

\begin{teo}\label{Christoffel-Geronimus formulas circle general mases}
	Given a sesquilinear from associated with a univariate linear functional supported on the unit circle $ \mathbb T$ and a Geronimus transformation with a prepared
	perturbing Laurent polynomial $L(z)$, for  $l\geq 2n$,
	we have the following expressions (whenever the involved quasideterminants  do exist)
	\begin{align*}%\label{Ger12}
	\check\phi_{1,l}(z)&=\Theta_*\begin{bmatrix}
	\mathcal J^{\bar L_*}_{ C_{1,l-2n}}-\mathcal J^L_{\phi_{1,l-2n}}\Xi\mathcal L_* & \phi_{1,l-2n}(z)\\\vdots&\vdots \\
	\mathcal J^{\bar L_*}_{ C_{1,l}}-\mathcal J^L_{\phi_{1,l-1}}\Xi\mathcal L_*&\phi_{1,l}(z)
	\end{bmatrix}=-{	\Theta_*\begin{bmatrix}
		\overline{	{ \mathcal J^{\bar L}_{C_{2,l-2n}}-	{\mathcal J^{L_*}_{\phi_{2,l-2n}}} \Xi^\dagger	\mathcal L}} & \frac{H_{l-2n}}{ L_{(-1)^ln}}\\
		\overline{	{ \mathcal J^{\bar L}_{C_{2,l-2n+1}}-{\mathcal J^{L_*}_{\phi_{2,l-2n+1}}} \Xi^\dagger	\mathcal L} }&0\\\vdots
		& \vdots\\\overline{	{ \mathcal J^{\bar L}_{C_{2,l-1}}-	{\mathcal J^{L_*}_{\phi_{2,l-1}}} \Xi^\dagger	\mathcal L}}&0\\
		{L(z)\Big(\overline{\mathcal J^{\bar L}_{K_{C,\phi}^{[l]}}(z)}-
			{\mathcal J^{L_*}_{K^{[l]}}
			\Xi^\dagger\mathcal L}\Big)+	\mathcal J^{\bar L}_{\delta L}(z)}&0
		\end{bmatrix}},\\
	%	\label{GerH2}
	\check H_{l}&=\Theta_*\begin{bmatrix}
	\mathcal J^{\bar L_*}_{ C_{1,l-2n}}-\mathcal J^L_{\phi_{1,l-2n}}\Xi\mathcal L_* &  \frac{H_{l-2n}}{\bar L_{(-1)^{l+1}n}}\\
	\mathcal J^{\bar L_*}_{ C_{1,l-2n}}-\mathcal J^L_{\phi_{1,l-2n+1}}\Xi\mathcal L_* & 0\\
	\vdots&\vdots \\
	\mathcal J^{\bar L_*}_{ C_{1,l}}-\mathcal J^L_{\phi_{1,l}}\Xi\mathcal L_*&0
	\end{bmatrix}=\overline{\Theta_*\begin{bmatrix}
		\mathcal J^{\bar L}_{C_{2,l-2n}}-{\mathcal J^{L_*}_{\phi_{2,l-2n}}} \Xi^\dagger{\mathcal L} & \frac{\bar H_{l-2n}}{\bar L_{(-1)^ln}}\\
		\mathcal J^{\bar L}_{C_{2,l-2n+1}}-{\mathcal J^{L_*}_{\phi_{2,l-2n+1}}} \Xi^\dagger {\mathcal L}&0\\
		\vdots & \vdots \\
		\mathcal J^{\bar L}_{C_{2,l}}-{\mathcal J^{L_*}_{\phi_{2,l}}} \Xi^\dagger{\mathcal L} &0
		\end{bmatrix}},\\
	%	\label{Ger22}
	\overline{\check \phi_{2,l}(z)}&=-
	\Theta_*\begin{bmatrix}
	\mathcal J^{\bar L_*} _{C_{1,l-2n}}-\mathcal J^L_{\phi_{1,l-2n}}\Xi\mathcal L_* &  \frac{H_{l-2n}}{ L_{(-1)^{l+1}n}}\\
	\mathcal J^{\bar L_*} _{C_{1,l-2n+1}}-\mathcal J^L_{\phi_{1,l-2n+1}}\Xi\mathcal L_*&0\\
	\vdots & \vdots\\
	\mathcal J^{\bar L_*} _{C_{1,l-1}}-\mathcal J^L_{\phi_{1,l-2n}}\Xi\mathcal L_*&0\\
	\overline{ L_*( z)} \Big(\mathcal  J^{\bar L_*}_{K_{\phi,C}^{[l]}}(\bar z)-	\mathcal J^L_{K^{[l]}}(\bar z)\Xi\mathcal L_*
	\Big)+
	\overline{\mathcal  J^{ L_*}_{\delta  L_*}( z) }& 0
	\end{bmatrix}=\overline{\Theta_*\begin{bmatrix}
		\mathcal J^{\bar L}_{C_{2,l-2n}}-{\mathcal J^{L_*}_{\phi_{2,l-2n}}}\Xi^\dagger{\mathcal L} &\phi_{2,l-2n}(z)\\
		\vdots & \vdots \\
		\mathcal J^{\bar L}_{C_{2,l}}-{\mathcal J^{L_*}_{\phi_{2,l}}} \Xi^\dagger{\mathcal L} &\phi_{2,l}(z)\
		\end{bmatrix}},
	\end{align*}
	in the first line spectral  jets of the Christoffel--Darboux kernels and of $ \delta \bar L $ are taken with respect  to its first variable, while in the third  line   spectral jets of the  Christoffel--Darboux  kernel and $ \delta L $ are taken with respect to the second variable.
\end{teo}

\subsubsection{Diagonal masses}
We now discuss what happens when the masses are restricted to be supported on the diagonal.
\begin{defi}	
	We define
	\begin{align*}
	\mathcal B^{[i]}_{k,j}&:=
	(-1)^k\!\!\!\!\!\!\!\!\!\!\!\!\!\!\!\sum_{\substack{j_1+j_2+\dots+j_{k-j+1}=j\\j_1+2j_2+\dots+(k-j+1)j_{k-j+1}=k}}\!\!\!\!\!\!\!\!\!\!
	\frac{k!}{j_1!\cdots j_{k-j+1}!}(\zeta_i)^{-k-j},\\
	\mathcal B^{[i]}&:=\begin{bmatrix}
	1 & 0 & 0 &0 &\dots &0\\
	0 &\mathcal B^{[i]}_{1,1} & 0&0 &\dots &0\\
	0 &\mathcal B^{[i]}_{2,1} & \mathcal B^{[i]}_{2,2} &0 &\dots &0\\
	0 &\mathcal B^{[i]}_{3,1} & \mathcal B^{[i]}_{3,2} &\mathcal B^{[i]}_{3,3}&\dots &0\\
	\vdots& \vdots&\vdots&\vdots&\ddots&\\
	0 &\mathcal B^{[i]}_{m_i-1,1} & \mathcal B^{[i]}_{m_i-1,2} &\mathcal B^{[i]}_{m_1-1,3}&\dots &\mathcal B^{[i]}_{m_i-1,m_i-1}
	\end{bmatrix},
	\end{align*}
	and given complex numbers $\Xi_l^i$ we consider
	\begin{align*}
	\Xi^i&:=\begin{bmatrix}
	\Xi_0^i &\frac{1}{1!1!}\Xi_1^i &\frac{1}{1!2!}\Xi_2^i&\dots &\frac{1}{1!(m_i-1)!}\Xi_{m_i-1}\\
	\frac{1}{1!1!}\Xi_1^ i&\frac{1}{1!1!}\Xi_2^i & & \iddots &0\\
	\frac{1}{2!1!}\Xi_2^i & & \iddots&\iddots &\vdots\\\vdots&\iddots&\iddots&&\vdots\\
	\frac{1}{(m_i-1)!1!}
	\Xi_{m_i-1}&0 & \dots&\dots &0
	\end{bmatrix}.
	\end{align*}
\end{defi}
\begin{pro}\label{Christoffel-Geronimus formulas circle diagonal mases}
For the particular choice $\Xi=\diag(\Xi_1,\dots,\Xi_d)$ with $\Xi_i=\Xi^i\mathcal B^{[i]}$ the mass term \eqref{mass_term} is of the form
\begin{align*}
\prodint{\mathcal M, M_1(z_1)\otimes \overline{M_2(z_2)}}=\prodint{\sum_{i=1}^{d}\sum_{l=0}^{m_i-1}\Xi^i_l\frac{(-1)^l}{l!}\delta^{(l)}(z-\zeta_i),M_1(z)M_{2,*}(z)}.
\end{align*}
\end{pro}

\begin{proof}
Using Bell polynomials and the Faà di Bruno formula one can show that
$(M_*)^{(k)}(\zeta_i)=\sum_{j=1}^{k} \bar M^{(j)}((\zeta_i)^{-1}) 	\mathcal B^{[i]}_{k,j}$
with $\mathcal B^{[i]}_{k,0}=\delta_{k,0}$.
Let us consider the expression
\begin{align*}
\prodint{\sum_{i=1}^{d}\sum_{l=0}^{m_i-1}\Xi^i_l\frac{(-1)^l}{l!}\delta^{(l)}(z-\zeta_i),M_1(z)M_{2,*}(z)}&=\sum_{i=1}^{d}\sum_{l=0}^{m_i-1}
\sum_{k=0}^l\Xi^i_l\frac{1}{l!}\binom{l}{k}M^{(l-k)}_1(\zeta_i)(M_{2,*})^{(k)}(\zeta_i)\\
%&=\sum_{i=1}^{d}\sum_{l=0}^{m_i-1}
%\sum_{k=0}^l\sum_{j=0}^{k} \Xi^i_l\frac{1}{(l-k)!k!}\mathcal B^{[i]}_{k,j}M^{(l-k)}_1(\zeta_i) \overline{M^{(j)}( (\bar\zeta_i)^{-1}) }\\
&=\sum_{i=1}^{d}\sum_{n,j=0}^{m_i-1}\sum_{k=j}^{m_i-1}
\Xi^i_{k+n}\frac{1}{n!k!}\mathcal B^{[i]}_{k,j}M^{(n)}_1(\zeta_i) \overline{M^{(j)}( (\bar\zeta_i)^{-1}) }\\
&=\sum_{i=1}^{d}\sum_{n,j=0}^{m_i-1}\Xi_{i,n|i,j}
M^{(n)}_1(\zeta_i) \overline{M^{(j)}( (\bar\zeta_i)^{-1}) },
\end{align*}
where $\Xi_{i,n|i,j}=\sum_{k=j}^{m_i-1} \Xi^i_{k+n}\frac{1}{n!k!}\mathcal B^{[i]}_{k,j}$.
Then, we find  that $\Xi=\diag(\Xi_1,\dots,\Xi_d)$ with $\Xi_i=\Xi^i\mathcal B^{[i]}$.
\end{proof}

Finally,  if the original functional is nonnegative, i. e. biorthogonality is just orthogonality and we consider a Geronimus
 transformation with a prepared Laurent polynomial $L(z)$ which is equal to its reciprocal, $L_*(z)=L(z)$  and $L(z)=|L(z)|$ for $z\in\mathbb T$,    and  masses being positive linear functionals, which is achieved only when  $\Xi_l^i=\xi^i\delta_{l,0}$,  $\xi^i\geq 0$, we get $
 	\Xi^i=\Xi_i=\begin{bsmallmatrix}
 		\xi^i &0&\dots &0\\
 	0&0&  &0\\
 	0&0 &\dots &0
 	\end{bsmallmatrix}\in\mathbb C^{m_i\times m_i}$
and, consequently, we find
$
\Xi \mathcal L_*=\Xi^\dagger\mathcal L=\diag\left(\begin{bsmallmatrix}
0  &\dots& 0&\xi^1 \overline{L_{[1]}( \zeta_1)}\\
0 &\dots &0&0\\
\vdots & & \vdots&\vdots\\
0&\dots&0&0
\end{bsmallmatrix} ,\dots, \begin{bsmallmatrix}
0  &\dots& 0&\xi^d \overline{L_{[d]}( \zeta_d)}\\
0 &\dots &0&0\\
\vdots & & \vdots&\vdots\\
0&\dots&0&0
\end{bsmallmatrix} \right)$,
so that
 \begin{align*}
 \mathcal J^L_{\phi} \Xi\mathcal L_*=\Big[0,\dots,0, \phi(\zeta_1)\xi^1 \overline{L_{[1]}( \zeta_1)},\dots, 0,\dots,0, \phi(\zeta_d)\xi^d\overline{ L_{[d]}( \zeta_d)}\Big].
 \end{align*}
We conclude that
\begin{align*}
	\check\phi_{l}(z)&=\Theta_*\begin{bmatrix}
\mathcal J^{\bar L}_{ C_{l-2n}}-\mathcal J^L_{\phi_{l-2n}}\Xi\mathcal L_* & \phi_{l-2n}(z)\\\vdots&\vdots \\
\mathcal J^{\bar L}_{ C_{l}}-\mathcal J^L_{\phi_{l}}\Xi\mathcal L_*&\phi_{l}(z)
\end{bmatrix}=-{	\Theta_*\begin{bmatrix}
	\overline{	{ \mathcal J^{\bar L}_{C_{l-2n}}-	{\mathcal J^{L}_{\phi_{l-2n}}} \Xi\mathcal L_*}} & \frac{H_{l-2n}}{ L_{(-1)^ln}}\\
	\overline{	{ \mathcal J^{\bar L}_{C_{l-2n+1}}-{\mathcal J^{L}_{\phi_{l-2n+1}}}\Xi\mathcal L_*} }&0\\\vdots
	& \vdots\\\overline{	{ \mathcal J^{\bar L}_{C_{l-1}}-	{\mathcal J^{L}_{\phi_{l-1}}} \Xi	\mathcal L_*}}&0\\
	{L(z)\Big(\overline{\mathcal J^{\bar L}_{K_{C,\phi}^{[l]}}(z)}-
		{\mathcal J^{L}_{K^{[l]}}
		\Xi\mathcal L_*}\Big)+	\mathcal J^{\bar L}_{\delta L}(z)}&0
	\end{bmatrix}},\\
%	\label{GerH2}
\check H_{l}&=\Theta_*\begin{bmatrix}
\mathcal J^{\bar L}_{ C_{l-2n}}-\mathcal J^L_{\phi_{l-2n}}\Xi\mathcal L_* &  \frac{H_{l-2n}}{\bar L_{(-1)^{l+1}n}}\\
\mathcal J^{\bar L}_{ C_{l-2n}}-\mathcal J^L_{\phi_{l-2n+1}}\Xi\mathcal L_* & 0\\
\vdots&\vdots \\
\mathcal J^{\bar L}_{ C_{l}}-\mathcal J^L_{\phi_{l}}\Xi\mathcal L_*&0
\end{bmatrix}=\overline{\Theta_*\begin{bmatrix}
	\mathcal J^{\bar L}_{C_{l-2n}}-{\mathcal J^{L}_{\phi_{l-2n}}} \Xi{\mathcal L_*} & \frac{\bar H_{l-2n}}{\bar L_{(-1)^ln}}\\
	\mathcal J^{\bar L}_{C_{l-2n+1}}-{\mathcal J^{L}_{\phi_{l-2n+1}}} \Xi {\mathcal L_*}&0\\
	\vdots & \vdots \\
	\mathcal J^{\bar L}_{C_{l}}-{\mathcal J^{L}_{\phi_{l}}} \Xi{\mathcal L_*} &0
	\end{bmatrix}}.
\end{align*}
In terms of  $\check \tau_l:=\begin{vsmallmatrix}
\mathcal J^{\bar L}_{ C_{l-2n}}-\mathcal J^L_{\phi_{l-2n}}\Xi\mathcal L_* \\\vdots\\
\mathcal J^{\bar L}_{ C_{l-1}}-\mathcal J^L_{\phi_{l-1}}\Xi\mathcal L_*
\end{vsmallmatrix}$ we can write
\begin{align*}
\check\phi_{l}(z)&=\frac{1}{\check \tau_l}\begin{vmatrix}
\mathcal J^{\bar L}_{ C_{l-2n}}-\mathcal J^L_{\phi_{l-2n}}\Xi\mathcal L_* & \phi_{l-2n}(z)\\\vdots&\vdots \\
\mathcal J^{\bar L}_{ C_{l}}-\mathcal J^L_{\phi_{l}}\Xi\mathcal L_*&\phi_{l}(z)
\end{vmatrix}=-\frac{H_{l-2n}}{ L_{(-1)^ln}}\frac{1}{\bar{\check \tau}_l}\begin{vmatrix}
	\overline{	{ \mathcal J^{\bar L}_{C_{l-2n+1}}-{\mathcal J^{L}_{\phi_{l-2n+1}}}\Xi\mathcal L_*} }\\\vdots
\\\overline{	{ \mathcal J^{\bar L}_{C_{l-1}}-	{\mathcal J^{L}_{\phi_{l-1}}} \Xi	\mathcal L_*}}\\
	{L(z)\Big(\overline{\mathcal J^{\bar L}_{K_{C,\phi}^{[l]}}(z)}-
		{\mathcal J^{L}_{K^{[l]}}
			\Xi\mathcal L_*}\Big)+	\mathcal J^{\bar L}_{\delta L}(z)}
	\end{vmatrix},\\
%	\label{GerH2}
\check H_{l}&=\frac{H_{l-2n}}{\bar L_{(-1)^{l+1}n}}\frac{\check \tau_{l+1}}{\check \tau_l}=\frac{H_{l-2n}}{ L_{(-1)^{l}n}}\frac{\bar{\check \tau}_{l+1}}{\bar{\check \tau}_l}.
\end{align*}

\appendix

\section*{Proofs}
%\subsection{}
%\subsection{Proof of Proposition \ref{ConnectionChristoffelCD}}
\begin{proof}[Proof of Proposition \ref{ConnectionChristoffelCD}]\label{proof:ConnectionChristoffelCD}
We use  the relations
\begin{align*}
\big(\hat{H}^{(1)}\big)^{-1}\omega^{(1)}_1\phi_1(z_2)&=L^{(1)}(z_2)\big(\hat{H}^{(1)}\big)^{-1}\hat\phi^{(1)}_1(z_2),&
\big(\hat{\phi}^{(1)}_2(z_1)\big)^\dagger\big(\hat{H}^{(1)}\big)^{-1}\omega^{(1)}_1&={\phi}{_2^\dagger}(z_1)H^{-1},
\end{align*}
which imply
\begin{gather*}
\begin{multlined}[t][0.9\textwidth]
\Big[\big({\hat\phi_2^{(1)}(z_1)\big)^\dagger}\Big]^{[l]}\left[\big(\hat{H}^{(1)}\big)^{-1}\omega^{(1)}_1\right]^{[l]}[\phi_1(z_2)]^{[l]}+
\Big[\big({\hat\phi_2^{(1)}(z_1)\big)^\dagger}\Big]^{[l]}\left[\big(\hat{H}^{(1)}\big)^{-1}\omega^{(1)}_1\right]^{[l,\geq l]}[\phi_1(z_2)]^{[\geq l]}\\=\Big[\big({\hat\phi_2^{(1)}(z_1)\big)^\dagger}\Big]^{[l]}L^{(1)}(y)\left[\big(\hat{H}^{(1)}\big)^{-1}\right]^{[l]}[\hat{\phi}_1(z_2)]^{[l]},
\end{multlined}\\
\Big[\big({\hat\phi_2^{(1)}(z_1)\big)^\dagger}\Big]^{[l]}\left[\big(\hat{H}^{(1)}\big)^{-1}\omega^{(1)}_1\right]^{[l]}[\phi_1(z_2)]^{[l]}=
\Big[\big({\phi_2(z_1)\big)^\dagger}\Big]^{[l]}\left[H^{-1}\right]^{[l]}[\phi_1(z_2)]^{[l]},
\end{gather*}
 and, after  subtracting and cleaning, you get
\begin{align*}
\Big[\big({\hat\phi_2^{(1)}(z_1)\big)^\dagger}\Big]^{[l]}\left[\big(\hat{H}^{(1)}\big)^{-1}\omega^{(1)}_1\right]^{[l,\geq l]}[\phi_1(z_2)]^{[\geq l]}=
L^{(1)}(z_2)(\hat K^{(1)})^{[l]}(\bar z_1,z_2)-K^{[l]}(\bar z_1,z_2).
\end{align*}

On the other hand, we have $
(\phi_2(z_1))^\dagger(\omega^{(2)}_2)^\dagger\big(\hat{H}^{(2)}\big)^{-1}=
\overline{L^{(2)}(z_1)}\hat\phi^{(2)}_2(z_1)\big(\hat{H}^{(2)}\big)^{-1}$ and $(\omega^{(2)}_2)^\dagger\big(\hat{H}^{(2)}\big)^{-1}\hat \phi^{(2)}_1(z_2)=H^{-1}\phi_1(z_2)$,
and, therefore, we conclude
\begin{gather*}
\begin{multlined}[t][0.9\textwidth]
\left[(\phi_2(z_1))^\dagger\right]^{[l]}\left[(\omega^{(2)}_2)^\dagger\big(\hat{H}^{(2)}\big)^{-1}\right]^{[l]}
\left[\hat \phi^{(2)}_1(z_2)\right]^{[l]}+
\left[(\phi_2(z_1))^\dagger\right]^{[\geq l]}\left[(\omega^{(2)}_2)^\dagger\big(\hat{H}^{(2)}\big)^{-1}\right]^{[\geq l,l]}\left[\hat \phi^{(2)}_1(z_2)\right]^{[l]}\\=
\overline{L^{(2)}(z_1)}\left[\hat\phi^{(2)}_2(x)\right]^{[l]}\left[\big(\hat{H}^{(2)}\big)^{-1}\right]^{[l]}\left[\hat \phi^{(2)}_1(z_2)\right]^{[l]}
\end{multlined},\\
\left[(\phi_2(z_1))^\dagger\right]^{[l]}\left[(\omega^{(2)}_2)^\dagger\big(\hat{H}^{(2)}\big)^{-1}\right]^{[l]}
\left[\hat \phi^{(2)}_1(z_2)\right]^{[l]}
=\left[(\phi_2(z_1))^\dagger\right]^{[l]}
\left[H^{-1}\right]^{[l]}\left[\phi_1(z_2)\right]^{[l]}.
\end{gather*}
Consequently, we arrive to the following expression
\begin{align*}
\left[(\hat \phi^{(2)}_1(z_2))^\dagger\right]^{[l]}
\left[(\bar{\hat{ H}}^{(2)}\big)^{-1}\omega^{(2)}_2\right]^{[l,\geq l]}
\left[\phi_2(z_1)\right]^{[\geq l]}
=
L^{(2)}(z_1)\overline{(\hat K^{(2)})^{[l]}(\bar z_1,z_2)}-\overline{K^{[l]}(\bar z_1,z_2)}.
\end{align*}
\end{proof}
\begin{proof}[Proof of Theorem \ref{Christoffel Formulas}]\label{Theo:Christoffel Formulas}
In the first formula of Proposition \ref{pro:conexión} we take the  $l$-th row and evaluate the spectral jets  along $L(z)$, then
$
\big[(\omega^{(1)}_1)_{l,l},\dots,(\omega^{(1)}_1)_{l,l+2n}\big]
\begin{bsmallmatrix}
		\mathcal J^{L^{(1)}}_{\phi_{1,l}}\\
		\vdots\\
		\mathcal J^{L^{(1)}}_{\phi_{1,l+2n}}
		\end{bsmallmatrix}=0.
$
Since the number of zeros, counting their multiplicities, is $ 2n $, we will have a square matrix
$\begin{bsmallmatrix}
		\mathcal J^{L^{(1)}}_{\phi_{1,l}}\\
		\vdots\\
		\mathcal J^{L^{(1)}}_{\phi_{1,l+2n-1}}
		\end{bsmallmatrix}\in\mathbb C^{2n\times 2n}$ and whenever it is not singular
%$\det\begin{bsmallmatrix}
%		\mathcal J^L_{\phi_{1,k}}\\
%		\vdots\\
%		\mathcal J^L_{\phi_{1,k+2m-1}}
%		\end{bsmallmatrix}\neq 0$,
we obtain
\begin{align*}
\big[(\omega^{(1)}_1)_{l,l},\dots,(\omega^{(1)}_1)_{l,l+2n-1}\big]=-(\omega^{(1)}_1)_{l,l+2n}
\mathcal J^{L^{(1)}}_{\phi_{1,l+2n}}	\begin{bmatrix}\mathcal J^{L^{(1)}}_{\phi_{1,l}}\\
		\vdots\\
		\mathcal J^{L^{(1)}}_{\phi_{1,l+2n-1}}
		\end{bmatrix}^{-1}.
\end{align*}

Using again Proposition  \ref{pro:conexión} we get \eqref{Chris11}. To obtain \eqref{Chris10} we multiply with the vector $\begin{bsmallmatrix}
1\\0\\\vdots\\
0
\end{bsmallmatrix}$.
 Taking \eqref{kernelChristoffelrelacion} and evaluating the spectral jet (on the $z_2$ variable) along  $L^{(1)}(z)$ of the Christoffel--Darboux kernel $K^{[l]}(\bar z_1,z_2)$ we obtain
 \begin{multline*}
 	-\mathcal J_{K^{[l]}}(\bar z)	\begin{bmatrix}
 	\mathcal J^{L^{(1)}}_{\phi_{1,l}} \\
 	\vdots \\
 	\mathcal J^{L^{(1)}}_{\phi_{1,l-1+2n}}\\
 	\end{bmatrix}^{-1}=
 	\overline{\left[\hat\phi_{2,l-2n}^{(1)}(z),\cdots,\hat{\phi}^{(1)}_{2,l-1}(z)\right]}\diag(
 	(\hat{H}^{(1)}_{l-2n})^{-1},\dots, (\hat H^{(1)}_{l-1})^{-1} )\\
 	\times\begin{bmatrix}
 		(\omega^{(1)}_1)_{l-2n,l} &0&0&\dots  & 0 \\
 		(\omega^{(1)}_1)_{l-2n+1,l} &	(\omega^{(1)}_1)_{l-2n+1,l+1} &0&\dots  & 0 \\
 		(\omega^{(1)}_1)_{l-2n+2,l} &(\omega^{(1)}_1)_{l-2n+2,l+1} &(\omega^{(1)}_1)_{l-2n+2,l+2}&\ddots  & 0 \\
 		\vdots & && \ddots & \\
 		(\omega^{(1)}_1)_{l-1,l} &(\omega^{(1)}_1)_{l-1,l+1} &(\omega^{(1)}_1)_{l-1,l+2} &\dots &(\omega^{(1)}_1)_{l-1,l-1+2n}
 	\end{bmatrix},
 \end{multline*}	
and multiplying  by
$\begin{bsmallmatrix}
0\\
\vdots\\
1
\end{bsmallmatrix}$, we get
$
 -\mathcal J^{L^{(1)}}_{K^{[l]}}(\bar z)	\begin{bsmallmatrix}
 \mathcal J^{L^{(1)}}_{\phi_{1,l}} \\
 \vdots \\
 \mathcal J^{L^{(1)}}_{\phi_{1,l-1+2n}}\\
 \end{bsmallmatrix}^{-1}\begin{bsmallmatrix}
 0\\
 \vdots\\
 1
 \end{bsmallmatrix}=
\overline{ \hat{\phi}^{(1)}_{2,l-1}(z)}
(\hat H^{(1)}_{l-1})^{-1}
 (\omega^{(1)}_1)_{l-1,l-1+2n}$.

Analogously, from the last equation in Proposition
\ref{pro:conexión}  we deduce
\begin{align*}
\big[(\omega^{(2)}_2)_{l,l},\dots,(\omega^{(2)}_2)_{l,l+2n-1}\big]=-(\omega^{(2)}_2)_{l,l+2n}
\mathcal J^{L^{(1)}}_{\phi_{1,l+2n}}	\begin{bmatrix}\mathcal J^{L^{(1)}}_{\phi_{1,l}}\\
\vdots\\
\mathcal J ^{L^{(1)}}_{\phi_{1,l+2n-1}}
\end{bmatrix}^{-1}.
\end{align*}
And using the same equation again we conclude \eqref{Chris22}.  Relation \eqref{Chris20}  is obtained just as we got \eqref{Chris10}.
 Finally,  taking \eqref{kernelChristoffelrelacion2} and evaluating the spectral jet (on the $z_1$ variable) along  $L^{(1)}(z)$ of the kernel $\overline{K^l(\bar z_1,z_2)}$, and proceeding as in the proof of \eqref{Chris12} we deduce \eqref{Chris21}.
\end{proof}

\begin{proof}[Proof of Proposition \ref{Geronimus Connection Cauchy}]
The Cauchy second kind functions are
\begin{align*}
\check C_{1}^{(1)}(z)&=\prodint{\check \phi_1^{(1)}(z_1), \frac{1}{\bar z-z_2}}_{\check u^{(1)}},&z&\not\in \overline{\operatorname{supp}_2(\check u^{(1)}),} \\
\check C_{1}^{(2)}(z)&=\prodint{\check \phi_1^{(2)}(z_1), \frac{1}{\bar z-z_2}}_{\check u^{(2)}}
&z&\not\in \overline{\operatorname{supp}_2(\check u^{(2)}),}\\
(	\check C_{2}^{(1)}(z))^\dagger&=\prodint{\frac{1}{\bar z-z_1}, (\check \phi^{(1)}_2(z_2))^\top}_{\check u^{(1)}},& z&\not\in \overline{\operatorname{supp}_1(\check u^{(1)})},\\
(	\check C_{2}^{(2)}(z))^\dagger&=\prodint{\frac{1}{\bar z-z_1}, (\check \phi^{(2)}_2(z_2))^\top}_{\check u^{(2)}},
& z&\not\in \overline{\operatorname{supp}_1(\check u^{(2)})}.
\end{align*}
%donde $\operatorname{supp}(\check u)=\operatorname{supp}( u)\cup \sigma(L)$.
To prove  \eqref{CC12}  we proceed as follows
\begin{align*}
(C_2(z))^\dagger \big(\Omega^{(1)}_2\big)^\dagger - L^{(1)}(\bar z)(\check C^{(1)}_2(z))^\dagger &=
\prodint{\frac{1}{\bar z-z_1},\big(\phi_2(z_2)\big)^\top}_u\big(\Omega^{(1)}_2\big)^\dagger -L^{(1)}(\bar z)
\prodint{\frac{1}{\bar z-z_1},\big(\check \phi^{(1)}_2(z_2)\big)^\top}_{\check u^{(1)}}\\&=
\prodint{\frac{1}{\bar z-z_1},\big(\check\phi^{(1)}_2(z_2)\big)^\top}_{L^{(1)}(z_1)\check u^{(1)}}-
\prodint{\frac{L^{(1)}(\bar z)}{\bar z-z_1},\big(\check \phi^{(1)}_2(z_2)\big)^\top}_{\check u^{(1)}}\\
&=-\prodint{\frac{L^{(1)}(\bar z)-L^{(1)}(z_1)}{\bar z-z_1},\big(\check\phi^{(1)}_2(z_2)\big)^\top}_{\check u^{(1)}}.
\end{align*}
For \eqref{CC21} we have
\begin{align*}
\Omega^{(2)}_1 C_{1}(z)-\check C_{1}^{(2)}(z)\overline{ L^{(2)}}(z)&=\Omega^{(2)}_1\prodint{ \phi_1(z_1), \frac{1}{\bar z-z_2}}_{ u}-\prodint{\check \phi_1^{(2)}(z_1), \frac{1}{\bar z-z_2}}_{\check u^{(2)}}\overline{ L^{(2)}}(z)\\
&=\prodint{ \check\phi^{(2)}_1(z_1), \frac{1}{\bar z-z_2}}_{ \check u^{(2)}\overline{L(z_2)}}
-\prodint{\check \phi_1^{(2)}(z_1),
	\frac{	L^{(2)}( \bar z)}{\bar z-z_2}}_{\check u^{(2)}}\\
&=-\prodint{ \check\phi^{(2)}_1(z_1), \frac{L^{(2)}( \bar z)-L^{(2)}(z_2)}{\bar z-z_2}}_{ \check u^{(2)}}.
\end{align*}
Relation  \eqref{CC11} is a consequence of
\begin{align*}
\Omega^{(1)}_1 \check C^{(1)}_1(z)&=\prodint{L^{(1)}(z_1)\phi_1(z),\frac{1}{\bar z-z_2}}_{\check u^{(1)}}\\
&=\prodint{\phi_1(z),\frac{1}{\bar z-z_2}}_{L^{(1)}(z_1)\check u^{(1)}}
\\
&=C_1(z),
\end{align*}
and, finally, to obtain \eqref{CC22} just observe that
\begin{align*}
\big(\check C^{(2)}_2(z)\big)^\dagger\big(\Omega_2^{(2)}\big)^\dagger&=\prodint{\frac{1}{\bar z-z_1}, L^{(2)}(z)\big(\phi_2(z)\big)^\top}_{\check u^{(2)}}\\=\big(C_2(z)\big)^\dagger.
\end{align*}
\end{proof}
\begin{proof}[Proof of Proposition \ref{Geronimus Connection CD}]
From
			\begin{align*}
(\phi_2(z_1))^\dagger(\Omega^{(1)}_2 )^\dagger\big(\check H^{(1)}\big)^{-1} &=(\check\phi^{(1)}_2(z_1))^\dagger\big(\check H^{(1)}\big)^{-1} ,&
	\big(\Omega^{(1)}_2\big)^\dagger\big(\check H^{(1)}\big)^{-1} \check\phi_1^{(1)}(z_2) = L^{(1)}(z_2)H^{-1}\phi_1(z_2),
			\end{align*}
we deduce that
\begin{align*}
&\left[(\phi_2(z_1)	 )^\dagger\right]^{[l]}\left[(\Omega^{(1)}_2 )^\dagger\big(\check H^{(1)}\big)^{-1} \right]^{[l]}\left[\check \phi^{(1)}_1(z_2)\right]^{[l]}=\left[(\check\phi^{(1)}_2(z_1))^\dagger\right]^{[l]}\left[\big(\check H^{(1)}\big)^{-1}\right]^{[l]}\left[\check \phi^{(1)}_1(z_2)\right]^{[l]},\\
&\begin{multlined}[t][0.9\textwidth]
\left[(\phi_2(z_1))^\dagger\right]^{[l]}\left[	\big(\Omega^{(1)}_2\big)^\dagger\big(\check H^{(1)}\big)^{-1}\right] ^{[l]}\left[\check\phi_1^{(1)}(z_2) \right]^{[l]}+\left[(\phi_2(z_1))^\dagger\right]^{[l]}
\left[	\big(\Omega^{(1)}_2\big)^\dagger\big(\check H^{(1)}\big)^{-1}\right] ^{[l,\geq l]}\left[\check\phi_1^{(1)}(z_2) \right]^{[\geq l]}\\= L^{(1)}(z_2)\left[(\phi_2(z_1))^\dagger\right]^{[l]}\left[H^{-1}\right]^{[l]}\left[\phi_1(z_2) \right]^{[l]},
\end{multlined}
\end{align*}
that once subtracted give
\begin{align*}
\left[(\check\phi_1^{(1)}(z_2) )^\top\right]^{[\geq l]}
\left[\big({\check H}^{(1)}\big)^{-1}	\bar \Omega^{(1)}_2\right] ^{[\geq l,l]}
\left[ \overline{\phi_2(z_1)}\right]^{[l]}
= {L^{(1)}(z_2)K^{[l]}(\bar z_1,z_2)}-{\check K^{(1),[l]}(\bar z_1,z_2)},
\end{align*}
from where we get \eqref{GerKerNor1}.
On the other hand, we have
	\begin{align*}
(	\check H^{(2)} )^{-1}	\Omega_1^{(2)} \phi_1(z_2)&=(	\check H^{(2)} )^{-1}\check\phi^{(2)}_1(z_2), & (\check\phi^{(2)}_2(z_1))^\dagger(	\check H^{(2)} )^{-1}\Omega^{(2)}_1 &=\overline{L^{(2)}(z_1)} (\phi_2(z_1))^\dagger H^{-1},
	\end{align*}
so that
\begin{align*}
&\left[(\check{\phi}_2^{(2)}(z_1))^\dagger\right]^{[l]}\left[(	\check H^{(2)} )^{-1}	\Omega_1^{(2)} \right]^{[l]}\left[\phi_1(z)\right]^{[l]}=
\left[(\check{\phi}_2^{(2)}(z_1))^\dagger\right]^{[l]}\left[(	\check H^{(2)} )^{-1}\right]^{[l]}\left[\check\phi^{(2)}_1(z_2)\right]^{[l]},\\&\begin{multlined}[t][0.9\textwidth]
\left[(\check\phi^{(2)}_2(z_1))^\dagger\right]^{[l]}\left[(	\check H^{(2)} )^{-1}\Omega^{(2)}_1 \right]^{[l]}\left[\phi_1(z_2)\right]^{[l]}+\left[(\check\phi^{(2)}_2(z_1))^\dagger\right]^{[\geq l]}\left[(	\check H^{(2)} )^{-1}\Omega^{(2)}_1 \right]^{[\geq l,l]}\left[\phi_1(z_2)\right]^{[l]}\\ =\overline{L^{(2)}(z_1)} \left[(\phi_2(z_1))^\dagger\right] ^{[l]}\left[H^{-1}\right]^{[l]}\left[\phi_1(z_2)\right]^{[l]},
\end{multlined}
\end{align*}
which lead to
\begin{align*}
\left[(\check\phi^{(2)}_2(z_1))^\dagger\right]^{[\geq l]}\left[(	\check H^{(2)} )^{-1}\Omega^{(2)}_1 \right]^{[\geq l,l]}\left[\phi_1(z_2)\right]^{[l]}=\overline{L^{(2)}(z_1)} K^{[l]}(\bar z_1,z_2)-\check K^{(2),[k]}(\bar z_1,z_2),
\end{align*}
and \eqref{GerKerNor2} is immediately deduced.
\end{proof}

\begin{proof}[Proof of Proposition \ref{Geronimus Connection mixed}]
Let's prove \eqref{mix CF12}. 	From the relations
	\begin{align*}
	(C_2(x_1))^\dagger \big(\Omega^{(1)}_2\big)^\dagger \big(\check H^{(1)}\big)^{-1}- L^{(1)}(\bar z_1)(\check C^{(1)}_2(x_1))^\dagger \big(\check H^{(1)}\big)^{-1}&=-\prodint{\delta L^{(1)}(\bar x_1,z_1),\big(\check\phi^{(1)}_2( z_2)\big)^\top}_{\check u^{(1)}}\big(\check H^{(1)}\big)^{-1},
	\end{align*}
		\begin{align*}
	\big(\Omega^{(1)}_2\big)^\dagger(\check H^{(1)})^{-1} \check\phi_1^{(1)}(x_2) &= L^{(1)}(x_2)H^{-1}\phi_1(x_2),
		\end{align*}
we obtain
\begin{align*}
&\begin{multlined}[t][0.9\textwidth]
\left[	(C_2(x_1))^\dagger \right]^{[l]}\left[\big(\Omega^{(1)}_2\big)^\dagger\big(\check H^{(1)}\big)^{-1} \right]^{[l]}	[\check \phi^{(1)}_{1}(x_2)]^{[l]}
-L^{(1)}(\bar z_1)\left[(\check C^{(1)}_2(x_1))^\dagger\right]^{[l]}\left[\big(\check H^{(1)}\big)^{-1}\right]^{[l]}	[\check \phi^{(1)}_{1}(x_2)]^{[l]}\\=
-\prodint{\delta L^{(1)}(\bar x_1,z_1),\big(\left[\check\phi^{(1)}_2( z_2)\right]^{[l]}\big)^\top}_{\check u^{(1)}}\left[\big(\check H^{(1)}\big)^{-1}\right]^{[l]}	[\check \phi^{(1)}_{1}(x_2)]^{[l]},
\end{multlined}\\
&\begin{multlined}[t][0.9\textwidth]
[C_{2}(x_1)^{\dagger}]^{[l]}\left[	\big(\Omega^{(1)}_2\big)^\dagger(\check H^{(1)})^{-1} \right]^{[l]}\left[\check\phi_1^{(1)}(x_2) \right]^{[l]}+[C_{2}(x_1)^{\dagger}]^{[l]}
\left[	\big(\Omega^{(1)}_2\big)^\dagger(\check H^{(1)})^{-1} \right]^{[l,\geq l]}\left[\check\phi_1^{(1)}(x_2)\right] ^{[\geq l]}\\
= L^{(1)}(x_2)[C_{2}(x_1)^{\dagger}]^{[l]}\left[H^{-1}\right]^{[l]}\left[\phi_1(x_2)\right]^{[l]}.
\end{multlined}
\end{align*}
When subtracted we get
	\begin{align*}
	\begin{multlined}[t][0.9\textwidth]
	\left[(\check\phi_1^{(1)}(x_2))^\top\right] ^{[\geq l]}
	\left[	({\check H}^{(1)})^{-1}\big(\bar \Omega^{(1)}_2\big) \right]^{[\geq l,l]}[\overline{C_{2}(x_1)}]^{[l]}+	L^{(1)}(\bar x_1)\check K_{C,\phi}^{(1),[l]}(\bar x_1,x_2)\\=
\prodint{\delta L^{(1)}(\bar x_1,z_1), \overline{\check K^{(1),[l]}(\bar z_2,x_2)}}_{\check u^{(1)}}
%	\prodint{\check u_{z_1,\bar z_2},\frac{L(\bar x_1)-L(z_1)}{\bar x_1- z_1}\otimes \check K^{(1),[l]}(z_2^{-1},x_2)}
	+ L^{(1)}(x_2)K_{C,\phi}^{[l]}(\bar x_1,x_2).
	\end{multlined}
	\end{align*}
	From Corollary \ref{CDproyeccion} and  Proposition \ref{simetricos} we obtain
	\begin{align*}
	\prodint{\frac{L^{(1)}(\bar x_1)-L^{(1)}(z_1)}{\bar x_1- z_1}, \overline{\check K^{(1),[l]}(\bar z_2,x_2)}}_{\check u^{(1)}}&=\frac{L^{(1)}(\bar x_1)-L^{(1)}(x_2)}{\bar x_1- x_2},
%	\\
%	\prodint{ \check K^{(2),[l]}(\bar x_1,z_1), \frac{L^{(2)}(  \bar x_2)- L^{(2)}(z_2)}{ \bar x_2-z_2}}_{\check u^{(2)}}&=\frac{\overline{ L^{(2)}}(\bar x_1)-\overline{ L^{(2)}}(   x_2)}{  \bar x_1-x_2}.
	\end{align*}
	and the desired result follows.
	%and introduce it into \eqref{mix CF12}.% and \eqref{mix FC2}.

For the proof of \eqref{mix CF2}, we notice that from
\begin{align*}
	\big(\check C^{(2)}_2(x_1)\big)^\dagger (\check H^{(2)} )^{-1}\Omega^{(2)}_1 &=\big(C_2(x_1)\big)^\dagger H^{-1},&
	(\check H^{(2)})^{-1}\Omega_1^{(2)} \phi_1(x_2)&=(\check H^{(2)})^{-1}\check\phi^{(2)}_1(x_2),
\end{align*}
we get
\begin{align*}
&\begin{multlined}[t][0.9\textwidth]\left[	\big(\check C^{(2)}_2(x_1)\big)^\dagger\right] ^{[l]}\left[(\check H^{(2)} )^{-1}\Omega^{(2)}_1\right]^{[l]}\left[\phi_1(x_2)\right]^{[l]}+
\left[	\big(\check C^{(2)}_2(x_1)\big)^\dagger\right] ^{[\geq l]}	\left[(\check H^{(2)})^{-1}\Omega_1^{(2)}\right]^{[\geq l,l]} \left[\phi_1(x_2)\right]^{[ l]}\\=\left[\big(C_2(x_1)\big)^\dagger\right]^{[l]} \left[H^{-1}\right]^{[l]}\left[\phi_1(x_2)\right]^{[l]},
\end{multlined}\\
&
\left[	\big(\check C^{(2)}_2(x_1)\big)^\dagger\right] ^{[l]}\left[	(\check H^{(2)})^{-1}\Omega_1^{(2)}\right]^{[l]} \left[\phi_1(x_2)\right]^{[l]}
=\left[	\big(\check C^{(2)}_2(x_1)\big)^\dagger\right] ^{[l]}\left[(\check H^{(2)})^{-1}\right]^{[l]}\left[\check\phi^{(2)}_1(x_2)\right]^{[l]},
\end{align*}
and, consequently, we conclude
\begin{align*}
 K_{C,\phi}^{[l]}(\bar x_1,x_2)-
\left[	\big(\check C^{(2)}_2(x_1)\big)^\dagger\right] ^{[\geq l]}	\left[(\check H^{(2)})^{-1}\Omega_1^{(2)}\right]^{[\geq l, l]} \left[\phi_1(x_2)\right]^{[ l]}=\check K_{C,\phi}^{(2),[l]}(\bar x_1,x_2).
\end{align*}

To prove \eqref{mix FC1} observe that
\begin{align*}
 \big(\Omega^{(1)}_2\big)^\dagger(\check H^{(1)})^{-1}\check C^{(1)}_1(x_2)&=
 H^{-1}C_1(x_2), &
 (\phi_2(x_1))^\dagger
 \big(\Omega^{(1)}_2 \big)^\dagger({\check H}^{(1)})^{-1}&=\big(\check\phi^{(1)}_2(x_1)\big)^\dagger({\check H}^{(1)})^{-1},
\end{align*}
imply
\begin{align*}
&\begin{multlined}[t][0.9\textwidth]
\left[(\phi_2(x_1))^\dagger\right]^{[l]}   \left[\big(\Omega^{(1)}_2\big)^\dagger(\check H^{(1)})^{-1}\right]^{[l]}\left[\check C^{(1)}_1(x_2)\right]^{[l]}+
\left[(\phi_2(x_1))^\dagger\right]^{[l]}    \left[\big(\Omega^{(1)}_2\big)^\dagger(\check H^{(1)})^{-1}\right]^{[l,\geq l]}\left[\check C^{(1)}_1(z)\right]^{[\geq l]}\\=
\left[(\phi_2(x_1))^\dagger\right]^{[l]}   \left[ H^{-1}\right]^{[l]}\left[C_1(x_2)\right]^{[l]}
\end{multlined},\\
&
\left[(\phi_2(x_1))^\dagger\right]^{[l]}\left[ \big(\Omega^{(1)}_2 \big)^\dagger({\check H}^{(1)})^{-1}\right]^{[l]}\left[\check C^{(1)}_{1}(x_2)\right]^{[l]}=
\left[\big(\check\phi^{(1)}_2(x_1)\big)^\dagger\right]^{[l]}\left[(\bar{\check H}^{(1)})^{-1}\right]^{[l]}
\left[\check C^{(1)}_{1}(x_2)\right]^{[l]}
\end{align*}
whose difference is
\begin{align*}
-\left[(\check C^{(1)}_1(x_2))^\top\right]^{[\geq l]}
 \left[(\check H^{(1)})^{-1}\bar\Omega^{(1)}_2\right]^{[\geq l,l]}\left[\overline{(\phi_2(x_1))}\right]^{[l]}   =
\check K_{\phi,C}^{(1),[l]}(\bar x_1,x_2)-K_{\phi,C}^{[l]}(\bar x_1,x_2).
\end{align*}

Finally, just consider
\begin{align*}
(\check H^{(2)})^{-1}	\Omega^{(2)}_1 C_{1}(x_2)-\overline{ L^{(2)}}(x_2)(\check H^{(2)})^{-1}	\check C_{1}^{(2)}(x_2)&=-(\check H^{(2)})^{-1}	\prodint{ \check\phi^{(2)}_1(z_1), \delta L^{(2)}(\bar x_2,z_2)}_{ \check u^{(2)} },\\
(\check\phi^{(2)}_2(x_1))	^\dagger (\check H^{(2)})^{-1}\Omega^{(2)}_1&=\overline{L^{(2)}(x_1)} (\phi_2(x_1))^\dagger H^{-1},\\
\end{align*}
that lead to
\begin{align*}
&\begin{multlined}[t][0.9\textwidth]
\left[(\check\phi^{(2)}_2(x_1))^\dagger\right]^{[l]}\left[(\check H^{(2)})^{-1}\Omega^{(2)}_1 \right]^{[l]}\left[C_{1}(x_2)\right]^{[l]}-\overline{ L^{(2)}}(x_2)\left[(\check\phi^{(2)}_2(x_1))^\dagger\right]^{[l]}\left[(\check H^{(2)})^{-1}\right]^{[l]}\left[\check C_{1}^{(2)}(x_2)\right]^{[l]}\\=-\left[(\check\phi^{(2)}_2(x_1))^\dagger\right]^{[l]}\left[(\check H^{(2)})^{-1}	\right]^{[l]}\prodint{ \left[\check\phi^{(2)}_1(z_1)\right]^{[l]},\delta L^{(2)}(\bar x_2,z_2)}_{ \check u^{(2)} },
\end{multlined}\\
&\begin{multlined}[t][0.9\textwidth]
\left[(\check\phi^{(2)}_2(x_1))	^\dagger \right]^{[l]}\left[(\check H^{(2)})^{-1}\Omega^{(2)}_1\right]^{[l]}\left[C_{1}(x_2)\right]^{[l]}+\left[(\check\phi^{(2)}_2(x_1))	^\dagger \right]^{[\geq l]}\left[(\check H^{(2)})^{-1}\Omega^{(2)}_1\right]^{[\geq l,l]}\left[C_{1}(x_2)\right]^{[l]}
\\=\overline{L^{(2)}(x_1)}\left[ (\phi_2(x_1))^\dagger\right]\left[ H^{-1}\right]^{[l]}\left[C_{1}(x_2)\right]^{[l]},
\end{multlined}
\end{align*}
that once subtracted provide us with the relation
\begin{multline*}
\left[(\check\phi^{(2)}_2(x_1))^\dagger \right]^{[\geq l]}\left[(\check H^{(2)})^{-1}\Omega^{(2)}_1\right]^{[\geq l,l]}\left[C_{1}(x_2)\right]^{[l]}
=\overline{L^{(2)}(x_1)} K_{\phi,C}^{[l]}(\bar x_1,x_2)-\overline{ L^{(2)}}(x_2)\check K_{\phi,C}^{(2),[l]}(\bar x_1,x_2)
\\+\prodint{ \check K^{(2),[l]}(\bar x_1,z_1), \delta L^{(2)}(\bar x_2,z_2)}_{\check u^{(2)}}.
\end{multline*}
Recall Corollary \ref{CDproyeccion} and  Proposition \ref{simetricos} we have
\begin{align*}
%\prodint{\frac{L^{(1)}(\bar x_1)-L^{(1)}(z_1)}{\bar x_1- z_1}, \overline{\check K^{(1),[l]}(\bar z_2,x_2)}}_{\check u^{(1)}}&=\frac{L^{(1)}(\bar x_1)-L^{(1)}(x_2)}{\bar x_1- x_2},\\
\prodint{ \check K^{(2),[l]}(\bar x_1,z_1), \frac{L^{(2)}(  \bar x_2)	- L^{(2)}(z_2)}{ \bar x_2-z_2}}_{\check u^{(2)}}&=\frac{
	\overline{ L^{(2)}}(\bar x_1)-\overline{ L^{(2)}}(   x_2)}{  \bar x_1-x_2}.
\end{align*}
and we get the result.
%and introduce it into \eqref{mix CF12} and \eqref{mix FC2}.
\end{proof}
\begin{proof}[Proof of Proposition \ref{restos}]
To get \eqref{caleGer}, observe that for  $l=0,\dots,m^{(2)}_{j}-1$ we have
$	\frac{\d^l}{\d z^l}\bigg|_{z=\bar \zeta^{(2)}_{j}}
	\prodint{\check\phi^{(2)}_{1,k}(z_1),\frac{L^{(2)}(\bar z)}{\bar z-z_2}}_{u(\overline{L^{(2)}(z_2)})^{-1}}=0$.
	This holds because
	\begin{gather*}
	\frac{\d^l}{\d z^l}\bigg|_{ z=\bar\zeta^{(2)}_{j}}
	\prodint{\check\phi^{(2)}_{1,k}(z_1),\frac{L^{(2)}(\bar z)}{\bar z-z_2}}_{u(\overline{L^{(2)}(z_2)})^{-1}}=
	\prodint{u_{z_1,\bar z_2},\check\phi^{(2)}_{1,k}(z_1)\otimes(\overline{L^{(2)}(z_2)})^{-1}\frac{\d^l}{\d z^l}\bigg|_{z=\bar \zeta^{(2)}_{j}}\frac{\overline{ L^{(2)}}( z)}{z-\bar z_2}}\\
	=\sum_{r=0}^{l}\binom{l}{r}\prodint{u_{z_1,\bar z_2},\check\phi^{(2)}_{1,k}(z_1)\otimes(\overline{L^{(2)}(z_2)})^{-1}\left(\frac{\d^r	\overline{ L^{(2)}}( z)}{\d z^r}\bigg|_{z=\bar\zeta^{(2)}_{j}}
		\frac{\d^{l-r}}{\d z^{l-r}}\bigg|_{z=\bar \zeta^{(2)}_{j}}\frac{1}{z-\bar z_2}\right)}
	\end{gather*}
	but $\dfrac{\d^r 	\overline{ L^{(2)}}( z)}{\d z^r}\bigg|_{z=\bar\zeta^{(2)}_{j}}=0$ for  $r\in\{0,\dots,m^{(2)}_{j}-1\}$, and
	since $\operatorname{supp}_2(u)\cap \overline{\sigma(L^{(2)})}=\varnothing$, we get the result. Therefore, we conclude that
	\begin{align*}
	\frac{\d^l}{\d z^l}\bigg|_{z=\bar \zeta^{(2)}_{j}}\overline{ L^{(2)}}(z)\check C_{1,k}^{(2)}(z)&=\sum_{i=1}^{d^{(2)}}\prodint{\xi^{(2)}_{i},	\check\phi^{(2)}_{1,k}}\frac{\d^l}{\d z^l}\bigg|_{z=\bar \zeta^{(2)}_{j}}	 \overline{L^{(2)}_{[i]}}(z)\begin{bmatrix}
	\big(\bar z- \zeta^{(2)}_{i}\big)^{m^{(2)}_{i}-1}\\ \big(\bar z- \zeta^{(2)}_{i}\big)^{m^{(2)}_{1}-2}\\\vdots\\1
	\end{bmatrix}.
	\end{align*}
For $l>0$,  we have
	\begin{align*}
	\frac{\d^l}{\d z^l}\bigg|_{z=\bar\zeta^{(2)}_{j}}	 \overline{L^{(2)}_{[i]}}(z)\big(z-\bar \zeta^{(2)}_{i}\big)^{m}
	&=\sum_{r=0}^{m}\binom{l}{m-r} \Big(\frac{\d^{l-m+r}}{\d z^{l-m+r}}\bigg|_{z=\bar\zeta^{(2)}_{j}}	\overline{L^{(2)}_{[i]}}(z)\Big)
	\frac{m!}{(m-r)!}\Big(\overline{\zeta^{(2)}_{j}}- \overline{\zeta^{(2)}_{i}}\Big)^{r}\\
	&=\sum_{r=0}^{m}
	\sum_{s=1}^{l-m+r}
	\frac{l!m!}{s!(m-r)!}
	\overline{\big(L^{(2)}_{[i]}\big)^{(s)}(\zeta^{(2)}_{j})}
\Big(\overline{\zeta^{(2)}_{j}}- \overline{\zeta^{(2)}_{i}}\Big)^{r}.
	\end{align*}		
	But, if $i\neq j$,  $\big(L^{(2)}_{[i]}\big)^{(s)}(\zeta^{(2)}_{j})=0$ for $s\in\big\{0,1,\dots,m^{(2)}_j-1\big\}$, which is our case because  $l\in\big\{0,1,\dots,m^{(2)}_{j}-1\big\}$; when $i=j$ we get that only  terms with $ r = 0 $ will survive and, therefore, $m\leq l$ with
$
	\frac{1}{l!}	\frac{\d^l}{\d z^l}\bigg|_{z=\bar\zeta^{(2)}_{j}}	 \overline{L^{(2)}_{[j]}}(z)\big(z-\bar \zeta^{(2)}_{j}\big)^{m}
	=\ell^{(2)}_{[j],l-m}$.

To show \eqref{chorro (1)}	let's compute
	$\dfrac{1}{l!}\dfrac{\d^l}{\d \bar z^l}\bigg|_{\bar z=\zeta^{(1)}_{j}}\overline{	\bar L^{(1)}(z)\check C_{2,k}^{(1)}(z)}$ para $l\in\{0,1,\dots,m^{(1)}_{j}-1\}$.  For that aim we evaluate
	\begin{multline*}
	\dfrac{\d^l}{\d \bar z^l}\bigg|_{\bar z=\zeta^{(1)}_{j}}\prodint{\frac{L^{(1)}(\bar z)}{\bar z-z_1}, \check \phi^{(1)}_{2,k}(z_2)}_{ (L^{(1)}(z_1))^{-1}u}=
	\prodint{\dfrac{\d^l}{\d \bar z^l}\bigg|_{\bar z=\zeta^{(1)}_{j}}\frac{L^{(1)}(\bar z)}{\bar z-z_1}, \check \phi^{(1)}_{2,k}(z_2)}_{ (L(z_1))^{-1}u}\\
	=\sum_{r=0}^l\binom{l}{r} (l-r)!\dfrac{\d^{r} L^{(1)}(\bar z)}{\d \bar z^{r}}\bigg|_{\bar z=\zeta^{(1)}_{j}}(-1)^{l-r}\prodint{\frac{1}{(\zeta^{(1)}_{j}-z_1)^{l-r+1}}, \check \phi^{(1)}_{2,k}(z_2)}_{ (L^{(1)}(z_1))^{-1}u},
	\end{multline*}
that,  remembering that the zeros are not in the support of the linear functional, vanishes. Finally we realize, that
	\begin{align*}
	\frac{1}{l!}\dfrac{\d^l}{\d \bar z^l}\bigg|_{\bar z=\zeta^{(1)}_{j}}L^{(1)}_{[i]}(\bar z) \big(\bar z-\zeta^{(1)}_i\big)^m&=
	\sum_{r=0}^{l} \binom{l}{r}\dfrac{\d^rL^{(1)}_{[i]}(\bar z)}{\d \bar z^r}\bigg|_{\bar z=\zeta^{(1)}_{j}}\frac{m!}{(l-r)!}\big(\zeta^{(1)}_{j}-\zeta^{(1)}_{i}\big)^{m-l+r}\\
	&=\begin{cases}
	0, & i\neq j,\\
	{	\bar \ell^{(1)}_{[j],l-m}},& i=j.
	\end{cases}
	\end{align*}
\end{proof}

\begin{proof}[Proof of Theorem \ref{Christoffel-Geronimus formulas}]
To show \eqref{Ger12} we proceed as follows. From  \eqref{CC21k>} we get
	\begin{align*}
	\mathcal J^{ \overline{L^{(2)}}} _{{ \overline{L^{(2)}}} \check C^{(2)}_{{1,l}}} &=\big[
	(\Omega^{(2)}_1)_{l,l-2n},\dots,(\Omega^{(2)}_1)_{l,l-1} \big]
	\begin{bmatrix}
	\mathcal J^{ \overline{L^{(2)}}} _{ C_{1,l-2n}}\\\vdots \\ \mathcal J^{ \overline{L^{(2)}}} _{C_{1,l-1}}
	\end{bmatrix}+\mathcal J^{ \overline{L^{(2)}}} _{C_{1,l}}.
	\end{align*}
	Using \eqref{caleGer} and
	\begin{align}\label{conexion}
	(\Omega_1^{(2)})_{l,l-2n}
	\phi_{1,l-2n}(z)+\dots +(\Omega_1^{(2)})_{l,l-1}\phi_{1,l-1}(z)+\phi_{1,l}(z)=\check\phi^{(2)}_1(z),
	\end{align}
	see Proposition \ref{Geronimus Connection Laurent}, we deduce that
	\begin{align*}
	\mathcal J^{ \overline{L^{(2)}}} _{{ \overline{L^{(2)}}} \check C^{(2)}_{{1,l}}} &=\big[
	(\Omega^{(2)}_1)_{l,l-2n},\dots,(\Omega^{(2)}_1)_{l,l-1} \big]
	\begin{bmatrix}
	\prodint{\xi^{(2)},\phi_{1,l-2n}}\\\vdots \\ \prodint{\xi^{(2)},\phi_{1,l-1}}
	\end{bmatrix}
	\mathcal L^{(2)}+\prodint{\xi^{(2)},\phi_{1,l}}\mathcal L^{(2)},
	\end{align*}
	and, consequently,
	\begin{align*}
	\big[
	(\Omega^{(2)}_1)_{l,l-2n},\dots,(\Omega^{(2)}_1)_{l,l-1} \big]=-\Big(\mathcal J^{ \overline{L^{(2)}}} _{C_{1,l}}-\prodint{\xi^{(2)},\phi_{1,l}}\mathcal L^{(2)}\Big)\begin{bmatrix}
	\mathcal J^{ \overline{L^{(2)}}} _{ C_{1,l-2n}}-\prodint{\xi^{(2)},\phi_{1,l-2n}}\mathcal L^{(2)}\\\vdots \\
	\mathcal J^{ \overline{L^{(2)}}} _{ C_{1,l-1}}-\prodint{\xi^{(2)},\phi_{1,l-1}}\mathcal L^{(2)}
	\end{bmatrix}^{-1}.
	\end{align*}
	Recalling  \eqref{conexion}  we conclude with the proof of this Christoffel formula. From this very equation, we obtain
	\begin{align*}
	(\Omega^{(2)}_1)_{l,l-2n}=-\big(\mathcal J^{ \overline{L^{(2)}}} _{C_{1,l}}-\prodint{\xi^{(2)},\phi_{1,l}}\mathcal L^{(2)}\big)\begin{bmatrix}
	\mathcal J^{ \overline{L^{(2)}}} _{ C_{1,l-2n}}-\prodint{\xi^{(2)},\phi_{1,l-2n}}\mathcal L^{(2)}\\\vdots \\
	\mathcal J^{ \overline{L^{(2)}}} _{ C_{1,l-1}}-\prodint{\xi^{(2)},\phi_{1,l-1}}\mathcal L^{(2)}
	\end{bmatrix}^{-1}\begin{bmatrix}
	1\\0\\\vdots\\0
	\end{bmatrix}
	\end{align*}
and recalling \eqref{Omegalambda} we get \eqref{GerH1}.
	
For the proof of \eqref{Ger21}, realize that the relation  \eqref{CC12k>}	gives
	\begin{align}
	\Big[(\Omega^{(1)}_2)_{l,l-2n} , \dots, (\Omega^{(1)}_2)_{l,l-1} \Big]\begin{bmatrix}
	\mathcal J^{ \overline{L^{(1)}}} _{C_{2,l-2n}}\\\vdots\\\mathcal J^{ \overline{L^{(1)}}} _{C_{2,l-2n}}
	\end{bmatrix}+\mathcal J^{ \overline{L^{(1)}}} _{C_{2,l-1}}=\mathcal J^{ \overline{L^{(1)}}}_ {{ \overline{L^{(1)}}}  \check C^{(1)}_{2,l}}.
	\end{align}
On the other hand,  from Proposition \ref{Geronimus Connection Laurent} we know that
	\begin{align*}
	(\Omega^{(1)}_2)_{l,l-2n}
	\phi_{2,l-2n}(z)+\dots
	+(\Omega^{(1)}_2)_{l,l-1}
	\phi_{2,l-1}(z)+\phi_{2,l}(z)=\check \phi_{2,l}^{(1)}(z).
	\end{align*}
	Now, taking into account \eqref{chorro (1)} we conclude
	\begin{align}
	\Big[(\Omega^{(1)}_2)_{l,l-2n} , \dots, (\Omega^{(1)}_2)_{l,l-1} \Big]=-\Big(\mathcal J^{ \overline{L^{(1)}}} _{C_{2,l-1}}-{\prodint{\xi^{(1)},	\phi_{2,l}}}{\mathcal L^{(1)}}\Big)\begin{bmatrix}
	\mathcal J^{ \overline{L^{(1)}}} _{C_{2,l-2n}}-{\prodint{\xi^{(1)},	\phi_{2,l-2n}}}{\mathcal L^{(1)}}\\\vdots\\\mathcal J^{ \overline{L^{(1)}}} _{C_{2,l-1}}
	-{\prodint{\xi^{(1)},	\phi_{2,l-1}}}{\mathcal L^{(1)}}\end{bmatrix}^{-1},
	\end{align} from where the Christoffel-Geronimus formula follows.
	
This equation also implies that
	 \begin{align}
	 (\Omega^{(1)}_2)_{l,l-2n}=-\Big(\mathcal J^{ \overline{L^{(1)}}} _{C_{2,l-1}}-{\prodint{\xi^{(1)},	\phi_{2,l}}}{\mathcal L^{(1)}}\Big)\begin{bmatrix}
	 \mathcal J^{ \overline{L^{(1)}}} _{C_{2,l-2n}}-{\prodint{\xi^{(1)},	\phi_{2,l-2n}}}{\mathcal L^{(1)}}\\\vdots\\\mathcal J^{ \overline{L^{(1)}}} _{C_{2,l-1}}
	 -{\prodint{\xi^{(1)},	\phi_{2,l-1}}}{\mathcal L^{(1)}}\end{bmatrix}^{-1}\begin{bmatrix}
	 1\\0\\\vdots\\0
	 \end{bmatrix},
	 \end{align}
	 and recalling \eqref{Omegalambda} we get \eqref{GerH2}.

We will prove  \eqref{Ger11} y \eqref{Ger22} simultaneously.
Let's write   \eqref{mix CF12}  and  \eqref{mix FC22}  as follows
		\begin{align*}
		%\label{mix CF12}	
		&\begin{multlined}[t][0.9\textwidth]
	\sum_{k=0}^{l-1}\overline{{ \overline{L^{(1)}}} (x_1)\check C^{(1)}_{2,k}(x_1)}(\check H^{(1)}_{k})^{-1}\check\phi^{(1)}_{1,k}(x_2)-	L^{(1)}(x_2)K_{C,\phi}^{[l]}(\bar x_1,x_2)-
		\delta L^{(1)}(\bar x_1,x_2)
		\\=
		-{\Big[\check \phi^{(1)}_{1,l}(x_2),\dots,
			\check \phi^{(1)}_{1,l+2n-1}(x_2)\Big]}
({\check H}^{(1)}_{[n,l]})^{-1}
		\bar\Omega^{(1)}_2[n,l]
		\begin{bmatrix}
		\overline{ C_{2,l-2n}(x_1)}\\ \vdots\\  \overline{
			C_{2,l-1}(x_1)}
		\end{bmatrix},
		\end{multlined}	
\\
			&\begin{multlined}[t][0.9\textwidth]
			\sum_{k=0}^{l-1}\overline{\check \phi^{(2)}_{2,k}(x_1)}(\check H^{(2)}_{k})^{-1}\overline{ L^{(2)}}(x_2)\check C^{(2)}_{1,k}(x_2)-	\overline{ L^{(2)}}(\bar x_1) K_{\phi,C}^{[l]}(\bar x_1,x_2)
			-\delta \overline{ L^{(2)}}(\bar x_1,x_2)
			\\=
			-\Big[\overline{\check \phi^{(2)}_{2,l}(x_1)},\dots, \overline{\check \phi^{(2)}_{2,l+2n-1}(x_1)}\Big]
			(\check H^{(2)}[n,l])^{-1}
			\Omega^{(2)}_1[n,l]
			\begin{bmatrix}
			C_{1,l-2n}(x_2)\\\vdots\\C_{1,l-1}(x_2)
			\end{bmatrix}.
			\end{multlined}
			\end{align*}
and let's compute the spectral jets, with respect to the first and second variables, respectively
				\begin{align*}
				&\begin{multlined}[t][0.9\textwidth]
	\sum_{k=0}^{l-1}(\check H^{(1)}_{k})^{-1}\check\phi^{(1)}_{1,k}(x_2)
				\overline{\mathcal J^{ \overline{L^{(1)}}} _{{ \overline{L^{(1)}}} \check C^{(1)}_{2,k}}}
				-	L^{(1)}(x_2)\mathcal J^{ \overline{L^{(1)}}} _{K_{C,\phi}^{[l]}}(x_2)-
			\mathcal J^{\bar L^{(1)}}_{\delta L^{(1)}}(x_2)
				\\=
				-{\Big[\check \phi^{(1)}_{1,l}(x_2),\dots,
					\check \phi^{(1)}_{1,l+2n-1}(x_2)\Big]}
				({\check H}^{(1)}_{[n,l]})^{-1}
				\bar\Omega^{(1)}_2[n,l]
				\begin{bmatrix}
				\overline{ \mathcal J^{ \overline{L^{(1)}}} _{C_{2,l-2n}}}\\ \vdots\\  \overline{
				\mathcal J^{ \overline{L^{(1)}}} _{	C_{2,l-1}}}
				\end{bmatrix},
				\end{multlined}	
				\end{align*}
				\begin{align*}
				&\begin{multlined}[t][0.9\textwidth]
				\sum_{k=0}^{l-1}\overline{\check \phi^{(2)}{_{2,k}}(x_1)}(\check H^{(2)}_{k})^{-1}\mathcal  J^{ \overline{L^{(2)}}} _{{ \overline{L^{(2)}}} \check C^{(2)}_{1,k}}-	\overline{ L^{(2)}}(\bar x_1) \mathcal  J^{{ \overline{L^{(2)}}} }_{K_{\phi,C}^{[l]}}(\bar x_1)
				-\mathcal  J^{ \overline{L^{(2)}}} _{\delta { \overline{L^{(2)}}} }(\bar x_1)
				\\=
				-\Big[\overline{\check \phi^{(2)}_{2,l}(x_1)},\dots, \overline{\check \phi^{(2)}_{2,l+2n-1}(x_1)}\Big]
				(\check H^{(2)}[n,l])^{-1}
				\Omega^{(2)}_1[n,l]
				\begin{bmatrix}
			\mathcal J^{{ \overline{L^{(2)}}} }_{	C_{1,l-2n}}\\\vdots\\\mathcal J^{{ \overline{L^{(2)}}} }_{C_{1,l-1}(x_2)}
				\end{bmatrix}.
				\end{multlined}
				\end{align*}
			From \eqref{caleGer} and \eqref{chorro (1)} we deduce
						\begin{align*}
						%\label{mix CF12}	
						&\begin{multlined}[t][0.9\textwidth]
						\sum_{k=0}^{l-1}(\check H^{(1)}_{k})^{-1}\check\phi^{(1)}_{1,k}(x_2)
					\overline{\prodint{\xi^{(1)},	\check\phi^{(1)}_{2,k}}	\mathcal L_1^{(1)}}
						-	L^{(1)}(x_2)\mathcal J^{ \overline{L^{(1)}} }_{K_{C,\phi}^{[l]}}(x_2)-
						\mathcal J^{{ \overline{L^{(1)}}} }_{\delta L}(x_2)
						\\=
						-{\Big[\check \phi^{(1)}_{1,l}(x_2),\dots,
							\check \phi^{(1)}_{1,l+2n-1}(x_2)\Big]}
						({\check H}^{(1)}_{[n,l]})^{-1}
						\bar\Omega^{(1)}_2[n,l]
						\begin{bmatrix}
						\overline{ \mathcal J^{ \overline{L^{(1)}}}_{C_{2,l-2n}}}\\ \vdots\\  \overline{
							\mathcal J^{ \overline{L^{(1)}}} _{	C_{2,l-1}}}
						\end{bmatrix},
						\end{multlined}	
						\end{align*}
						\begin{align*}%\label{mix FC22}	
						&\begin{multlined}[t][0.9\textwidth]
						\sum_{k=0}^{l-1}\overline{\check \phi^{(2)}_{2,k}(x_1)}(\check H^{(2)}_{k})^{-1}
					\prodint{\xi^{(2)},\check\phi^{(2)}_{1,k}}\mathcal L^{(2)}
						-	\overline{ L^{(2)}}(\bar x_1) \mathcal  J^{{ \overline{L^{(2)}}} }_{K_{\phi,C}^{[l]}}(\bar x_1)
						-\mathcal  J^{{ \overline{L^{(2)}}} }_{\delta { \overline{L^{(2)}}} }(\bar x_1)
						\\=
						-\Big[\overline{\check \phi^{(2)}_{2,l}(x_1)},\dots, \overline{\check \phi^{(2)}_{2,l+2n-1}(x_1)}\Big]
						(\check H^{(2)}[n,l])^{-1}
						\Omega^{(2)}_1[n,l]
						\begin{bmatrix}
						\mathcal J^{ \overline{L^{(2)}}}  _{C_{1,l-2n}}\\\vdots\\ 	\mathcal J^{ \overline{L^{(2)}}}  _{C_{1,l-1}}
						\end{bmatrix}.
						\end{multlined}
						\end{align*}
Recalling \eqref{kernelChristoffeldefinicion} we get
						\begin{align*}
						&\begin{multlined}[t][0.9\textwidth]
						\prodint{\overline{(\xi^{(1)})_z},	\check K^{(1),[l]}(\bar z,x_2)}	\overline{		\mathcal L^{(1)}}
						=	L^{(1)}(x_2)\mathcal J^{ \overline{L^{(1)}}} _{K_{C,\phi}^{[l]}}(x_2)+
						\mathcal J^{ \overline{L^{(1)}}} _{\delta L^{(1)}}(x_2)
						\\
						-{\Big[\check \phi^{(1)}_{1,l}(x_2),\dots,
							\check \phi^{(1)}_{1,l+2n-1}(x_2)\Big]}
						({\check H}^{(1)}_{[n,l]})^{-1}
						\bar\Omega^{(1)}_2[n,l]
						\begin{bmatrix}
						\overline{ \mathcal J^{ \overline{L^{(1)}}} _{C_{2,l-2n}}}\\ \vdots\\  \overline{
							\mathcal J^{ \overline{L^{(1)}}} _{	C_{2,l-1}}}
						\end{bmatrix},
						\end{multlined}	
						\end{align*}
						\begin{align*}
						&\begin{multlined}[t][0.9\textwidth]
						\prodint{(\xi^{(2)})_z,\check K^{(2),[l]}(\bar x_1,z)}\mathcal L^{(2)}
						=	\overline{ L^{(2)}}(\bar x_1) \mathcal  J^{\bar L^{(2}}_{K_{\phi,C}^{[l]}}(\bar x_1)
						+\mathcal  J^{ \overline{L^{(2)}}} _{\delta { \overline{L^{(2)}}} }(\bar x_1)
						\\
						-\Big[\overline{\check \phi^{(2)}_{2,l}(x_1)},\dots, \overline{\check \phi^{(2)}_{2,l+2n-1}(x_1)}\Big]
						(\check H^{(2)}[n,l])^{-1}
						\Omega^{(2)}_1[n,l]
						\begin{bmatrix}
						\mathcal J^{ \overline{L^{(2)}}}  _{C_{1,l-2n}}\\\vdots\\ 	\mathcal J^{ \overline{L^{(2)}}}  _{C_{1,l-1}}
						\end{bmatrix}.
						\end{multlined}
						\end{align*}
Now, from 	 \eqref{GerKerNor1} and \eqref{GerKerNor2}, we get		
			\begin{align*}%\label{GerKerNor1}{GerKerNor2}
			&\begin{multlined}[t][0.9\textwidth]
		\prodint{\overline{(\xi^{(1)})_z,}	{\check K^{(1),[l]}(\bar z,x_2)}}\overline{		\mathcal L^{(1)}}={L^{(1)}(x_2)}
		 	\prodint{\overline{(\xi^{(1)})_z},	{K^{[l]}(\bar z,x_2)}}\overline{		\mathcal L^{(1)}}\\-{\Big[\check \phi^{(1)}_{1,l}(x_2),\dots,
				\check \phi^{(1)}_{1,l+2n-1}(x_2)\Big]}
			({\check H}^{(1)}[n,l])^{-1}\bar\Omega_2^{(1)}[n,l]
			\begin{bmatrix}
			\overline{		\prodint{(\xi^{(1)})_z,	{ \phi_{2,l-2n}(z)}}	\mathcal L^{(1)}}\\ \vdots\\  	\overline{	\prodint{(\xi^{(1)})_z,	{
				\phi_{2,l-1}(z)}}	\mathcal L^{(1)}}
			\end{bmatrix},
			\end{multlined}	\end{align*}
			
			\begin{align*}%\label{GerKerNor2}
			&\begin{multlined}[t][0.9\textwidth]
		\prodint{(\xi^{(2)})_z,	\check K^{(2),[l]}(\bar x_1,z)}\mathcal L^{(2)}= {\overline{ L^{(2)}}(\bar x_1)}
		\prodint{(\xi^{(2)})_z,K^{[l]}(\bar x_1,z)}\mathcal L^{(2)}\\-\Big[\overline{\check\phi^{(2)}_{2,l}(x_1)},\dots, \overline{\check\phi^{(2)}_{2,l+2n-1}(x_1)}\Big](
			(\check H^{(2)}[n,l])^{-1}\Omega^{(2)}_1[n,l]
			\begin{bmatrix}\prodint{\xi^{(2)},
			\phi_{1,l-2n}}\mathcal L^{(2)}\\\vdots\\\prodint{\xi^{(2)},\phi_{1,l-1}}\mathcal L^{(2)}
			\end{bmatrix}.
			\end{multlined}
			\end{align*}
Therefore, we conclude
%	\begin{align*}
%	&\begin{multlined}[t][0.9\textwidth]
%	L^{(1)}(x_2)\Big(\mathcal J^{\bar L^{(1)}}_{K_{C,\phi}^{[l]}}(x_2)-\prodint{\overline{(\xi^{(1)})_z},	{K^{[l]}(\bar z,x_2)}}\overline{\mathcal L^{(1)}}
%	\Big)+
%	\mathcal J^{\bar L^{(1)}}_{\delta L_1}(x_2)
%	\\=
%	-{\Big[\check \phi^{(1)}_{1,l}(x_2),\dots,
%		\check \phi^{(1)}_{1,l+2n-1}(x_2)\Big]}
%	({\check H}^{(1)}_{[n,l]})^{-1}
%	\bar\Omega^{(1)}_2[n,l]
%	\begin{bmatrix}
%	\overline{ \mathcal J^{\bar L^{(1)}}_{C_{2,l-2n}}-	{\prodint{(\xi^{(1)})_z,	{ \phi_{2,l-2n}(z)}}	}	\mathcal L^{(1)}}\\ \vdots\\  \overline{
%		\mathcal J^{\bar L^{(1)}}_{	C_{2,l-1}}-
%		{\prodint{(\xi^{(1)})_z,	{ \phi_{2l-1}(z)}}	}	\mathcal L^{(1)}}
%	\end{bmatrix},
%	\end{multlined}	
%	\end{align*}
%	\begin{align*}
%	&\begin{multlined}[t][0.9\textwidth]
%	\overline{ L^{(2)}}(\bar x_1) \Big(\mathcal  J^{\bar L^{(2}}_{K_{\phi,C}^{[l]}}(\bar x_1)-	\prodint{(\xi^{(2)})_z,K^{[l]}(\bar x_1,z)}\mathcal L^{(2)}
%	\Big)+
%	\mathcal  J^{\bar L^{(2}}_{\delta \bar L_2}(\bar x_1)	\\=
%	-\Big[\overline{\check \phi^{(2)}_{2,l}(x_1)},\dots, \overline{\check \phi^{(2)}_{2,l+2m-1}(x_1)}\Big]
%	(\check H^{(2)}[n,l])^{-1}
%	\Omega^{(2)}_1[n,l]
%	\begin{bmatrix}
%	\mathcal J^{\bar L^{(2}} _{C_{1,l-2n}}-\prodint{\xi^{(2)},\phi_{1,l-2n}}\mathcal L^{(2)}\\\vdots\\ 	\mathcal J^{\bar L^{(2}} _{C_{1,l-1}}-\prodint{\xi^{(2)},\phi_{1,l-1}}\mathcal L^{(2)}
%	\end{bmatrix},
%	\end{multlined}
%	\end{align*}
%from where we infer that
\begin{align*}
&\begin{multlined}[t][0.9\textwidth]
\Big(L^{(1)}(x_2)\Big(\mathcal J^{\bar L}_{K_{C,\phi}^{[l]}}(x_2)-\prodint{\overline{(\xi^{(1)})_z},	{K^{[l]}(\bar z,x_2)}}\overline{\mathcal L^{(1)}}
\Big)+
\mathcal J^{\bar L^{(1)}}_{\delta L_1}(x_2)\Big)\begin{bmatrix}
\overline{ \mathcal J^{\bar L^{(1)}}_{C_{2,l-2n}}-	{\prodint{(\xi^{(1)})_z,	{ \phi_{2,l-2n}(z)}}	}	\mathcal L^{(1)}}\\ \vdots\\  \overline{
	\mathcal J^{\bar L^{(1)}}_{	C_{2,l-1}}
-	{\prodint{(\xi^{(1)})_z,	{ \phi_{2,l-1}(z)}}	}	\mathcal L^{(1)}}
\end{bmatrix}^{-1}
\\=
{\Big[\check \phi^{(1)}_{1,l}(x_2),\dots,
	\check \phi^{(1)}_{1,l+2n-1}(x_2)\Big]}
({\check H^{(1)}[n,l]})^{-1}
\bar\Omega^{(1)}_2[n,l],
\end{multlined}	
\end{align*}
\begin{align*}
&\begin{multlined}[t][0.9\textwidth]
\Big(\overline{ L^{(2)}}(\bar x_1) \Big(\mathcal  J^{\bar L^{(2}}_{K_{\phi,C}^{[l]}}(\bar x_1)-	\prodint{(\xi^{(2)})_z,K^{[l]}(\bar x_1,z)}\mathcal L^{(2)}
\Big)+
\mathcal  J^{\bar L^{(2}}_{\delta \bar L_2}(\bar x_1)\Big)\begin{bmatrix}
\mathcal J^{\bar L^{(2}} _{C_{1,l-2n}}-\prodint{\xi^{(2)},\phi_{1,l-2n}}\mathcal L^{(2)}\\\vdots\\ 	\mathcal J^{\bar L^{(2}} _{C_{1,l-1}}-\prodint{\xi^{(2)},\phi_{1,l-1}}\mathcal L^{(2)}
\end{bmatrix}^{-1}
\\=
\Big[\overline{\check \phi_{2,l}(x_1)},\dots, \overline{\check \phi_{2,l+2m-1}(x_1)}\Big]
(\check H^{(2)}[n,l])^{-1}
\Omega^{(2)}_1[n,l],
\end{multlined}
\end{align*}
and, in particular, recalling \eqref{Omegalambda}, we obtain
	\begin{align*}
	&\begin{multlined}[t][\textwidth]
	\check \phi^{(1)}_{1,l}(x_2)
	 L^{(1)}_{(-1)^ln}\\=\Big(L^{(1)}(x_2)\Big(\mathcal J^{ \overline{L^{(1)}}} _{K_{C,\phi}^{[l]}}(x_2)-\prodint{\overline{(\xi^{(1)})_z},	{K^{[l]}(\bar z,x_2)}}\overline{\mathcal L^{(1)}}
	\Big)+
	\mathcal J^{ \overline{L^{(1)}}} _{\delta L^{(1)}}(x_2)\Big)\begin{bmatrix}
	\overline{ \mathcal J^{ \overline{L^{(1)}}} _{C_{2,l-2n}}-	{\prodint{(\xi^{(1)})_z,	{ \phi_{2,l-2n}(z)}}	}	\mathcal L^{(1)}}\\ \vdots\\  \overline{
		\mathcal J^{ \overline{L^{(1)}}} _{	C_{2,l-1}}
		-{\prodint{(\xi^{(1)})_z,	{ \phi_{2,l-1}(z)}}	}	\mathcal L^{(1)}}
	\end{bmatrix}^{-1}\begin{bmatrix}
	H_{l-2n}\\0\\\vdots\\0
	\end{bmatrix},
	\end{multlined}	
	\end{align*}
	\begin{align*}
	&\begin{multlined}[t][0.9\textwidth]
	\overline{\check \phi^{(2)}_{2,l}(x_1)L^{(2)}_{(-1)^ln}}  \\=\Big(\overline{ L^{(2)}}(\bar x_1) \Big(\mathcal  J^{ \overline{L^{(2)}}} _{K_{\phi,C}^{[l]}}(\bar x_1)-	\prodint{(\xi^{(2)})_z,K^{[l]}(\bar x_1,z)}\mathcal L^{(2)}
	\Big)+
	\mathcal  J^{ \overline{L^{(2)}}} _{\delta { \overline{L^{(2)}}} }(\bar x_1)\Big)\begin{bmatrix}
	\mathcal J^{ \overline{L^{(2)}}}  _{C_{1,l-2n}}-\prodint{\xi^{(2)},\phi_{1,l-2n}}\mathcal L^{(2)}\\\vdots\\ 	\mathcal J^{ \overline{L^{(2)}}}  _{C_{1,l-1}}-\prodint{\xi^{(2)},\phi_{1,l-1}}\mathcal L^{(2)}
	\end{bmatrix}^{-1}\begin{bmatrix}
	H_{l-2n}\\0\\\vdots\\0
	\end{bmatrix}.
	\end{multlined}
	\end{align*}
\end{proof}


\begin{thebibliography}{99}
	
		\bibitem{adler}M. Adler and P. van Moerbeke, \emph{Group factorization, moment matrices and Toda lattices}, International Mathematics Research Notices \textbf{12} (1997) 556-572.
		
		\bibitem{adler-van moerbeke} M. Adler and P. van Moerbeke,  \emph{Generalized orthogonal polynomials, discrete KP and
			Riemann--Hilbert  problems},  Communications in  Mathematical  Physics \textbf{207} (1999) 589-620.
	
	\bibitem{alvarez2015Christoffel} C. Álvarez-Fernández, G. Ariznabarreta, J. C. García-Ardila, M. Mañas,  and F. Marcellán, \emph{Christoffel transformations for matrix orthogonal polynomials in the real line and the non-Abelian 2D Toda lattice hierarchy}, International Mathematics Research Notices, doi:10.1093/imrn/rnw027 (2016).
	
		\bibitem{alvarez2016Transformation} C. Álvarez-Fernández, G. Ariznabarreta, J. C. García-Ardila, M. Mañas, and F. Marcellán, \emph{Transformation theory and Christoffel-Geronimus formulas for matrix biorthogonal polynomials on the real line}, Arxiv:1605.04617 (2016).

	\bibitem{CM} C. Álvarez-Fernández,  and M. Mañas, \emph{Orthogonal Laurent polynomials on the unit circle, extended CMV ordering and 2D Toda Type integrable hierarchies}, Advances in Mathematics \textbf{240} (2013) 132-193.
			
		\bibitem{am} C. Álvarez-Fernández and M. Mañas, \emph{On the Christoffel-Darboux formula for generalized matrix orthogonal polynomials}, Journal of Mathematical Analysis and Applications \textbf{418} (2014) 238-247.
	
			\bibitem{amu} C. Álvarez-Fernández, U. Fidalgo Prieto, and M. Mañas, \emph{Multiple orthogonal polynomials of mixed type: Gauss-Borel factorization and the multi-component 2D Toda hierarchy}, Advances in Mathematics \textbf{227} (2011) 1451-1525.
			
			\bibitem{ari} G. Ariznabarreta and M. Mañas, \emph{Matrix orthogonal Laurent polynomials on the unit circle and Toda type integrable systems}, Advances in Mathematics \textbf{ 264}, (2014) 396-463.
	
		\bibitem{ari3} G. Ariznabarreta and M. Mañas, \emph{A Jacobi type Christoffel-Darboux formula for multiple orthogonal polynomials of mixed type}, Linear Algebra and its Applications \textbf{468}  (2015) 154-170.
	
	\bibitem{MVOPR} G. Ariznabarreta and M. Mañas, \emph{Multivariate orthogonal polynomials and integrable systems}, Advances in  Mathematics \textbf{ 302} (2016) 628-739.
			
					\bibitem{ari0} G. Ariznabarreta and M. Mañas, \emph{Darboux transformations for multivariate orthogonal polynomials,}  arXiv:1503.04786.
					
						\bibitem{ari1} G. Ariznabarreta and  M. Mañas, \emph{Multivariate orthogonal Laurent polynomials and integrable systems,} 	arXiv:1506.08708.
						
							\bibitem{Barroso-Daruis} R. Cruz-Barroso, L. Daruis,  P. Gonz\'{a}lez-Vera, and O. Nj\r{a}stad, \emph{Sequences of orthogonal Laurent polynomials, biorthogonality and
								quadrature formulas on the unit circle}, Journal of Computational and Applied Mathematics \textbf{200} (2007) 424-440.
							
							\bibitem{Barroso-Snake} R. Cruz-Barroso and S. Delvaux, \emph{Orthogonal Laurent polynomials on the unit circle and snake-shaped matrix factorizations}, Journal of Approximation Theory \textbf{161} (2009) 65-87.
							
							\bibitem{Barroso-Vera} R. Cruz-Barroso and P. Gonz\'{a}lez-Vera, \emph{A Christoffel--Darboux formula and a Favard's theorem for Laurent orthogonal polynomials on the unit circle}, Journal of Computational and Applied Mathematics \textbf{179} (2005) 157-173.
						
	\bibitem{Berriochoa} E. Berriochoa, A. Cachafeiro, and J. M. Garcia-Amor, \emph{Connection between orthogonal polynomials on the unit circle
		and bounded interval}, Journal of Computational and Applied Mathematics \textbf{177} (2005) 205-223.
	
				
%				\bibitem{Bue1} M. I. Bueno and F. Marcellán, \emph{Darboux transformation and perturbation of linear functionals}, Linear Algebra and its Applications  \textbf{384} (2004) 215-242.
				
				\bibitem{Bueno} M. I. Bueno and F. Marcellán, \emph{Polynomial perturbations of bilinear functionals and Hessenberg matrices}, Linear Algebra and its Applications \textbf{414 }(2006) 64-83.
			
			\bibitem{Bul} A. Bultheel, P. Gonz\'{a}lez-Vera, E. Hendriksen, and O. Nj\r{a}stad, \emph{Orthogonal rational functions}, Cambridge Monographs on
			Applied and Computational Mathematics, vol.5, Cambridge University Press, Cambridge (1999).
			
				\bibitem{CMV} M. J. Cantero, L. Moral, and L. Velázquez, \emph{Five-diagonal matrices and zeros of orthogonal polynomials on the unit circle}, Linear Algebra and its  Applications. \textbf{ 362} (2003) 29-56.
				
					\bibitem{Cantero}  M. J. Cantero, F. Marcellán, L. Moral, and L. Velázquez, \emph{Darboux transformations for CMV matrices}, Advances in Mathematics \textbf{298} (2016) 122-206.
					
				
%				\bibitem{cachafeiro} A. Cachafeiro and F. Marcellán, \emph{Modifications of Toeplitz matrices: jump functions}, Rocky Mountain Journal of Mathematics \textbf{ 23} (1993) 521-531.
				
			\bibitem{Chi} T. S. Chihara, \emph{An Introduction to Orthogonal Polynomials}. In :\emph{ Mathematics and its Applications Series}, Vol. \textbf{13}. Gordon and Breach Science Publishers, New York-London-Paris, 1978.


\bibitem{christoffel} E. B. Christoffel, \emph{Über die Gaussische Quadratur und eine Verallgemeinerung derselben},
Journal für die reine und angewandte Mathematik (Crelle's journal) \textbf{55 }(1858) 61-82 (in German).
		
			\bibitem{darboux2} G. Darboux,\emph{ Sur une proposition relative aux équations linéaires},  Comptes Rendus hebdomadaires Academie des Sciences Paris \textbf{94} (1882) 1456-1459 (in French).
		
			\bibitem{Derevyagin} M. Derevyagin	and F. Marcellán,	\emph{A note on the Geronimus transformation and Sobolev orthogonal polynomials},   Numerical Algorithms \textbf{67} (2014)  271-287.
			
			\bibitem{DereM}  M. Derevyagin, J. C. García-Ardila, and F. Marcellán, \emph{Multiple Geronimus transformations}, Linear Algebra and its Applications \textbf{454} (2014) 158-183.
		
				\bibitem{dini}	J. Dini and  P. Maroni,  \emph{La multiplication d'une forme linéaire par une fraction rationnelle. Application aux formes de Laguerre-Hahn}, Annales Polonici Mathematici \textbf{52} (1990)  175?185  (in French).
		
			\bibitem{dsm} A. Doliwa, P. M. Santini, and M. Mañas, \emph{Transformations of Quadrilateral Lattices}, Journal of Mathematical Physics \textbf{41} (2000) 944-990.
		
		\bibitem{eisenhart} L. P. Eisenhart, \emph{Transformations of Surfaces}, Princeton University Press, London, 1923. Reprinted by  Sagwan Press, 2015.

\bibitem{fejer} L. Fejér, \emph{Über trigonometrische Polynome}, Journal für die reine und angewandte Mathematik (Crelle's journal)  \textbf{146} (1915) 53-82 (in German).

\bibitem{Freud} G. Freud, \emph{Orthogonal Polynomials}, Akad\'{e}miai Kiad\'{o}, Budapest and Pergamon Press, Oxford, 1971, 1985.

	\bibitem{Gaut1} W. Gautschi, \emph{An algorithmic implementation of the generalized Christoffel theorem} in Numerical Integration, edited by G. Hämmerlin, International Series of Numerical Mathematics, \textbf{57}, Birkhäuser, Basel, 1982. 89-106.
	
	\bibitem{Gaut2} W. Gautschi, \emph{Orthogonal Polynomials Computation and Approximation}, Numerical Mathematics and Scientific Computation, Oxford University Press, Oxford 2004.
	
	
\bibitem{Garza1}	 L. Garza, J. Hernández, and F. Marcellán, \emph{Spectral transformations of measures supported on the unit
circle and the Szeg\H{o} transformation}, Numerical  Algorithms \textbf{49} (1-4) (2008) 169-185.
%In this paper we analyze spectral transformations of measures sup- ported on the unit circle with real moments. The connection with spectral transformations of measures supported on the interval [?1, 1] using the Szego ? transformation is presented. Some numerical examples are studied.


 \bibitem{Garza2} L. Garza, J. Hernández, and F. Marcellán, \emph{Orthogonal polynomials and measures on the unit circle. The
Geronimus transformations}, Journal of Computational and  Applied  Mathematics \textbf{233} (5) (2010) 1220-1231.
%In this paper we analyze a perturbation of a nontrivial positive measure supported on the unit circle. This perturbation is the inverse of the Christoffel transformation and is called the Geronimus transformation. We study the corresponding sequences of monic orthogonal polynomials as well as the connection between the associated Hessenberg matrices. Finally, we show an example of this kind of transformation.


\bibitem{Garza3} L. Garza and F. Marcellán, \emph{Linear spectral transformations and Laurent polynomials}, Mediterranean  Journal of
Mathematics \textbf{6} (3) (2009) 273-289. %In this manuscript we analyze some linear spectral transformations of a Hermitian linear functional using the multiplication by some class of Laurent polynomials. We focus our attention in the behavior of the Verblunsky parameters of the perturbed linear functional. Some illustrative examples are pointed out.
	
	\bibitem{gelfand-distribu1} I. M. Gel'fand and G. E. Shilov, \emph{Generalized Functions. Volume I: Properties and Operations}, Academic Press, New York, 1964.
	Reprinted in the  AMS Chelsea Publishing,  American Mathematical Society, Providence, RI, 2016.
	
	\bibitem{gelfand-distribu2} I. M. Gel'fand and G. E. Shilov, \emph{Generalized Functions. Volume II: Spaces of Fundamental Solutions and Generalized Functions}, Academic Press, New York, 1968. 	Reprinted in the  AMS Chelsea Publishing,  American Mathematical Society, Providence, RI, 2016.
	
	\bibitem{gelfand} I. M. Gel'fand, S. Gel'fand,  V. S. Retakh, and R. Wilson,  \emph{Quasideterminants}, Advances in   Mathematics \textbf{193} (2005) 56-141.
		
			\bibitem{Geronimus} J. Geronimus, \emph{On polynomials orthogonal with regard to a given sequence of numbers and a theorem by W. Hahn},
			%Comm. Inst. Sci. Math. Mec. Univ. Kharkoff [Zapiski Inst. Mat. Mech.] (4) \textbf{17} (1940) 3-18.
			Izvestiya Akademii Nauk SSSR \textbf{4} (1940) 215-228 (in Russian).
		
	\bibitem{Godoy1} E. Godoy and F. Marcellán,  \emph{An analog of the Christoffel formula for polynomial modification of a measure on the unit circle},
	Bollettino dell'Unione Matematica Italiana  A (7)  \textbf{5} (1991)   1-12.
	
	\bibitem{Godoy2} E. Godoy and F. Marcellán, \emph{Orthogonal polynomials and rational modifications of measures}, Canadian  Journal  of  Mathematics \textbf {45}  (1993) 930-943.
	
\bibitem{Grenader}	U. Grenander and G. Szeg\H{o}, \emph{Toeplitz Forms and Their Applications}, University  of California Press, Berkeley/Los Angeles/Cambridge (UK), 1958 (reprinted by Chelsea, New York, 1984).
	
\bibitem{gru1} F. A. Grünbaum, \emph{The Darboux process and a noncommutative bispectral problem: some explorations and challenges,} Geometric Aspects of Analysis and Mechanics Progress in Mathematics {\bf 292}  (2011),  161-177, Birkhäuser/Springer, New York, 2011.

\bibitem{gru2} F. A. Grünbaum and L. Haine, \emph{Bispectral Darboux transformations: an extension of the Krall polynomials,} International Mathematics Research Notices \textbf{8} (1997) 359-392.


\bibitem{Hann} W. 	Hahn, \emph{Über die Jacobischen Polynome und zwei verwandte Polynomklassen}, Mathematische Zeitschrift \textbf{39} (1935) 634-38 (in German).

	\bibitem{Hormander} L. Hörmander, \emph{The analysis of Partial Differential Operators I, Distribution Theory and Fourier Analysis}, second edition, Springer, New York 1990.


	\bibitem{ismail} M. E. H. Ismail and  R. W. Ruedemann, \emph{Relation between polynomials orthogonal on the unit circle with respect to different weights},  Journal of  Approximation Theory \textbf {71}, (1992) 39-60.
	
%		\bibitem{Jones-5} W. B. Jones,  O. Nj\r{a}stad, and W.  J. Thron, \emph{Two-point Pade expansions for a family of analytic functions}, Journal of Computational and Applied Mathematics  \textbf{9} (1983) 105.123.
		
		\bibitem{Jones-3} W.  B. Jones and W. J. Thron, \emph{Orthogonal Laurent polynomials and Gaussian quadrature}, in: K.E. Gustafson, W. P. Reinhardt (Eds.), Quantum
		Mechanics in Mathematics Chemistry and Physics, Plenum, NewYork, 1981, pp. 449-455.
		
		\bibitem{Jones-4} W. B. Jones, W. J. Thron, and H. Waadeland, \emph{A strong Stieltjes moment problem}, Transactions of the American Mathematical Society \textbf{261} (1980) 503-528.
		
		\bibitem{Maroni1985espaces} P. Maroni, \emph{Sur quelques espaces de distributions qui sont des formes
			linéaires sur l'espace vectoriel des polynômes}, in Orthogonal polynomials
		and their applications, vol. 1171 of Lecture Notes in Mathematics, eds. C. Brezinski et al,
		Springer-Verlag, 1985, pp. 184--194  (in French).
		
		\bibitem{Maroni1988calcul} P. Maroni, \emph{Le calcul des formes
			linéaires et les polynômes orthogonaux semiclassiques}, in Orthogonal
		polynomials and their applications, M. Alfaro et al, ed., vol.~1329 of Lecture
		Notes in Mathematics, Springer-Verlag, 1988, pp. 279--290 (in French).
		
		\bibitem{Maro} P. Maroni, \emph{ Sur la suite de polynômes orthogonaux associée à la forme $u = \delta_c +\lambda(x-c)^{-1}L$}, Periodica Mathematica Hungarica \textbf{21} (1990) 223-248 (in French).
		
			\bibitem{matveev} V. B. Matveev and  M. A. Salle, \emph{Differential-difference evolution equations. II: Darboux Transformation for the Toda lattice}, Letters in Mathematical Physics \textbf{3} (1979) 425-429.
	
		\bibitem{matveev-book} V.  B. Matveev and  M. A. Salle, \emph{Darboux Transformations and Solitons}, Springer Series in Nonlinear Dynamics, Springer-Verlag, Berlin, 1991.	
	
%		\bibitem{mir2} L. Miranian, \emph{Matrix valued orthogonal polynomials in the unit circle: some extensions of the classical theory}, Canadian Mathematical Bulletin \textbf{52} (2009) 95-104.
	
		\bibitem{moutard} Th. F. Moutard,
		\emph{Sur la construction des équations de la forme $(1/z)\partial^2 z/\partial x\partial y=\lambda(x,y)$
			qui admettenent une intégrale générale explicite},  Journal de l'École polytechnique \textbf{45} (1878) 1-11 (in French).
		
			\bibitem{Nikiforov1991Discrete} A. F. Nikiforov,  S. K. Suslov,  and V. B. Uvarov,\emph{ Classical Orthogonal Polynomials of a Discrete Variable}, Springer, Berlin,
			1991.

	\bibitem{Olver} P. Olver, On multivariate interpolation, Studies in  Applied Mathematics, \textbf{ 116} (2) (2004).
	
	\bibitem{riesz} F. Riesz,  \emph{Über ein Problem des Herrn Caratheodory}, Journal fur die reine und angewandte Mathematik (Crelle's journal)   \textbf{146} (1915) 83-87 (in German).
	
		\bibitem{rogers-schief} C. Rogers  and W. K. Schief, \emph{Bäcklund and Darboux Transformations: Geometry and Modern Applications in Soliton Theory}, Cambridge University Press, Cambridge, 2002.
	
		\bibitem{R. W. Ruedemann}  R. W. Ruedemann, \emph{Some results on biorthogonal polynomials},  International Journal of  Mathematics \& Mathematical Sciences \textbf{17} (4), (1994) 625-636.
		
			\bibitem{Schwartz1}  L. Schwartz,  \emph{Théorie des noyaux}, Proceedings of the International Congress of Mathematicians
			(Cambridge, MA, 1950), vol. 1 p. 220-230,  American Mathematical Society, Providence, RI, 1952 (in French).
			
			\bibitem{Schwartz}  L. Schwartz, \emph{Théorie des distributions}, Hermann, Paris, 1978 (in French).
		
		\bibitem{Simon1} B. Simon, Orthogonal Polynomials on the unit circle, American  Mathematical Society Colloquium Publications \textbf {54} (2005).
		
		\bibitem{CMV-Simon} B. Simon, \emph{CMV matrices: Five years after}, Journal of Computational and Applied Mathematics \textbf{208} (2007) 120-154.
		
		\bibitem{Simon-S} B. Simon \emph{Szeg\H{o}'s Theorem and its descendants}, Princeton University Press, Princeton, New Jersey (2011).
		
			\bibitem{Sze}  G. Szeg\H{o}, Orthogonal Polynomials, American  Mathematical Society Colloquium Publications \textbf{ 23} (1939).
			
			\bibitem{Thron} W. J. Thron, \emph{L-polynomials orthogonal on the unit circle}, in: A. Cuyt (Ed.), Nonlinear Methods and Rational Approximation, Reidel Publishing Company, Dordrecht, 1988, pp. 271-278.
			
				\bibitem{Yoon} G. Yoon,  \emph{Darboux transforms and orthogonal polynomials}, Bulletin Korean Mathematical Society \textbf{39} (2002)  359-376.
				
%				\bibitem{Yakhlef1} H. O. Yakhlef, F. Marcellán, and M. A. Piñar, \emph{Relative Asymptotics for Orthogonal Matrix Polynomials with Convergent Recurrence Coefficients},  Journal of Approximation Theory \textbf{111} (2001) 1-30.
%				
%				\bibitem{Yakhlef2} H. O. Yakhlef, F. Marcellán, and M. A. Piñar, \emph{Perturbations in the Nevai matrix class of orthogonal matrix polynomials},  Linear Algebra and its Applications \textbf{336} (2001) 231-254.
%				
%				\bibitem{Yakhlef3} H. O. Yakhlef and  F. Marcellán, \emph{Relative Asymptotics for  Matrix Orthogonal Polynomials for Uvarov Perturbations: The Degenerate Cases},  Mediterranean Journal of Mathematics (2016) published online.
%			
%				\bibitem{Uva0}	V. B. Uvarov, \emph{Relation between polynomials orthogonal with different weights}, Doklady Akademii Nauk SSSR \textbf{126} (1),  (1959)  33?36 (in Russian).
				
				\bibitem{Uva} V. B. Uvarov, \emph{The connection between systems of polynomials that are orthogonal with respect to different distribution functions,} USSR Computational and Mathematical Physics \textbf{9} (1969) 25-36.
				
				\bibitem{watkins} D. S. Watkins, \emph{Some perspectives on the eigenvalue problem}, SIAM Review \textbf{35} (1993), 430-471.
	
	  \bibitem{zhang} F. Zhang (editor), \emph{The Schur Complement and its Applications}, Springer, New York, 2005.
				
					\bibitem{Zhe} A. Zhedanov, \emph{Rational spectral transformations and orthogonal polynomials}, Journal of  Computational and  Applied Mathematics \textbf{ 85} (1997) 67-86.
					
\end{thebibliography}
\end{document}